\documentclass[12pt]{amsart}
\usepackage{amsmath, amsthm, amssymb, latexsym, url}
\usepackage{amsfonts}
\usepackage{graphicx,subfig}
\usepackage{stmaryrd} 
\usepackage{hyperref} 

\allowdisplaybreaks

\author{J. Vandehey}
\thanks{Email: \href{mailto:vandehe2@illinois.edu}{\nolinkurl{vandehe2@illinois.edu}}}
\title{Error term improvements for van der Corput transforms.}
\date{\today}

\newtheorem{thm}{Theorem}[section]
\newtheorem{cor}[thm]{Corollary}
\newtheorem{lem}[thm]{Lemma}
\newtheorem{prop}[thm]{Proposition}
\newtheorem{rem}[thm]{Remark}

\begin{document}

\maketitle

\begin{abstract}
We improve the error term in the van der Corput transform for exponential sums
\[
\sideset{}{^*}\sum_{a \le n \le b} g(n) exp(2\pi i f(n)).
\]
For many functions $g$ and $f$, we can extract the next term in the asymptotic, showing that previous results, such as those of Karatsuba and Korolev, are sharp.  Of particular note, the methods of this paper avoid the use of the truncated Poisson formula, and thus can be applied to much longer intervals $[a,b]$ with far better results.  We provide a detailed analysis of the error term in the case $g(x)=1$ and $f(x)=(x/3)^{3/2}$.
\end{abstract}

\section{Introduction}

We are interested here in estimating the error term $\Delta$ associated with the van der Corput transform,
\begin{equation}\label{eq:vdCtransform}
\sideset{}{^*}\sum_{a\le n \le b} g(n) e(f(n)) = \sideset{}{^*}\sum_{f'(a)\le r \le f'(b)}  \frac{g(x_r)e(f(x_r)-rx_r+\frac{1}{8})}{\sqrt{f''(x_r)}}+\Delta,
\end{equation}
where $f$ and $g$ are several times continuously differentiable functions with $f''(x)>0$ for $x\in [a,b]$ and where $x_r$ is defined by $f'(x_r)=r$.  A starred sum indicates that if a limit of summation is an integer, the corresponding summand is multiplied by $1/2$.  The function $e(x)$ denotes $e^{2 \pi i x}$.

The van der Corput transform is best-known in number theory for being the crucial element of Process B in the theory of exponent pairs and is sometimes simply referred to as ``Process B.''  As a method of estimating exponential sums, the van der Corput transform is often presented alongside other methods of Weyl, van der Corput, and Vinogradov.  Direct application of the van der Corput transform can take a complicated sum to one more amenable to estimation techniques or it can reduce the number of terms and therefore make computational estimations easier.  The van der Corput transform itself is involutive---applying it to the right-hand side of \eqref{eq:vdCtransform} will simply return the sum on the left-hand side of \eqref{eq:vdCtransform}---so one gains nothing by applying it twice in a row; but one could alternate applications of the van der Corput transform with other techniques (such as the Process A of the theory of exponent pairs) to achieve better results.  This alternation method is still a fruitful ground for modern research.  Recently, Cellarosi \cite{cellarosi} attained interesting new results in the classical case of $g(x)=1$ and $f(x)=\alpha x^2$, where the alternating technique employed is simply reducing $\alpha$ modulo $1$; Nakai \cite{nakai1,nakai2,nakai3} has investigated the possibility of using an analogous method when $f(x)$ is cubic or quartic.  The van der Corput transform has also seen recent use in physical applications, including quantum optics and wave processes (see \cite{karatsubakorolev} and the papers cited there for more details).

Van der Corput \cite{vdC2} originally showed that, given 
\[
 |f''(x)| \asymp \lambda_2, \qquad |f^{(3)}(x)| \ll \lambda_3, \qquad g(x)=1, \qquad \text{for } x\in [a,b],
\]
the error term can be bounded like 
\[
 \Delta = O(\lambda_2^{-1/2}) + O(\log(f'(b)-f'(a)+2))+O((b-a)\lambda_2^{1/5}\lambda_3^{1/5}).
\]
(Here we use the Landau and Vinodradov asymptotic notations which will be defined explicitly in Section \ref{section:notation}.)  Phillips \cite{phillips} improved this error under the additional assumptions
\[
|f^{(4)}(x)| \ll \lambda_4 \qquad \text{for } x\in [a,b] \qquad  \text{and} \qquad \lambda_3^2=\lambda_2 \lambda_4;
\]
in this case, we can replace $O((b-a)\lambda_2^{1/5}\lambda_3^{1/5})$ with $O((b-a)\lambda_3^{1/3})$.

The form of the error term found in most modern texts on analytic number theory \cite{grakol,iwakow,karatsuba,montgomery}\footnote{Curiously, \cite{sandor} skips this form of the error term entirely.} has its roots in the work of Kolesnik \cite{kolesnik} and Heath-Brown \cite{heathbrown}, although the results of the latter authors required analyticity of the function $f$, an assumption which has since been circumvented.  This moderately-difficult-to-prove form of the error term suffices for many basic application of the van der Corput transform.  We present this modern bound on the error in the notation used by Huxley.

\begin{thm}\label{thm:textbook}(Lemma 5.5.3 in \cite{huxley1})

Suppose that $f(x)$ is real and four times continuously differentiable on $[a,b]$.  Suppose that there are positive parameters $M$ and $T$, with $M\ge b-a$, such that, for $x\in [a,b]$, we have
\[
f''(x) \asymp T/M^2,\qquad f^{(3)}(x) \ll T/M^3, \qquad\text{and}\qquad f^{(4)}(x) \ll T/M^4.
\]
Let $g(x)$ be a real function of bounded variation $V$ on the closed interval $[a,b]$.  Then
\begin{align*}
\sum_{a\le n \le b} g(n)e(f(n)) &=\sum_{f'(a)\le r \le f'(b)}  \frac{g(x_r)e(f(x_r)-rx_r+\frac{1}{8})}{\sqrt{f''(x_r)}}\\
&\qquad + O\left( (V+|g(a)|)\left(\frac{M}{\sqrt{T}} +\log(f'(b)-f'(a)+2)\right)\right),
\end{align*}
where $x_r$ is the unique solution in $[a,b]$ to $f'(x_r)=r$.  The implicit constant in the big-O term depends on the implicit constants in the relations between $T$, $M$, and the derivatives of $f(x)$.
\end{thm}

The error term $M/\sqrt{T}$ here corresponds to the $\lambda_2^{-1/2}$ term in the estimates of van der Corput and Phillips.

Unfortunately, for many interesting cases, the above error is insufficient.  As Huxley \cite[p.~475]{huxley1} notes, when applying the van der Corput transform to a multi-dimensional exponential sum, ``...the error terms and the truncation error in the Poisson summation formula may add up to more than the estimate for the reflected sum.''  Thus, finer error terms, useful for a broad spectrum of problems including computation and physical applications, have been given by various people, including Kolesnik \cite{kolesnik2}, Liu \cite{liu}, Redouaby and Sargos \cite{redsar2}, and Karatsuba and Korolev \cite{karatsubakorolev}.  Liu extends an (unfortunately obscure) earlier work of Min \cite{min}, removing the latter's condition that $f(x)$ be an algebraic function.  Redouaby and Sargos show that the conditions on $f^{(4)}(x)$ and $g''(x)$ can be removed without greatly increasing the bound on the error term.  The work of Karatsuba and Korolev is unique among all papers cited here, in that it is the only to give the implicit constants in the big-O terms explicitly, making it superior for most computational work.  

Outside of Redouaby and Sargos' result, we cannot briefly state any of these other forms of the error term in full detail; we will, however provide the following inexplicit form of Karatsuba and Korolev's result as an example of the comparative strength of these errors compared with Theorem \ref{thm:textbook}.

\begin{thm}\label{thm:karatsubakorolev}

Suppose that $f(x)$ and $g(x)$ are real-valued functions with $f$ four times continuously differentiable and $g$ two times continuously differentiable on the interval $[a,b]$.  Suppose there are positive constants $M$, $T$, and $U$, with $M\asymp b-a$, such that, for $x\in[a,b]$,
\begin{align}
f''(x) &\asymp T/M^2, & |f^{(3)}(x)| &\ll T/M^3, &  |f^{(4)}(x)| &\ll T/M^4 \label{eq:fbound}\\
|g(x)| &\ll U, & |g'(x) | &\ll U/M, & |g''(x)| &\ll U/M^2. \label{eq:gbound}
\end{align}
Then,
\begin{align*}
\sideset{}{^*}\sum_{a\le n \le b} g(n)e(f(n)) &= \sideset{}{^*}\sum_{f'(a)\le r \le f'(b)} \frac{g(x_r)e(f(x_r)-rx_r+1/8)}{\sqrt{f''(x_r)}}+O(U(T(a)+T(b)))\\
&\qquad + O\left(U\left(  \log (f'(b)-f'(a)+2) +\frac{M}{T} +\frac{T}{M^2}+1 \right)\right),
\end{align*}
where $x_r$ is the unique solution to $f'(x_r)=r$ in the interval $[a,b]$ and
\begin{equation}\label{eq:Tdef}
 T(\mu) = \begin{cases}
           0, & \|f'(\mu)\|=0\\
           \min \left\{ \dfrac{M}{\sqrt{T}}, \dfrac{1}{\|f'(\mu)\|} \right\}, & \|f'(\mu)\| \neq 0
          \end{cases}.
\end{equation}
The size of the implicit constant in the big-O term depends on the implicit constants in the relations of $M$, $T$, $U$, and the derivatives of $f(x)$ and $g(x)$.

If, in addition, $M \ll T \ll M^2$ and $M\ge b-a $, we may remove the terms $O(T/M^2+1)$.\footnote{This supplementary result comes from Karatsuba and Voronin \cite{karatsubavoronin}.}
\end{thm}

The $M/\sqrt{T}$ term again makes an appearance in this theorem.  It cannot, in fact, be completely removed.  If $b$ is not an integer, but $f'(b)$ is an integer, then changing $b$ to $b-\epsilon$ for some very small $\epsilon > 0$ will not change the value of the sum on the left-hand side of \eqref{eq:vdCtransform}, but removes a term of size $M/\sqrt{T}$ from the right-hand side.

Van der Corput's results in \cite{vdC} also deserve a mention here, as they are of an entirely different flavor from those listed above and because they are generally not very well-known.\footnote{This may be largely van der Corput's fault.  Robert Schmidt, in reviewing van der Corput's paper for Zentralblatt, remarked that because van der Corput provided no context for his results, neither in terms of past results or future goals, that the paper was unlikely to spark much interest, and indeed it has so far only been cited once elsewhere, in \cite{coutkaz1}.  Due to the desire to avoid repeating van der Corput's mistake, the complicated main theorem of this paper will not be presented in the introduction.}  Instead of having a coefficient of $g(x_r)/\sqrt{f''(x_r)}$ in the transformed sum, he has a more general set of coefficients.  In addition, his error term---which would require too many definitions to state succinctly here---bears no resemblance to any of the other error estimates cited or formulated in this paper.  

While van der Corput's results in \cite{vdC} are quite complicated to use, they are aesthetically pleasing.  As we remarked above, the van der Corput transform is involutive, but in all the other results given above as well as the main result of this paper, one can obtain very different error terms when one applies the transform to the right-hand side of \eqref{eq:vdCtransform} instead of the left, assuming that the conditions necessary to apply the results would even still hold.  In \cite{vdC}, van der Corput shows that the conditions are still satisfied and the error term unchanged regardless of which side one applies his transform to; his \emph{theorem} is involutive.

The van der Corput transform and its error has been studied in much more general settings than we go into here: of particular interest, Jutila \cite{jutila} considered sums of the form $\sum b(n) g(n) e(f(n))$ for certain multiplicative functions $b(n)$ (see also \cite[Ch.~20]{huxley1}), and Kr{\"a}tzel \cite{kratzel2} considered the van der Corput transform of a convergent infinite series.

One may ask what the best possible error term could be.  Given the frequent restriction in theorems on the van der Corput transform that $f''(x)\asymp \lambda_2$ (or, equivalently $f''(x) \asymp T/M^2$), it is not surprising that the case where $g$ is constant and $f$ is quadratic (that is, $f''(x)=\lambda_2$) is the most-commonly studied special case of the van der Corput transform and the one with the best error terms \cite{cellarosi,coutkaz1,fedklopp,mordell,wilton}.  Fedotov and Klopp \cite{fedklopp}\footnote{This paper contains a small error in line $(0.4)$ that helped to spark the author's investigation into the van der Corput transform.} have given the error term  in this case as an explicit integral.  But perhaps most amazing are the results of Coutsias and Kazarinoff \cite{coutkaz1}: they showed that for positive integers $n$, we have
\[
\left|\sideset{}{^*}\sum_{k=0}^N e\left(\omega\cdot \frac{k^2}{2} \right) - \frac{e(sgn(\omega)/8)}{\sqrt{|\omega|}}\sideset{}{^*}\sum_{k=0}^n e\left(-\frac{1}{\omega}\cdot \frac{k^2}{2} \right)  \right| \le C\left|N- \frac{n}{\omega} \right|
\]
for $0<|\omega| <1$, $N=\left\llbracket n/\omega\right\rrbracket $ is the nearest integer to $n/\omega$, and $1<C<3.14$ is a particular constant.  Not only is the error bounded but it shrinks to zero as $n/\omega$ nears an integer.

The Coutsias-Kazarinoff result suggests that the van der Corput transform should be very accurate; in particular, the van der Corput transform for nice enough functions $f$ and $g$ shouldn't have compounding error terms (such as the $\log(f'(b)-f'(a)+2)$) seen in all the other results mentioned above.

The main theorems of this paper (and Theorem \ref{thm:toinfinity}, especially) confirm this hypothesis, allowing the van der Corput transform to be applied on very long intervals with a much higher degree of accuracy than in previous results.

As a quick example of this, consider the following well-known transform which appears in Iwaniec and Kowalski's book \cite[p.~211]{iwakow} (a version of this transform also appears in \cite{liu}).
Given $X>0$, $N>0$, and $\alpha>1$, $\nu>1$, consider 
\begin{equation}\label{eq:iwakowexample}
 \sideset{}{^*}\sum_{N\le n \le \nu N} \left(\frac{\alpha}{n}\right)^{\frac{1}{2}}e\left( \frac{X}{\alpha} \left(\frac{n}{N}\right)^{\alpha}\right) = \sideset{}{^*}\sum_{M\le m \le \mu M} \left(\frac{\beta}{m}\right)^{\frac{1}{2}} e\left(\frac{1}{8}-\frac{X}{\beta} \left(\frac{m}{M}\right)^{\beta}\right) +\Delta,
\end{equation}
where $1/\alpha+1/\beta=1$, $\mu^{\beta}=\nu^{\alpha}$, and $MN=X$.  Using a form of Theorem \ref{thm:textbook}, Iwaniec and Kowalski\footnote{A typo in the book has the logarithms in opposite places.} give
\[
 \Delta = O(N^{-1/2}\log (M+2)+M^{-1/2}\log (N+2)),
\]
with an implicit constant dependent on $\alpha$ and $\nu$.

Using the results of this paper, we may improve this to the following.

\begin{cor}\label{cor:iwakowimprovement}
Provided $N\gg 1$ and $M\gg 1$, in equation \eqref{eq:iwakowexample}, we may take
\begin{equation}\label{eq:iwakowimprovement}
 \Delta=O(N^{-1/2}+M^{-1/2})
\end{equation}
with an implicit constant dependent \emph{only} on $\alpha$ and the implicit constants in the lower bounds on $N$ and $M$.
\end{cor}

Corollary \ref{cor:iwakowimprovement} allows us to take $\nu$ as large as we like (and hence our intervals as large as we like) without increasing the bound on $\Delta$.  This, as well as the loss of the logarithmic factor, are common to applications of the results in this paper.

More powerful results are possible.  In many cases, the main theorems of this paper also allow one to extract the next term in the asymptotic for the van der Corput transform.  Under the same general hypotheses of Theorems \ref{thm:textbook} and \ref{thm:karatsubakorolev}, we can improve the error terms to the following.

\begin{cor}\label{thm:karatsubakorolevimprovement}
Suppose that $f(x)$ and $g(x)$ are real-valued functions with $f$ four times continuously differentiable and $g$ two times continuously differentiable on the interval $[a,b]$.  Suppose there exist constants $M$, $T$, and $U$ satisfying $M=b-a \gg 1$, $T\gg 1$, and the bounds on lines \eqref{eq:fbound} and \eqref{eq:gbound}.

Then,
\begin{align*}
\sideset{}{^*}\sum_{a\le n \le b} g(n)e(f(n)) &= \sideset{}{^*}\sum_{f'(a)\le r \le f'(b)} \frac{g(x_r)e(f(x_r)-rx_r+1/8)}{\sqrt{f''(x_r)}}\\
&\qquad - \mathcal{T}(b)+\mathcal{T}(a)+O\left( \frac{U}{\sqrt{T}}\left(1+\frac{M}{T}\right) \right),
\end{align*}
where $\mathcal{T}(\mu)$ equals
\[
\begin{cases}
 \dfrac{g(\mu)f^{(3)}(\mu)e(f(\mu))}{6\pi i f''(\mu)^2} - \dfrac{g'(\mu)e(f(\mu))}{2\pi i f''(\mu)} & \|f'(\mu)\|=0 \\
O\left(\dfrac{UM}{\sqrt{T}}\right) & 0 < \|f'(\mu)\| \le \sqrt{f''(\mu)} \\
\begin{aligned}&g(\mu)e(f(\mu)+\llbracket f'(\mu) \rrbracket \mu)\left(-\dfrac{1}{2\pi i \langle f'(\mu)\rangle}+\psi(\mu,\langle f'(\mu)\rangle)\right)\\ &\quad +O\left( \dfrac{U}{M\|f'(x) \|^2}+\dfrac{UT}{M^2\|f'(x)\|^3} \right)\end{aligned}  & \|f'(\mu) \| \ge \sqrt{f''(\mu)} 
\end{cases}.
\]

Here, $\llbracket x\rrbracket$ represents the nearest integer\footnote{If $x=n+1/2$ for an integer $n$, then it doesn't matter if we let $\llbracket x\rrbracket$ equal $n$ or $n+1$ provided that we do so consistently.} to $x$; $\langle x \rangle$, the difference between $x$ and the nearest integer to $x$, namely $x-\llbracket x \rrbracket$; and $\| x\|$, the distance between $x$ and the nearest integer to $x$, so that $\|x \|= |\langle x \rangle |$.  The function $\psi(x, \epsilon)$ equals 
\[
-\frac{1}{2\pi i} \lim_{R\to \infty} \sum_{0<|r|< R} \frac{e(rx)}{r+\epsilon} \qquad \text{ for }|\epsilon|\le \frac{1}{2},
\]
which converges and is uniformly bounded for all real $x$.

The size of the implicit constant in the big-O term depends on the implicit constants in the relations of $M$, $T$, $U$, and the derivatives of $f(x)$ and $g(x)$.
\end{cor}

This in particular suggests that the size of $T(\mu)$ in Karatsuba and Korolev's result is optimal when $\|f'(\mu)\| \ge \sqrt{f''(\mu)}$.  When $\|f'(\mu)\|=0$---that is, when $f'(\mu)$ is an integer---the term $\mathcal{T}(\mu)$ can bemore simply bounded by $O(UM/T)$.

Corollary \ref{thm:karatsubakorolevimprovement} can, in certain cases, be used to improve the Kusmin-Landau inequality, a common result in the study of exponential sums.

\begin{thm}(The Kusmin-Landau inequality---Theorem 2.1 in \cite{grakol})\footnote{A short history of this theorem is given on page 20 of \cite{grakol}.}

 Suppose $f(x)$ is continuously differentiable and that $f'(x)$ is monotonic on an interval $[a,b]$.  Moreover suppose $\| f'(x) \|\ge \theta>0$ on $[a,b]$, where $\|f'(x)\|$ is the distance from $f'(x)$ to the nearest integer.  Then
\[
 \sum_{a\le n \le b} e(f(n)) \le \cot\left( \frac{\pi\theta}{2}\right).
\]
\end{thm}

\begin{cor}\label{thm:expk-l}

 Suppose $f$, $T$, and $M$ satisfy the conditions of Corollary \ref{thm:karatsubakorolevimprovement}. 

If $\theta = \min_{z\in[a,b]}\|f'(z)\|$ is positive, $\|f'(a)\|>\sqrt{f''(a)}$, and $\|f'(b) \|>\sqrt{f''(b)}$, then
\begin{align*}
 \sideset{}{^*}\sum_{a \le n \le b} e(f(n)) &= \frac{e(f(b))}{2\pi i \langle f'(b) \rangle} - \frac{e(f(a))}{2\pi i \langle f'(a) \rangle}\\
&\qquad + O\left( \frac{1}{M \theta^2}+ \frac{T}{M^2\theta^3}+\frac{1}{\sqrt{T}}\left(1+\frac{M}{T}\right) \right),
\end{align*}
where $\langle x\rangle$ represents the difference between $x$ and the nearest integer to $x$.
\end{cor}

Corollary \ref{thm:expk-l} follows immediately as a special case of Corollary \ref{thm:karatsubakorolevimprovement}.  Moreover, it suggests that the constant in the classical Kusmin-Landau inequality may not be optimal.  The function $\cot(\pi \theta/2)$ grows like $2/\pi\theta$ as $\theta$ approaches zero; however, Corollary \ref{thm:expk-l} suggests that the growth should be at worst like $1/\pi \theta$ as $\theta$ approaches zero.

While Corollaries \ref{thm:karatsubakorolevimprovement} and \ref{thm:expk-l} strengthen many previous results, they are constrained to apply to short intervals.  The results of this paper do not give similarly simple conditions and error terms when the size of the interval is allowed to grow arbitrarily large.  But, given a specific sum, we can show very great improvements as one endpoint of the interval tends towards infinity.  Consider the specific transform
\begin{equation}\label{eq:exampleerror}
 \sideset{}{^*}\sum_{1 \le n \le N} e\left( \left(\frac{n}{3}\right)^{3/2}\right) = \sideset{}{^*}\sum_{(1/12)^{1/2} \le r \le (N/12)^{1/2}} (24 r)^{1/2} \cdot e(-4r^3+1/8) + \Delta,
\end{equation}
where $N$ is an integer.  (We have kept the $(1/12)^{1/2}$ due to it's natural appearance in the application of the van der Corput transform.  It may be replaced by $1/2$ with no change in value on the right-hand side, however.)

If we apply the form of the error term from Theorem \ref{thm:textbook}, then we obtain the following result, which was included as an example in \cite{montgomery}.
\begin{cor}\label{cor:monterror}
 In line \eqref{eq:exampleerror}, we have
\[
 \Delta = O(N^{1/4}).
\]
\end{cor}

If we apply Theorem \ref{thm:karatsubakorolev} from Karatsuba and Korolev instead, we obtain much finer results.
\begin{cor}\label{cor:karakorerror}
 In line \eqref{eq:exampleerror}, we have
\[
 \Delta =\begin{cases}
           O((\log N)^2) + O\left( \min\left\{ N^{1/4}, \dfrac{1}{\| (N/12)^{1/2}\|} \right\}\right) & \| (N/12)^{1/2} \| \neq 0\\
           O((\log N)^2) & \| (N/12)^{1/2} \| = 0
         \end{cases}
\]
where $\| x\|$ is the distance from $x$ to the nearest integer.
\end{cor}

In this case, the $(\log N)^2$ term comes from needing to break the sum on the left-hand side of \eqref{eq:exampleerror} into roughly diadic intervals, each of which contributes an error term of size around $\log N$.  One must also choose the endpoints of these intervals to be values which are $12$ times a square, in order to keep the error term $T(\mu)$ zero.

Using the results of this paper, we can show that the result of Karatsuba and Korolev is almost sharp, in that we will extract an explicit term of size $N^{1/4}$ when $N^{1/4}$ is smaller than $\|(N/12)^{1/2}\|^{-1}$ and extract an explicit term of size $\|(N/12)^{1/2}\|^{-1}$ when $\|(N/12)^{1/2}\|^{-1}$ is smaller than $N^{1/4}$.

\begin{cor}\label{thm:exampleerror}
In line \eqref{eq:exampleerror}, we have
\[
\Delta = \begin{cases}
c+ O\left(N^{-1/2}\right) & \|(N/12)^{1/2}\|=0\\
\begin{aligned}
& 2\psi(\sqrt{N/12})(3N)^{1/4}    e((N/3)^{3/2}+1/8)\\
&\qquad+O(N^{3/20})\\
&\qquad +O\left(N^{5/12}\| (N/12)^{1/2}\|^{2/3}\right)
\end{aligned}  & \begin{aligned}&0< \| (N/12)^{1/2} \|\\ &\quad \le  (12N)^{-1/4}\end{aligned}\\
\begin{aligned}&e\left(\left(N/3\right)^{3/2}\right)\times \\ &\quad \times\left(\dfrac{1}{2\pi i\langle (N/12)^{1/2}\rangle} - \psi(N,\langle (N/12)^{1/2}\rangle ) \right) \\
&\qquad +c+ O\left(\frac{1}{N^{1/2} \|(N/12) ^{1/2}\|^3}\right) 
\end{aligned}  & \| (N/12)^{1/2} \| > (12N)^{-1/4}
\end{cases},
\]
where $c$ is a particular constant\footnote{Rough Mathematica calculations with $N=120,000$ give $c \approx 0.168-0.320i$.} and $\psi(x)=\psi(x,0)$ is the sawtooth function.  The remaining functions are all as in Corollary \ref{thm:karatsubakorolevimprovement}.
\end{cor}

One way to understand these new error terms is via the geometry of the associated sums.  Given functions $f$ and $g$, one typically considers the curve $S(t):[0,\infty)\to \mathbb{R}^2$ given by
\[
 S(t) =\sum_{1\le n \le t} g(n)e(f(n)) + \{t\} g(\lfloor t\rfloor +1) e(f(\lfloor t \rfloor +1)),
\]
where $\lfloor x \rfloor$ represents the floor of a real number $x$ and $\{ x\}$ represents the fractional part of $x$.  Geometric aspects of these curves have been well-studied \cite{berrygold,dmf,desh,moore}.  

\begin{figure}[t! b]
 \centering
\subfloat{\includegraphics[width=7cm]{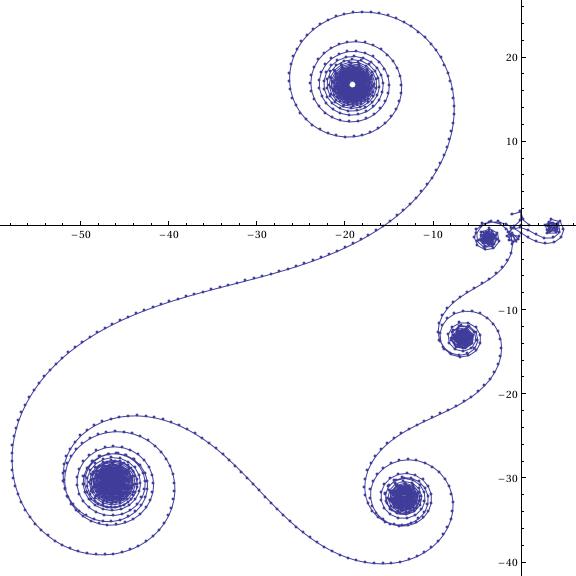}}\qquad 
\subfloat{\includegraphics[width=7cm]{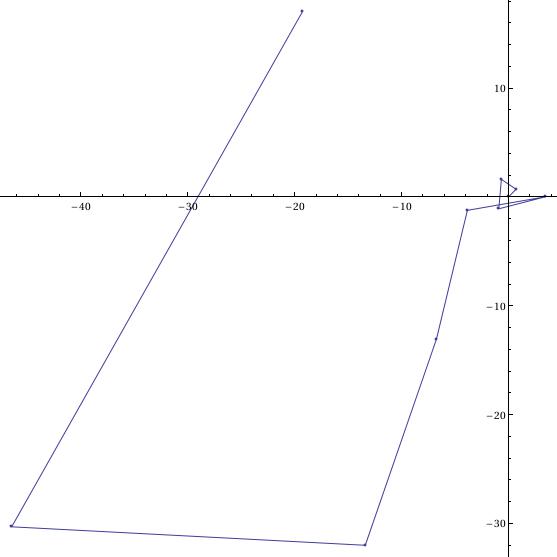}}
\caption{The curves associated to the exponential sum $\sum_{n\le N} e(n \log n)$ and it slightly extended van der Corput transform $\sum_{n\le \log(N)+1}\sqrt{e^{n-1}} \cdot e(-e^{n-1}+1/8)$ with $N=4000$.}
\label{fig:nlogn}
\end{figure}

Especially when $f''$ is small, these curves generate a series of spiral-like figures.  The van der Corput transform of the sum associated to $S(t)$ then generates a new curve that can be seen to connect the center point of successive spirals by straight lines like in Figure \ref{fig:nlogn}.  Thus the van der Corput transform can be seen as smoothing (if $f''$ is small) or roughening (if $f''$ is big) the curve.  An easily accessible explanation for why this occurs is given in \cite{coutkaz2}.

For our example case \eqref{eq:exampleerror}, the corresponding curve is displayed in Figure \ref{fig:exampleerror}.

\begin{figure}[t b]
 \centering
\subfloat{\includegraphics[width=7cm]{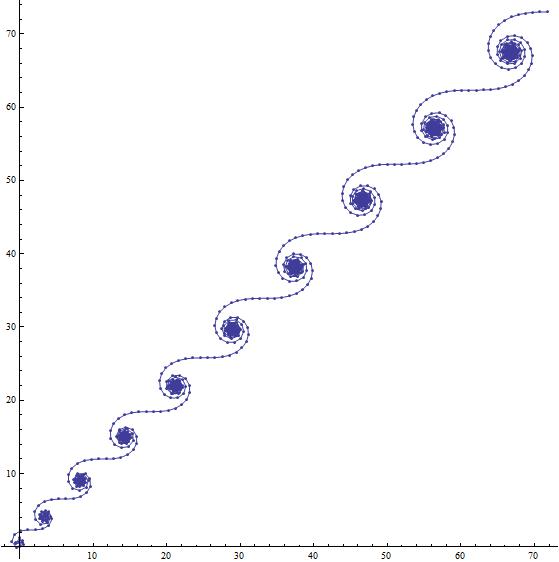}}\qquad
\subfloat{\includegraphics[width=7cm]{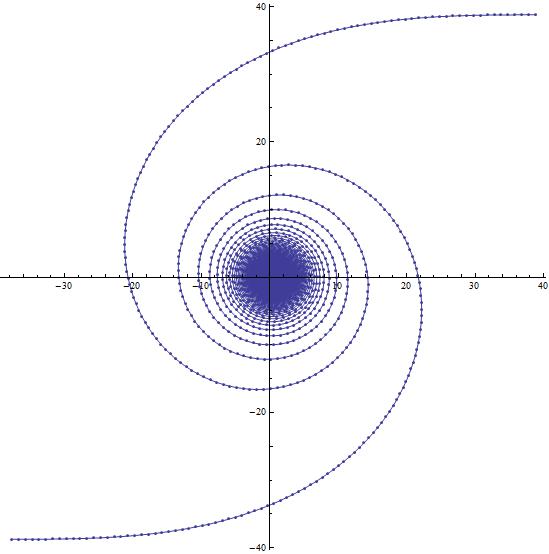}}
\caption{The curve associated to the exponential sums on the left-hand side of \eqref{eq:exampleerror} with $N=1,2,\dots,1200$, and the $500$th spiral of this curve---this and subsequent pictures are off-set closer to the origin.}
\label{fig:exampleerror}
\end{figure}

The term 
\begin{equation}\label{eq:examplecase1}
2\psi(\sqrt{N/12})(3N)^{1/4}    e((N/3)^{3/2}+1/8)
\end{equation}
from Corollary \ref{thm:exampleerror} describes the outer arm of the spirals and the connections between them.  Figure \ref{fig:outerarm} depicts the outer arm of the $500$th spiral and its approximation using Corollary \ref{thm:exampleerror}.  The reason why this approximation appears so poor is due, in part, to the number of points on this part of the spiral. The $m$th spiral has on the order of $m$ points in it.  The outer arm of the spiral, where
\[
\| (N/12)^{1/2} \|\le  (12N)^{-1/4},
\]
accounts for only about $\sqrt{m}$ points.  The big-O term $O\left(N^{5/12}\| (N/12)^{1/2}\|^{2/3}\right)$ is on the order of the explicit term \eqref{eq:examplecase1} when $\|(N/12)^{1/2} \| \asymp (12N)^{-1/4}$, so that \eqref{eq:examplecase1} is a good estimate for only $o(\sqrt{m})$ points, a negligible piece of the spiral.  However, from the picture, we can see that the approximation does a very good job of estimating the rotation around the spiral and appears to be of the correct order of magnitude for the distance to the center of the spiral.  Although we do not calculate the big-O constants explicitly, it is quite possible the big-O terms are smaller in modulus than \eqref{eq:examplecase1} for $\| (N/12)^{1/2} \|\le  (12N)^{-1/4}$.

\begin{figure}[t b]
 \centering
\subfloat{\includegraphics[width=7cm]{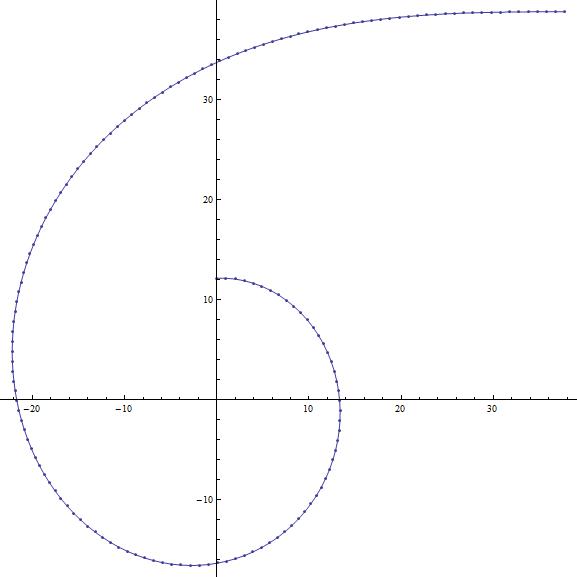}}\qquad
\subfloat{\includegraphics[width=7cm]{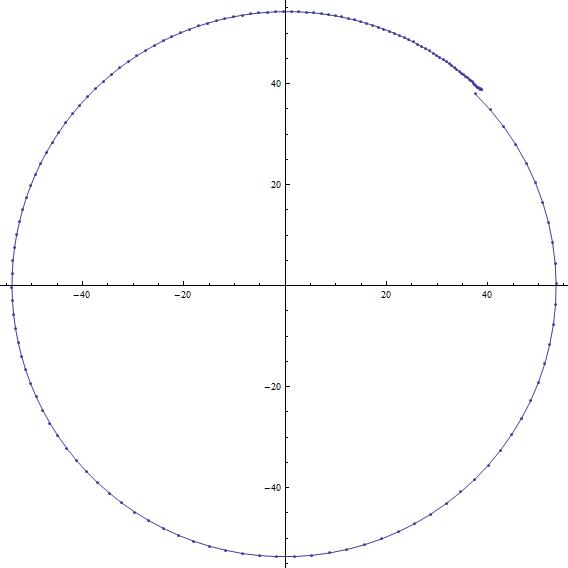}}
\caption{The outer arm of the 500th spiral of the curve associated to \eqref{eq:exampleerror} and its approximation by Corollary \ref{thm:exampleerror}.}
\label{fig:outerarm}
\end{figure}

It is possible to obtain far more explicit results along the outer spiral by making use of the Fresnel integral functions; however, this method does not provide good asymptotic data since it uses values of the Fresnel integral functions for which good asymptotics do not currently exist.  More details will be given in Section \ref{section:fresnel}.

The remaining points (in fact, almost all of the points) are described by the term 
\begin{equation}\label{eq:examplecase2}
e((N/3)^{3/2})\left(\dfrac{1}{2\pi i\langle (N/12)^{1/2}\rangle} - \psi(N,\langle (N/12)^{1/2}\rangle ) \right) ,
\end{equation}
which gives the tight inner weave of a given spiral.  The term $\psi(N,\langle (N/12)^{1/2}\rangle)$ can be bounded by $O(1)$ for all $N$.  The big-O term $ O(N^{-1/2} \|(N/12) ^{1/2}\|^{-3})$ is bigger than $O(1)$ for only $O(m^{2/3})$ of the points on the $m$th spiral, again a negligible piece.  Thus, we know almost all points on the inner weave (and hence almost all points on the full curve) to within a $O(1)$ error and often much better.

\begin{figure}[t b]
 \centering
\subfloat{\includegraphics[width=7cm]{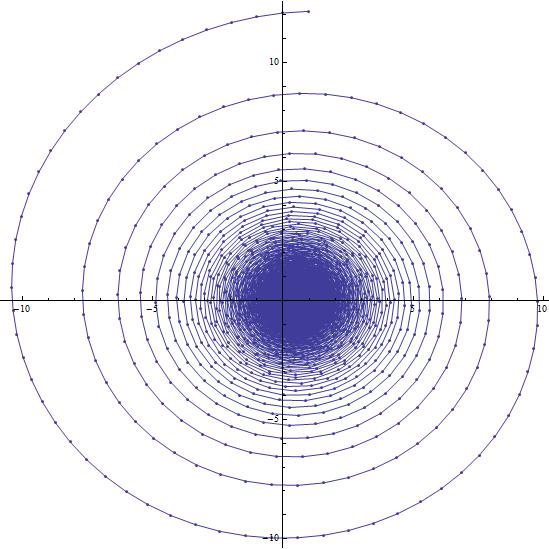}}\qquad
\subfloat{\includegraphics[width=7cm]{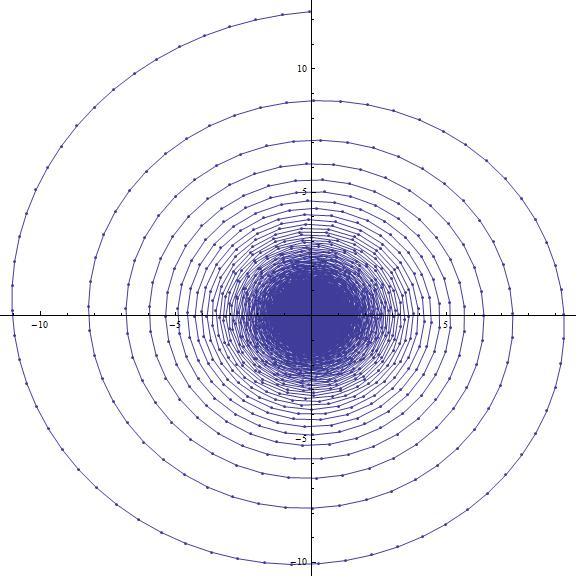}}
\caption{The inner weave of the 500th spiral of the curve associated to \eqref{eq:exampleerror} and its approximation by Corollary \ref{thm:exampleerror}, ignoring the contribution of the $\psi(x,\epsilon)$ term.}
\label{fig:innerspiral}
\end{figure}

The big-O terms grow to the size of the explicit terms \eqref{eq:examplecase1} and \eqref{eq:examplecase2} as $\|(N/12)^{1/2} \|$ approaches $(12N)^{-1/4}$ in both cases, so we don't have good information about how the curve transitions from the inner weave to the outer arm and vice-versa.  We also do not have a good heuristic explanation for why $e(1/8)$ appears in \eqref{eq:examplecase1} on the outer arm, but $1/i=e(-1/4)$ appears in \eqref{eq:examplecase2} for the inner weave.

This paper is arranged as follows.

In Section \ref{section:notation}, we will provide a short catalog of notations that are frequently used in this paper.

In Section \ref{section:heuristics}, we will outline the methods of this paper and how they overcome difficulties encountered in other results.

In Section \ref{section:results}, we will state the main theorems of this paper as well as a briefly discuss them and their assumptions; in this section, we also include a proof of Corollaries \ref{cor:iwakowimprovement} and \ref{thm:karatsubakorolevimprovement}.

In Section \ref{sec:lemstationaryphase}, we will cite a number of lemmas from other papers and books that will be necessary in the proof of the theorems, as well as give propositions that show how the lemmas apply under the specific conditions of the main theorems.

In the remaining sections, we prove the primary theorems of our paper and then prove Corollary \ref{thm:exampleerror}.

\section{Notation}\label{section:notation}

We will frequently use the Landau and Vinogradov asymptotic notations.  The big-O notation $f(x)=O(g(x))$ (equivalently, $f(x)\ll g(x)$) means that there exists some constant $c$ such that $|f(x)|\le c|g(x)|$ on the domain in question.  By $O(f(x))=O(g(x))$, we mean that a function which is asymptotically bounded by $f(x)$ will also be asymptotically bounded by $g(x)$.  The little-o notation $f(x)=o(g(x))$ as $x\to a$---typically with $a$ equal to $\infty$---means that
\[
\lim_{x\to a} \frac{f(x)}{g(x)}=0.
\]
 By $f(x) \asymp g(x)$, we shall mean that $g(x) \ll f(x) \ll g(x)$.

We will need several functions that relate a real number $x$ to the integers nearby it.  Let $\lfloor x\rfloor$ denote the usual floor of a real number $x$, the largest integer less than or equal to $x$, and let $\lceil x\rceil$ denote the usual ceiling of a real number $x$, the smallest integer greater than or equal to $x$.  Let $\{x \}:=x-\lfloor x\rfloor$ denote the fractional part of a real number $x$.  Let $s(x):=\{x\}-1/2$ denote the sawtooth function, and let 
\[
\psi(x):= \begin{cases}
s(x) & x\not\in \mathbb{Z}\\
0 & x\in \mathbb{Z}
\end{cases}
\]
denote the ``smoothed'' sawtooth function, with Fourier series
\[
\psi(x)=-\frac{1}{\pi}\sum_{r=1}^\infty \frac{\sin(2\pi r x)}{r}.
\]
We also wish to have a modified sawtooth function, given by
\[
\psi(x,\epsilon) := -\frac{1}{2\pi i} \lim_{R\to \infty} \sum_{0<|r|< R} \frac{e(rx)}{r+\epsilon} \qquad \text{for }|\epsilon|\le \frac{1}{2}.
\]
The convergence of $\psi(x,\epsilon)$ will be guaranteed by Proposition \ref{prop:fouriertail}.

\begin{figure}[t! b]
 \centering
\subfloat{\includegraphics[width=7cm]{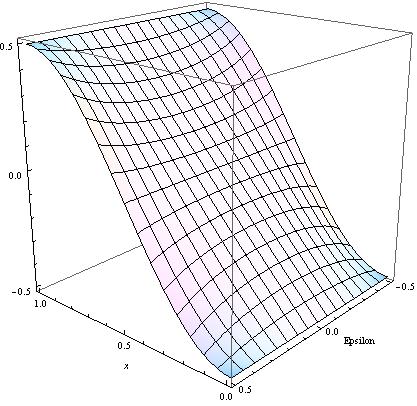}}\qquad 
\subfloat{\includegraphics[width=7cm]{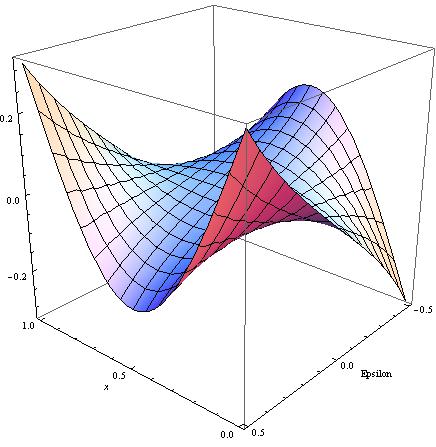}}
\caption{The real (left) and imaginary (right) parts of the function $\psi(x,\epsilon)$}
\end{figure}

Let $\langle x\rangle:=s(x+1/2)$ denote the difference between $x$ and the nearest integer to $x$, let $\left\llbracket x\right\rrbracket:=x-\langle x\rangle$ denote the nearest integer to $x$, and let $\| x\|:=\min\{1-\{x\},\{x\}\}$ denote the distance from $x$ to $\llbracket x \rrbracket$. Let $\|x\|^*$ be given by
\[
\|x\|^*= \begin{cases}
\|x\| & \|x\| \neq 0\\
1 & \|x \| = 0
\end{cases}
\]

A few more definitions will simplify the (nonetheless still complicated) statement of the theorem.  For visibility's sake, we will avoid writing $(x)$ all the time when the choice of argument is always the same.  
Given functions $f$ and $g$, let 
\begin{align*}
H&=gf^{(3)}+3g'f''\\
G&=12gg''(f'')^2\\ 
W_\pm &=\frac{(2g'')^2g'}{(H\pm \sqrt{H^2-G})^2} -\frac{(2g'')^3(f''g)}{(H\pm \sqrt{H^2-G})^3} \\
W_0 &= -\frac{H^2 f^{(3)}}{27g(f'')^5}\\
r_\pm &=f'-\frac{H\pm \sqrt{H^2-G}}{2g''}\\
r_0 &=f'-\frac{3g(f'')^2}{H},
\end{align*}
and given an integer $r$, let
\[
h_r(x) = \frac{(f'(x)-r)g'(x) - g(x)f''(x)}{(f'(x)-r)^3}.
\]

\section{Heuristics for the van der Corput transform and the new ideas of this paper}\label{section:heuristics}

The transform is a simple two-step process: the Poisson summation formula is first applied to obtain a sum of integrals, and then each integral is estimated using stationary phase methods.  In particular, the Poisson summation formula (see Lemma \ref{lem:poissonsummation}) gives
\begin{equation}\label{eq:poisson1}
\sideset{}{^*}\sum_{a\le n \le b}  g(n)e(f(n)) = \lim_{R\to \infty} \sum_{r=-R}^R \int_a^b g(x)e(f(x)-rx) \ dx.
\end{equation}

An oscillatory integral, like 
\begin{equation}\label{eq:oscint}
\int_a^b g(x)e(f(x))\ dx,
\end{equation}
 is said to have a stationary phase point if there exists an $x'\in[a,b]$ such that $f'(x')=0$; if such a point exists, then we expect \eqref{eq:oscint} to be roughly 
\[\frac{g(x')e(f(x')+1/8)}{\sqrt{f''(x')}} \cdot c(x'),\]
plus some small error, where $c(x')$ equals $1/2$ if $x'$ equals $a$ or $b$ and $1$ otherwise.  In our case, the stationary phase point of the integrand
\[
 g(x)e(f(x)-rx)
\]
occurs at $x=x_r$, which is defined by $f'(x_r)-r=0$.  Provided $f'(a)\le r \le f'(b)$, this stationary phase point will be inside the interval of integration, and so we expect
\[
\sideset{}{^*}\sum_{a\le n \le b} g(n) e(f(n)) \approx \sideset{}{^*}\sum_{f'(a)\le r \le f'(b)}  \frac{g(x_r)e(f(x_r)-rx_r+1/8)}{\sqrt{f''(x_r)}}.
\]

The error term for the van der Corput transform thus comes from the error in estimating all the integrals of \eqref{eq:poisson1}.

The first problem in doing such an estimation comes from the integrals without a stationary phase point.  Typically the method one uses to estimate such an integral is an application of integration by parts.
\[
 \int_a^b e(f(x)-rx)\ dx = \left. \frac{e(f(x)-rx)}{2\pi i(f'(x)-r)}\right]_a^b + \int_a^b \frac{f''(x)}{2\pi i(f'(x)-r)^2} e(f(x)-rx) \ dx
\]
(We will assume that $g(x)=1$ for a while to simplify the heuristics.)  The terms $e(f(x)-rx)/2\pi i(f'(x)-r)$ at $a$ and $b$ are referred to as the first-order endpoint contributions and exhibit a great deal of cancellation (see Proposition \ref{prop:fouriertail}); for example, if $f'(a)$ is an integer, then
\[
\lim_{R\to \infty} \sum_{\substack{|r|\le R \\ r\neq f'(a)}} -\frac{e(f(a)-ra)}{2\pi i(f'(a)-r)} = \psi(a) e(f(a)),
\]
and similarly if $f'(b)$ is an integer.

However, we are still left with the integral
\[
 \int_a^b \frac{f''(x)}{2\pi i(f'(x)-r)^2} e(f(x)-rx) \ dx
\]
to evaluate. We could estimate the integral by taking absolute values of the integrand, but we would then obtain an estimate of the size $O((r-f'(a))^{-1})+O((f'(b)-r)^{-1})$, which is of the same order as the first-order endpoint contributions, but without the cancellation that would allow them to be nicely summed.

One possible solution would be to apply integration by parts a second time: this gives
\begin{multline}\label{eq:secondordererror}
\int_a^b \frac{f''(x)}{2\pi i(f'(x)-r)^2} e(f(x)-rx) \ dx = \left. \frac{f''(x)e(f(x)-rx)}{(2\pi i)^2(f'(x)-r)^3} \right]_a^b\\ - \int_a^b \frac{d}{dx}\left( \frac{f''(x)}{(2\pi i)^2(f'(x)-r)^3} \right) e(f(x)-rx) \ dx,
\end{multline}
with a second-order endpoint contribution and a new integral.  The second-order endpoint terms are absolutely convergent, and if $f'(a)$ and $f'(b)$ are integers, they sum to $O(f''(a))+O(f''(b))$, which is a good estimate if $f''$ is small at the endpoints.  On the other hand, if we try to bound the new integral in \eqref{eq:secondordererror} by taking absolute values of the integrand, we get terms of the same order of magnitude as the second-order endpoint contributions (which, unlike the first-order endpoints, are summable).  But in addition, we have the total variation of $f''(x)/(f'(x)-r)^3$ on $[a,b]$, which may be roughly bounded by the sum of the moduli of local maxima and minima of $f''(x)/(f'(x)-r)^3$ on $[a,b]$.

Unless $f(x)$ is quadratic---as in Coutsias and Kazarinoff's case\footnote{Coutsias and Kazarinoff actually analyze the resulting integrals and endpoint contributions in the quadratic case when inegration by parts is repeated many times.}---it is difficult to find good estimates on these integrals.  More applications of integration by parts or the presence of a non-constant $g$ only make things worse.  Therefore, these terms are often avoided entirely by applying the truncated Poisson formula instead of the full Poisson summation formula.

\begin{prop}(Truncated Poisson formula---Proposition 8.7 in \cite{iwakow})
\footnote{An explicit version of the truncated Poisson formula for non-trivial $g$ is given in Lemma 7 of \cite{karatsubakorolev}.}

 Let $f(x)$ be a real function with $f''(x)>0$ on the interval $[a,b]$.  We then have
\begin{equation}\label{eq:truncpoisson}
 \sum_{a<n < b} e(f(n)) = \sum_{\alpha-\epsilon<r<\beta+\epsilon} \int_a^b e(f(x)-rx) \ dx + O(\epsilon^{-1}+\log(\beta-\alpha+2))
\end{equation}
where $\alpha$, $\beta$, and $\epsilon$ are any numbers with $\alpha\le f'(a)\le f'(b)\le \beta$ and $0< \epsilon \le 1$, the implied constant being absolute.
\end{prop}

The truncated Poisson formula is the source of the $O(\log(f'(b)-f'(a)+2))$ error term in many of the results mentioned above.

After applying the truncated Poisson formula with $\beta=f'(b)$ and $\alpha=f'(a)$, one is left with approximately $f'(b)-f'(a)$ integrals with which one hopes to apply stationary phase estimates.  These estimates work best on a small interval around the stationary phase point, where the second derivative of $f$ is fairly constant, and the higher derivatives of $f$ are small.  So we would like to break the integrals in \eqref{eq:truncpoisson} into several pieces, such as 
\begin{equation}\label{eq:stationaryphasesplit}
 \int_a^b = \int_a^{a'} + \int_{a'}^{b'}+\int_{b'}^b
\end{equation}
where the point of stationary phase is near the middle of $[a',b']$.  However, to make this effective, we would again require good estimates on integrals with no stationary phase point (the integrals from $a$ to $a'$ and $b'$ to $b$).  In addition, one is faced with possible first-order (and higher) endpoint contributions at $a'$ and $b'$.  (While estimates of stationary phase integrals benefit from integrating on a small interval, they suffer again if the interval is too small.  The $O(T(a)+T(b))$ terms on line \eqref{eq:Tdef} that appear in stronger results on the van der Corput transform arise from bounding the stationary phase integrals where $x_r$ is closest to $a$ or $b$.)

Because of this, many results on the van der Corput transform seek to treat the entire integral from $a$ to $b$ as a single stationary phase integral, hence the common conditions, as in Theorem \ref{thm:textbook}, that $f''(x)$ has constant order on the entire interval $[a,b]$ and that higher derivatives of $f$ be likewise small on the entire interval $[a,b]$.

The techniques of this paper seek to overcome some of these difficulties.  

First, we use the recent and powerful stationary phase estimates first proved by Huxley \cite{huxley2,huxley1} and refined by Redouaby and Sargos \cite{redsar}.  These estimates show that stationary phase integrals, such as the middle term in \eqref{eq:stationaryphasesplit}, also generate first-order endpoint contributions.  In fact, these contributions directly cancel the first-order endpoint contributions at $a'$ and $b'$ generated by the first and third term in \eqref{eq:stationaryphasesplit}.  Thus we no longer need to treat the full integral from $a$ to $b$ with stationary phase estimates and so can replace the global restrictions on the derivatives of $f$ and $g$ (such as those in Theorem \ref{thm:karatsubakorolev}) with local restrictions.

In particular, we replace the the standard assumptions
\begin{align*}
 f''(x) &\asymp TM^{-2}, \\
|f^{(3)}(x)|&\ll  TM^{-3},\qquad \text{ and } \\
|f^{(4)}(x)|&\ll  TM^{-4}
\end{align*}
for $x\in [a,b]$, with an assumption that looks like
\begin{align*}
f''(z) &\asymp T(x)M(x)^{-2}, \\
|f^{(3)}(z)|&\ll T(x) M(x)^{-3},\qquad \text{ and }  \\
|f^{(4)}(z)|&\ll T(x) M(x)^{-4} 
\end{align*}
for $z\in[x-M(x),x+M(x)]$ and $x\in[a,b]$ with some function $M(x)$ (see Section \ref{section:conditionM} for more details about the function $M(x)$).

Second, to avoid use of the truncated Poisson summation formula, we develop a method to get reasonable bounds on the integrals arising from applying integration by parts twice, as in \eqref{eq:secondordererror}.  Instead of looking at one integer $r$ at a time and counting the contribution of $f''(x)/(f'(x)-r)^3$ for each $x$'s which give local maxima and minima, we instead look at each $x$ and imagine a possible fixed real-valued $r$ that causes the point $x$ to be a critical point of this function.  This then defines a function $r(x)$, and we estimate the sum over all values of $x$ where the function $r(x)$ takes integer values using a variant of Euler-Maclaurin summation (see Proposition \ref{prop:inversesum} and Section \ref{sec:variation}).

Third, since the previous two techniques frequently benefit from $f''$ being large, we develop a method to deal with the large second-order endpoint contributions in this case.  The reason for these terms being large is that 
\[
 \frac{f''(a)}{(f'(a)-r)^3}
\]
is large when $r$ is close to $f'(a)$ (and likewise at $b$).  However, if $f''(a)$ is large and $f^{(3)}(a)$ not too large in comparison, then a small shift in $a$ to, say, $a+\epsilon$ should greatly increase the size of the denominator while keeping the numerator roughly the same size.  Therefore, by altering the endpoints of the integrals in the Poisson summation formula by a small amount when $r$ is close to $f'(a)$ or $f'(b)$, we can reduce the size of the second-order endpoint contributions.

With these techniques, the main theorems of this paper give the van der Corput transform in the following form.
\begin{align*}
\sideset{}{^*}\sum_{a\le n \le b} g(n) e(f(n)) &= \sideset{}{^*}\sum_{f'(a)\le r \le f'(b)}  \frac{g(x_r)e(f(x_r)-rx_r+\frac{1}{8})}{\sqrt{f''(x_r)}}\\
&\qquad - \mathcal{D}(b)+\mathcal{D}(a)+O(\Delta_1+\Delta_2+\Delta_3+\Delta_4)
\end{align*}
The $\mathcal{D}$ terms will usually equal the first-order endpoint contributions and will be explicit in many cases.  The term $\Delta_1$ will estimate the error that arises when there exists an $x_r$ close to $a$ or $b$ and is analogous to the $T(a)+T(b)$ terms present in the results of Kolesnik, Liu, and Karatsuba and Korolev.  The term $\Delta_2$ will estimate the second-order endpoint contributions at $a$ and $b$.  The term $\Delta_3$ (and a bit of $\Delta_4$) will estimate the remaining error in applying stationary-phase methods and the second-order endpoint contributions at $a'$ and $b'$.  The remainder of $\Delta_4$ will estimate the size of the integrals that arise from applying integration by parts twice.

The techniques outlined in this section allow for a number of interesting extensions to the van der Corput transform.

Since we replace the global restrictions on the derivatives of $f$ and $g$ with local restrictions, we can apply the van der Corput transform to many new sums, including those where $f$ and $g$ have moderate oscillations.  Likewise, we can now directly apply the van der Corput transform to sums where $f''$ is large.  

The function $f(f'^{-1}(x))-xf'^{-1}(x)$ that appears on the right-hand side of the van der Corput transform \eqref{eq:vdCtransform} is sometimes referred to as the van der Corput reciprocal of the function $f$. \footnote{Redouaby and Sargos \cite{redsar2} and the papers cited within have studied the van der Corput reciprocal in much greater detail, including useful asymptotics for how the reciprocal changes as $f$ is perturbed.} If the second derivative of $f$ is large, then the second derivative of the van der Corput reciprocal of $f$ will be small, so the results cited above, if they cannot be applied directly to a sum, can be---and frequently are---applied ``backwards'' to the van der Corput transform of the sum.

For example, if we attempted to apply the form of the error from Theorem \ref{thm:karatsubakorolev} directly to the sum
\[
\sum_{0\le n \le N} e(\alpha \cdot \beta^n)
\]  
for some $\beta>1$, then the error term $O(\log(f'(b)-f'(a)+2))$ would be at least $O(N)$, the size of the trivial bound on the sum (and that doesn't even take into account the size of the implicit constant!).  Theorem \ref{thm:main} allows us to apply the transform to this sum directly and obtain a $O(1)$ error.

In general, if $f''\ge 1$ is reasonably large compared to $g$ and both are free of wild oscillations, then $\Delta$ should be no larger than $O(\max_{a\le x \le b} g(x))$.

As a final note, the results of this paper do not give an improvement in all possible cases.  If, for example, $f''$ varies between a very large value and a value very close to zero, then our results may give very poor error terms.  Another case where our results do not give any improvement is 
\[
\sideset{}{^*}\sum_{a\le n \le b} \sin(\alpha n) e(\beta n^2).
\]

\section{Details of the theorem}\label{section:results}

\subsection{The auxiliary function $M(x)$ and condition $(M)$}\label{section:conditionM}

We will use a function $M(x)$ to measure the length of an interval around $x$ where the functions $f''(x)$ and $g(x)$ are well-approximated by their Taylor polynomials up to the second degree.  The larger $M(x)$ is, the more linear the functions $f''(x)$ and $g(x)$ will appear at the point $x$.

 To be more concrete, by condition $(M)$, we shall refer to the existence of positive, bounded, continuously differentiable functions $M(x)$ and $U(x)$ on $[a,b]$, along with several associated positive constants $C_2$, $C_{2^-}$, $C_4$, $D_0$, $D_1$, $D_2$, $\delta$, which together satisfy several conditions:
\begin{enumerate}
\item[(I)] $\max\{M(a),M(b)\}\le b-a$;
\item[(II)] $\delta<1$, and let $\eta := 3 \delta/C_{2^-}$ satisfy $\eta<2$;
\item[(III)] If $J:=[a-c(a)\cdot M(a),b+c(b) \cdot M(b)]$, where
\[
c(x)=\begin{cases}
0, & \text{if }\mathbb{Z}\cap (f'(x)-f''(x),f'(x)+f''(x))\setminus \{f'(x)\} = \emptyset \text{, and}\\
1, & \text{otherwise,}
\end{cases}
\]
 then $f$ is $C^4[J]$ and $g$ is $C^3[J]$; and,
\item[(IV)] If $I_x$ denotes the intersection $[x-M(x),x+M(x)]\cap J$, then for all $x\in [a,b]$ we expect all the following inequalities to hold for $z\in I_x$:
\begin{align*}
\frac{1}{C_{2^-}}f''(x)\le f''(z) &\le C_2 f''(x), & & &|g(z)| &\le D_0 U(x) ,\\
 |f^{(3)}(z)| &\le \eta \frac{f''(x)}{ M(x)}, & & &  |g'(z)| &\le D_1 \frac{ U(x)}{M(x)}, \\
 |f^{(4)}(z)| &\le \eta^2 C_4 \frac{ f''(x)}{ M(x)^2},& &\text{and} & |g''(z)| &\le D_2 \frac{U(x)}{M(x)^2}.
\end{align*}
\end{enumerate}

For any given $f(x)$, $g(x)$, $a$, and $b$, there are infinitely many possible choices of the functions $M(x)$, $U(x)$, and the associated constants that satisfy the above conditions.  Therefore, for the remainder of the paper, we shall assume that if $f(x)$, $g(x)$, $a$, and $b$ remain unchanged, the particular functions $M(x)$ and $U(x)$ and associated constants we reference will be likewise unchanged.  If $f(x)$ and $g(x)$ are unchanged, but $a$ and $b$ allowed to vary, then we will assume, first, that the associated constants will be unchanged, and, second, that given two intervals $[a_1,b_1]$ and $[a_2,b_2]$ with corresponding auxiliary functions $M_1(x)$, $U_1(x)$ and $M_2(x)$, $U_2(x)$, we have $M_1(x)=M_2(x)$ and $U_1(x)=U_2(x)$ on $(a_1,b_1)\cap (a_2,b_2)$.

Since condition $(M)$ is rather intricate, we pause a moment to illuminate it further.

\begin{enumerate}
\item[(II)] The requirement of $\delta<1$ and $\eta<2$ in condition $(M)$ part (II) guarantee that the Taylor approximations to $f''$ has good properties (see Proposition \ref{prop:f'bound}, Lemma \ref{lem:redsar1}, and the remark following the lemma).

\item[(III)] The restriction of condition $(M)$ part (III) to have $f$ and $g$ be several times continuously differentiable on the larger interval $J$ is often no restriction at all.  Many applications of the van der Corput transform have $f(x)$ and $g(x)$ be polynomials, exponentials, logarithms, or other $C^{\infty}$ functions that exist on large domains.  

The bounds on the derivatives of $f''$ and $g$ in condition $(M)$ part (IV) allow us to use estimates on stationary phase integrals that will be discussed in Section \ref{sec:lemstationaryphase}.  By expecting the properties of $f''$ and $g$ to extend to the larger interval $J$, we can also extend the integrals under consideration to this larger interval.

(While typical stationary phase results require only that $g$ is twice continuously differentiable, we require three times in order for the functions $W_\pm'$ and $r_\pm'$ to exist.)

\item[(IV)] To understand the complex system of inequalities in condition $(M)$ part (IV), it is helpful to rewrite them in terms of Taylor approximations.  The first three inequalities then become 
\begin{align}
f''(z) &\asymp f''(x) \notag \\
f''(z) &= f''(x)\left( 1+O\left( \frac{\eta}{M(x)}(z-x)\right) \right) \notag \\
f''(z) &= f''(x)\left( 1+O\left( \frac{\eta^2 C_4}{M(x)^2} (z-x)^2 \right) \right) + f^{(3)}(x)(z-x) \label{eq:conditionEvar}
\end{align}
for $z\in I_x$ with implicit constant $1$ in the second and third lines.  The equality on line \eqref{eq:conditionEvar} gives definite form to our earlier statement that the larger $M(x)$ can be, the more linear $f''$ appears at the point $x$.

Alternately, one can interpret $M(x)$ as somehow representing the rate of decay as one takes successive derivatives.  At the point $z=x$, the various inequalities imply---among other things---that $|f^{(3)}(x)| \ll f''(x)/M(x)$ and $|f^{(4)}(x)| \ll f''(x)/M(x)^2$.  

The conspicuous absense of a constant $C_3$ is intentional.  We could include such a constant, but later estimates are made slightly simpler by having $C_3=1$
\end{enumerate}

Here are some examples of the function $M(x)$ in different circumstances.  In each case we assume that the associated constants always take the same values: $\delta$ equals $1/2$, and $C_2$, $C_{2^-}$, $C_4$, $D_0$, $D_1$, and $D_2$ all equal $2$.
\begin{itemize}
\item If $f(x)$ and $g(x)$ are polynomials, there exists an $\epsilon>0$ such that condition $(M)$ is satisfied with $M(x)=\epsilon \cdot x$ and $U(x)=g(x)$ for sufficiently large $x$.
\item If $f(x)$ and $g(x)$ are exponential functions ($\alpha\cdot \beta^x $, with $\beta>1$), there exists an $\epsilon>0$ such that condition $(M)$ is satisfied with $M(x)=\epsilon$ and $U(x)=g(x)$.
\item If $f(x)=t(\log x)/2\pi$ and $g(x)=x^{-\sigma}$, which corresponds to the Riemann Zeta function at $z=\sigma+i t$, then there exists an $\epsilon >0$ such that condition $(M)$ is satisfied with $M(x)=\epsilon x$ and $U(x)=g(x)$ independent of $t$.
\item If $f(x) = \alpha x^2 +\beta x^{-1} \sin(\gamma x)$ and $g(x)=1$, then there exists an $\epsilon >0$ such that condition $(M)$ is satisfied with $M(x)=\epsilon\sqrt{x}$ and $U(x)=1$.
\item If $f(x)$ is a power function and $g(x) = \sin(\alpha x)$ for some constant $\alpha$, then there exists an $\epsilon>0$ such that condition $(M)$ is satisfied with $M(x)=\epsilon$ and $U(x)=1$.
\end{itemize}

\subsection{The assumptions}\label{section:assumptions}

Assume that condition $(M)$ holds for some function $M(x)$ and $U(x)$.  Let $J$ be as in condition $(M)$ part (III).

We assume that $f''(x)$ is a positive, real-valued function, bounded away from $0$ on the interval $J$, and that $g(x)$ is a real-valued function on $J$ as well. (The case $f''(x)<0$ may be considered by taking the conjugate of the sum.)  This will guarantee that $f'(x)$ is a continuous, monotonic function, so that $x_r:= f'^{-1}(r)$ is well defined.  

Let $J_\pm$ be the union of all intervals $[a',b']\subset J$ such that $G(x)\neq 0$ and $H(x)^2-G(x)\ge 0$ for $x\in [a',b']$, let $J_0$ be the union of all intervals $[a',b']\subset J$ such that $g''(x)=0$, $g(x)\neq 0$, and $H(x) \neq 0$ for $x\in [a',b']$, and let $J_0$ be the set of points $x\in J$ such that $g(x)=0$ but $g'(x)\neq 0$ and $g''(x)\neq0$.  It is possible that $J_\pm$ will contain isolated points due to $H(x)^2-G(x)$ having a $0$ on an interval where it is otherwise negative, and it is also possible that $J_0$ will contain isoluted points due to $g''(x)$ having a $0$ on an interval where it is otherwise non-zero; we denote the set of isolated points in $J_\pm$ and $J_0$ by $J_\pm^*$ and $J_0^*$ respectively.  Let $\partial J_\pm$ and $\partial J_0$ denote the endpoints of the non-zero-length intervals contained in their respective sets.  In particular, $\partial J_\pm \cap J_\pm^* = \varnothing$ and $\partial J_0 \cap J_0^* = \varnothing$.

For our final assumption, suppose that if $(a',b')$ is an interval contained in $J_\pm$, then either $H(x)^2-G(x)$ equals $0$ on the whole interval and $H(x)$ does not equal $0$ at any point on the interval nor does it tend to $0$ at the endpoints, or $H(x)^2-G(x)$ does not equal $0$ at any point on the interval nor does it tend to $0$ at the endpoints; and if $(a',b')$ is an interval contained in $J_0$, then $g(x)$ does not tend to $0$ at the endpoints of this interval.

\subsection{Statement of the theorems and proofs of some corollaries}

\begin{thm}\label{thm:main}
Assume all the conditions of Section \ref{section:assumptions} hold.  Then 
\[
\sideset{}{^*}\sum_{a\le n \le b} g(n) e(f(n)) = \sideset{}{^*}\sum_{f'(a)\le r \le f'(b)} \frac{g(x_r)e(f(x_r)-rx_r+\frac{1}{8})}{\sqrt{f''(x_r)}}-\mathcal{D}(b)+\mathcal{D}(a)+\Delta
\]
where 
\begin{align*}
\mathcal{D}(x) &= \mathcal{D}^\circ(x)+\mathcal{D}^*(x)\\
\mathcal{D}^\circ (x) &= 
\begin{cases}
\begin{aligned}&g(x)e(f(x)+\left\llbracket f'(x)\right\rrbracket x)\times \\ & \quad\times \left(-\dfrac{1}{ 2\pi i \langle f'(x) \rangle} + \psi(x,\langle f'(x) \rangle)\right) \end{aligned}& f''(x) < \| f'(x) \|\\
g(x)e(f(x)+\left\llbracket f'(x)\right\rrbracket x)\psi(x,\langle f'(x) \rangle)& \| f'(x) \| \le f''(x)< 1-\|f'(x)\|\\
O(U(x)) & 1-\|f'(x)\| \le f''(x)<1\\
O\left(U(x)\left( 1+ \dfrac{1}{M(x)}+\dfrac{1}{\sqrt{f''(x)}M(x)} \right)\right) & f''(x) \ge 1 \\
\end{cases}\\
\mathcal{D}^*(x) &= \begin{cases}
\dfrac{g(x)f^{(3)}(x)e(f(x))}{6\pi i f''(x)^2}-\dfrac{g'(x)e(f(x))}{2\pi i f''(x)} & \|f'(x)\|=0\\
0 & \|f'(x)\|\neq 0
\end{cases}\\
\Delta &= \sum_{i=1}^3 O(\Delta_i(a)+\Delta_i(b)) + \Delta_4\\
 \Delta_1(x) &= \begin{cases}
                 \min\left\{ \dfrac{U(x)}{\sqrt{f''(x)}} , \dfrac{U(x)}{\| f'(x)\|} \right\} & \| f'(x)\| \neq 0, \quad m_x \ge 1\\
                 \dfrac{U(x)}{f''(x)^2(b-a)^3} & \| f'(x)\| =0 \\
                 0 & \| f'(x)\| \neq 0, \quad m_x = 0
                \end{cases}\\
\Delta_2(x) &=\frac{U(x)}{f''(x)^2M(x)^3}(1+\sqrt{f''(x)}M(x))(1+f''(x))+ \frac{U(x)m_x}{f''(x)M(x)}\\
&\qquad + \begin{cases}
       \begin{aligned}&\dfrac{U(x)}{M(x)}\min\left\{1,\dfrac{1}{f''(x)}  \right\}\\ &\quad +U(x)\min\left\{f''(x),\dfrac{1}{f''(x)}  \right\}\end{aligned} & \text{if }\|f'(x)\|=0 \text{ or } m_x \ge 1\\
       \dfrac{U(x)}{M(x)\|f'(x)\|^2} + \dfrac{U(x)f''(x)}{\|f'(x)\|^3} & \text{otherwise}
      \end{cases}\\
\Delta_3(a) &= \int_{\overline{a}}^{b} \frac{U(x)}{f''(x)(x-a)^3}\left(1+\frac{1}{f''(x)M(x)}+\frac{1}{f''(x)(x-a)}  \right) \ dx \\
&\qquad + \frac{U(\overline{a})}{f''(\overline{a})^2(\overline{a}-a)^3} +\frac{U(b)}{f''(b)^2(b-a)^3}\\
\Delta_3(b) &= \int_{a}^{\overline{b}} \frac{U(x)}{f''(x)(b-x)^3}\left(1+\frac{1}{f''(x)M(x)}+\frac{1}{f''(x)(b-x)}  \right) \ dx \\
&\qquad + \frac{U(\overline{b})}{f''(\overline{b})^2(b-\overline{b})^3} + \frac{U(a)}{f''(a)^2(b-a)^3}\\
\Delta_4& = \int_a^b \frac{U(x)}{f''(x)M(x)^3} \left(1+\sqrt{f''(x)}M(x)\right) \left( 1+ \frac{1+|M'(x)|}{f''(x)M(x)}\right) \ dx \\
&\qquad + \mathcal{K}(J_0,W_0,r_0)+ \mathcal{K}(J_\pm,W_+,r_+) + \mathcal{K}(J_\pm,W_-,r_-)\\
&\qquad + \sum_{x \in J_{null}} \frac{g''(x)^2}{g'(x)f''(x)^2}
\end{align*}
and
\begin{align*}
\mathcal{K}(I,W,r)&= \int_I \left(|W(x)| |r'(x) | +|W'(x)| \right) \ dx + \sum_{x\in I^*} |W(x)|\\
&\qquad \qquad + \sum_{\substack{x\in I \text{ and } r'\text{ changes sign at }x \\ \text{or }x\in \partial I}} |s(x)\cdot W(x)|.
\end{align*}

The number $m_x$ equals the cardinality of the set
\[
 \mathbb{Z}\cap (f'(x)-f''(x),f'(x)+f''(x))\setminus\{f'(x)\}.
\]
In particular, 
\[
 m_x=\begin{cases}
      0 & f''(x) \le \| f'(x) \|\\
      O(1+f''(x)) & f''(x) > \| f'(x) \|
     \end{cases}.
\]
Also, $\overline{a}$ equals the smallest value in the interval $[a+\min\{M(a),C_2^{-1}\},b]$ such that $f'(\overline{a})$ is an integer; if no such value exists, then we may take $\Delta_3(a)=0$.  Similarly, $\overline{b}$ equals the largest value in the interval $[a,b-\min\{M(b),C_2^{-1}\}]$ such that $f'(\overline{b})$ is an integer, and if no such value exists, then we may take $\Delta_3(b)=0$.

The functions $M(x)$ and $U(x)$ are as in condition $(M)$, the sets $J_\pm$, $J_\pm^*$, $J_0$, and $J_0^*$ are all as in Section \ref{section:assumptions},  and the functions $\psi$, $s$, $W_0$, $r_0$, $W_\pm$, $r_\pm$, $\langle \cdot \rangle$, $\llbracket \cdot \rrbracket$, and $\| \cdot \|$ are all as in Section \ref{section:notation}.

The implicit constants in the big-O terms are dependent upon the constants $C_2$, $C_{2^-}$, $C_4$, $D_0$, $D_1$, $D_2$, $\delta$ in condition $(M)$.  The convergence of the integrals in $\Delta_4$ will be guaranteed by Proposition \ref{prop:integrability}
\end{thm}

Before continuing to further refinements of the main theorem, let us pause to assure the reader that the vast collection of error terms above are more psychologically daunting than they are computationally difficult.  As an instructive example, consider the case of Corollary \ref{cor:iwakowimprovement}, where
\[
f(x) = \frac{X}{\alpha} \left( \frac{x}{N} \right)^\alpha, \qquad g(x) = \left( \frac{\alpha}{x} \right)^{\frac{1}{2}},
\]
$a=N$, and $b=\nu N$.  $N$ and $X$ are assumed to be positive and $\alpha$ and $\nu$ both strictly greater than $1$.  Without loss of generality, we may assume $N\le \sqrt{X}$ (otherwise, we would instead consider the right-hand side of \eqref{eq:iwakowexample}).  

Note that $f(x)$ and $g(x)$ are both power functions, so condition $(M)$ holds with $C_2$, $C_{2^-}$, $C_4$, $D_0$, $D_1$, and $D_2$ all equal to $2$, $\delta$ equal to $1/2$, $M(x)$ equal to $\epsilon x$ for some small $\epsilon>0$ dependent on $\alpha$ only, and $U(x)$ equal to $g(x)$.

The benefit of our conditions $1 \ll N \le \sqrt{X}$ is that on the interval $[N,\nu N]$, the functions $f''(x)$ and $M(x)$ are both bounded below by some constant (depending on $\alpha$ and the implicit constant in $1 \ll N$).  This immediately gives that the $\mathcal{D}(x)$, $\Delta_1(x)$, and $\Delta_2(x)$ terms are all $O(U(a)+U(b))=O(N^{-1/2})$, since each summand of these three terms contains a factor of $U(x)$ divided by some positive power of $f''(x)$ and $M(x)$, which are bounded from below.  (Note that $\|f'(x) \|$ terms only appear if $\|f'(x)\|$ is bigger than $f''(x)$ to begin with.)

For the $\Delta_3$ terms, we need to consider $\overline{a}$ and $\overline{b}$.  It might be quite difficult to calculate these values explicitly, but we only make the $\Delta_3$ terms bigger by assuming $\overline{a}$ is as small as possible (that is, $\overline{a}=a+\min\{M(a),C_2^{-1}\}$) and $\overline{b}$ is as large as possible. Since $M(x) \gg 1$ on $[N,\nu N]$, we have that $\overline{a}-a \gg 1$ and $b-\overline{b}\gg 1$.  Thus, by first applying the fact that $f''(x)\gg 1$ and $M(x) \gg 1$ and then that $\overline{a}-a \gg 1$, the integral in $\Delta_3(a)$ may be bounded as
\begin{align*}
&\int_{\overline{a}}^{b} \frac{U(x)}{f''(x)(x-a)^3}\left(1+\frac{1}{f''(x)M(x)}+\frac{1}{f''(x)(x-a)}  \right) \ dx\\
&\qquad \ll U(a) \int_{\overline{a}}^{b} \frac{1}{(x-a)^3}+\frac{1}{(x-a)^4} \ dx\\
&\qquad \ll \frac{U(a)}{(\overline{a}-a)^2} \\
&\qquad \ll U(a),
\end{align*}
which is $O(N^{-1/2})$ again.  A similar argument shows that the remaining terms in $\Delta_3(a)$ and all the terms in $\Delta_3(b)$ are $O(N^{-1/2})$ as well.

This is one reason why the error terms are much easier to deal with than at first glance: \emph{provided $f''(x)\gg 1$ and $M(x)\gg 1$ on $[a,b]$, the error terms $\mathcal{D}(x)$ and $\Delta_i(x)$ for $i=1,2,3$ are all bounded by $O(\max_{[a,b]}U(x))$, where the implicit constant is dependent on the implicit constants in the lower bounds of $f''$ and $M$.}

To nicely bound the first integral in $\Delta_4$, however, we must use the fact that $M(x)$ and $f''(x)$ are not only bounded below, but also somewhat large.  In particular, we have
\begin{align*}
 &\int_a^b \frac{U(x)}{f''(x)M(x)^3} \left(1+\sqrt{f''(x)}M(x)\right) \left( 1+ \frac{1+|M'(x)|}{f''(x)M(x)}\right) \ dx\\
 &\qquad \ll \int_a^b \frac{U(x)}{M(x)} \ dx \\
 &\qquad \ll \int_N^{\infty} x^{-3/2} \ dx \\
 &\qquad \ll N^{-1/2}.
\end{align*}

The greatest complication comes in evaluating the remaining condition from Section \ref{section:assumptions} dealing with the functions $G$ and $H$ and the remaining error terms in $\Delta_4$.  In this case, the function $g''$ is never zero on the interval $J=[N(1-\epsilon),\nu N(1+\epsilon)]$.  Hence $J_0$ and $J_{null}$ are empty, so both $\mathcal{K}(J_0,W_0,r_0)$ and the final sum of $\Delta_4$ are $0$.  This also shows that the functions $f$ and $g$ satisfy the remaining condition from Section \ref{section:assumptions}.

Since $f$ and $g$ are power functions, we can quickly verify, with simple hand calculations, that the functions $H$, $G$, $W_+$, and $W_-$ are all power functions as well; that is, each is of the form $c_1 \cdot x^{c_2}$.  In particular, the functions $W_\pm(x)$ are equal to
\[
c_{\alpha,\pm} \left( \frac{N^\alpha}{X} \right)^{2} x^{\frac{1}{2}-2\alpha},
\]
where $c_{\alpha,\pm}$ is a constant depending on $\alpha$ and the sign of $\pm$.  The functions $r_\pm(x)$ are equal to
\[
c'_{\alpha,\pm} \frac{X}{N^\alpha} x^{\alpha-6} +c''_{\alpha,\pm} \frac{X}{N^\alpha} x^{\alpha-1},
\]
with new constants dependent on $\alpha$ and the sign of $\pm$.

 Also, $H(x)^2-G(x)$ is a power function; so, $J_\pm$ is either the interval $[N(1-\epsilon),\nu N(1+\epsilon)]$ or is empty (if $H(x)^2-G(x) <0$ for all $x$, in which case the remaining $\mathcal{K}$ terms equal $0$ and we are finished).

In general, we expect that if $g(x)$ and $f''(x)$ share the same rate of decay in their derivatives, then the various functions $H$, $G$, and the $W$ and $r$ functions should all be relatively well-behaved.  In our case, since $g$ and $f''$ are both power functions, taking derivatives of each entails multiplying by a constant and dividing by $x$.  Thus, we have that $g(x) f^{(3)}(x)$ is a constant multiple of $g'(x)f''(x)$, hence why $H$ is a power function as well.

Returning our attention to the remaining error terms of $\Delta_4$, first consider the integrals.  We have
\begin{align*}
 &\int_{J_\pm} \left( |W_-(x)| |r_-'(x)|+|W_+(x)| |r_+'(x)|\right)+\left( |W_+'(x)|+ |W_-'(x)|\right)\ dx\\
 &\qquad \ll \int_{N(1-\epsilon)}^{\nu N(1+\epsilon)} \frac{N^\alpha}{X} x^{-\frac{1}{2}-\alpha} + \frac{N^{2\alpha}}{X^2} x^{-\frac{1}{2}-2\alpha} \ dx \\
& \qquad \ll \frac{N^\alpha}{X} \cdot N^{\frac{1}{2}-\alpha} + \frac{N^{2\alpha}}{X^2}\cdot N^{\frac{1}{2}-2\alpha} \\
&\qquad \ll \frac{N^{1/2}}{X}\\
&\qquad \le N^{-1/2}
\end{align*}
The functions $r'_+(x)$ and $r'_-(x)$ have at most $5$ zeroes on $[a,b]$, and since the functions $W_+(x)$ and $W_-(x)$ are decreasing and since $J_\pm^*$ is empty, the remaining terms in $\Delta_4$ are bounded by
\begin{align*}
&|s(N(1-\epsilon))\cdot W_-(N(1-\epsilon))|+|s(N(1-\epsilon))\cdot W_+(N(1-\epsilon))|\\
&\qquad \ll  \left( \frac{N^\alpha}{X} \right)^{2} N^{\frac{1}{2}-2\alpha}\\
&\qquad \ll N^{-1/2}.
\end{align*}
(We implicitly used that $N\gg 1$ and $M\gg 1$ implies $X\gg 1$.)  

Since we assumed $N\le \sqrt{X}$, we have $N^{-1/2} \ge M^{-1/2}$, this completes the proof of Corollary \ref{cor:iwakowimprovement}.

Under fairly strong and yet quite common conditions, we can remove much of the complication from the assumptions of Section \ref{section:assumptions} and from $\Delta_4$.

\begin{thm}\label{thm:alternate4}
Assume that $f$ and $g$ be real-valued functions satisfying $f''(x) > 0 $ on $[a,b]$ and condition $(M)$ for a function $M(x)$ such that
\[
M(x) \ge \max\{b-x,x-a\}
\]
on $[a,b]$.  Also assume that $m_a=m_b=0$, where, as in Theorem \ref{thm:main}, $m_x$ is the cardinality of the set
\[
 \mathbb{Z}\cap (f'(x)-f''(x),f'(x)+f''(x))\setminus\{f'(x)\}.
\]

Under these conditions, the result of Theorem \ref{thm:main} holds, with $\Delta_4$ now just equal to 
\[
\int_a^b \frac{U(x)}{f''(x)M(x)^3} \left(1+\sqrt{f''(x)}M(x)\right) \left( 1+ \frac{1+|M'(x)|}{f''(x)M(x)}\right) \ dx.
\]

The condition that $m_a=m_b=0$ also simplifies more of the terms to the following.
\begin{align*}
\mathcal{D}(x)+O(\Delta_1(x)) &= 
\begin{cases}
\begin{aligned}&g(x)e(f(x)+\left\llbracket f'(x)\right\rrbracket x)\times \\ & \quad\times \left(-\dfrac{1}{ 2\pi i \langle f'(x) \rangle} + \psi(x,\langle f'(x) \rangle)\right) \end{aligned}& \|f'(x) \| \neq 0 \\
\begin{aligned}&\dfrac{g(x)f^{(3)}(x)e(f(x))}{6\pi i f''(x)^2}-\dfrac{g'(x)e(f(x))}{2\pi i f''(x)}\\ &\quad+ O\left(\dfrac{U(x)}{f''(x)^2(b-a)^3}\right)\end{aligned} & \| f'(x) \| =0
\end{cases}
\end{align*}
\end{thm}

Using Theorem \ref{thm:alternate4}, we may now prove Corollary \ref{thm:karatsubakorolevimprovement}.

For Corollary \ref{thm:karatsubakorolevimprovement} recall that the conditions are that $f(x)$ and $g(x)$ are real-valued functions, $f(x)$ four times continuously differentiable and $g(x)$ two times continuously differentiable on $[a,b]$.  Also,
\[
f''(x) \gg T/M^2 \qquad f^{(2+r)}(x)\ll T/M^{2+r} \qquad g^{(r)}(x) \ll U/M^r
\]
for $r=0,1,2$ on $[a,b]$, for constants $T\gg 1$, $M=b-a\gg 1$, and $U$.

First, we wish to remove from consideration the case when 
\begin{equation}\label{eq:tooclosetointeger}
0< \|f'(\mu)\| < \sqrt{f''(\mu)}
\end{equation} 
for $\mu$ equal to $a$ or $b$.  

Suppose that not only is \eqref{eq:tooclosetointeger} satisfied for both $a$ and $b$, but that $\llbracket f'(a) \rrbracket = \llbracket f'(b) \rrbracket$---that is, the nearest integer to $f'(a)$ and $f'(b)$ are the same.  In this case, we have by the mean value theorem that $f'(b)-f'(a)=(b-a)f''(\xi)$ for some $\xi \in [a,b]$.  But by our assumption on the size of $\|f'(a)\|$ and $\|f'(b)\|$, we have that $f'(b)-f'(a) \ll \sqrt{T}/M$.  Combining these we have that $b-a \ll f''(\xi)^{-1} \sqrt{T}/M \ll M/\sqrt{T}$.  Thus we have
\[
\sideset{}{^*}\sum_{a\le n \le b} g(n)e(f(n)) = O\left( \frac{UM}{\sqrt{T}} \right).
\]
We can add in the sum
\[
\sideset{}{^*}\sum_{f'(a)\le r \le f'(b)} \frac{g(x_r)}{\sqrt{f''(x_r)}}e(f(x_r)-rx_r+1/8)
\]
to the right-hand side, since this sum by assumption has at most one term, which is of size $UM/\sqrt{T}$.  Thus the corollary is proved in this case, since $U/\sqrt{T}$ ad $UM/T^{3/2}$ are bounded by $UM/\sqrt{T}$.

If $\llbracket f'(a) \rrbracket$ does not equal $\llbracket f'(b) \rrbracket$, then we may replace an endpoint $\mu$ where \eqref{eq:tooclosetointeger} holds with $\mu'$ the closest value to $\mu$ inside the interval $[a,b]$ where $\|f'(\mu')\| = \sqrt{f''(\mu')}$ or $\|f'(\mu')\|=0$; by the same argument as the previous paragraph, doing this generates an error of size $O(UM/\sqrt{T})$.  The only way we could not do this is if $\llbracket f'(a) \rrbracket +1 = \llbracket f'(b) \rrbracket$ and $f''(x)$ is at least $1/4$ at some point $x$ on $[a,b]$, in which case, $f'(b)-f'(a) \le 1 \ll \sqrt{T}/M$,  so that $b-a \ll M/\sqrt{T}$ and the corollary follows in this case by the same argument as in the last paragraph.

Thus it suffices now to prove Corollary \ref{thm:karatsubakorolevimprovement} in the case where $\|f'(\mu)\|\ge \sqrt{f''(\mu)}$ or $\|f'(\mu)\|=0$ for $\mu$ equal to $a$ and $b$.  The conditions of Theorem \ref{thm:alternate4} hold with $M(x) = b-a$ on $[a,b]$.  The size of the constants in condition $(M)$ occur in the bounds on the derivatives of $f$ and $g$ in terms of the constants $T$, $M$, and $U$, and $U(x)$ just equals $U$.

Now, we go through the error terms of Theorems \ref{thm:main} and \ref{thm:alternate4}.  First, we have that
\begin{align*}
\mathcal{D}(\mu)+O(\Delta_1(\mu)) &= 
\begin{cases}
\begin{aligned}&g(\mu)e(f(\mu)+\left\llbracket f'(\mu)\right\rrbracket \mu)\times \\ & \quad\times \left(-\dfrac{1}{ 2\pi i \langle f'(\mu) \rangle} + \psi(\mu,\langle f'(\mu) \rangle)\right) \end{aligned}& \|f'(\mu) \|  \ge \sqrt{f''(\mu)} \\
\begin{aligned}&\dfrac{g(\mu)f^{(3)}(\mu)e(f(\mu))}{6\pi i f''(\mu)^2}-\dfrac{g'(\mu)e(f(\mu))}{2\pi i f''(\mu)}\\ &\quad+ O\left(\dfrac{UM}{T^2}\right)\end{aligned} & \| f'(x) \| =0
\end{cases}
\end{align*}

For the $\Delta_2(x)$ terms, we use that $M(x)=b-a=M$ and $f''(x) \asymp T/M^2$ to obtain
\begin{align*}
\Delta_2(\mu)&\ll \frac{UM}{T^2}\left(1+\sqrt{T}\right)\left(1+\frac{T}{M^2}\right)\\
&\qquad + \frac{U}{M\|f'(\mu)\|^2} +\frac{UT}{M^2\|f'(\mu)\|^3}\\
&\ll \frac{UM}{T^{3/2}}+ \frac{U}{M\|f'(\mu)\|^2} +\frac{UT}{M^2\|f'(\mu)\|^3}.
\end{align*}

For the $\Delta_3(x)$ terms, we again need to understand $\overline{a}$ and $\overline{b}$.  In the proof of Corollary \ref{cor:iwakowimprovement}, we used the fact that $f''$ and $M$ were generally quite large to get good bounds.  Here instead, when $f''$ is small, we shall show that $\overline{a}-a$ and $b-\overline{b}$ are going to be very large most of the time.  In particular, by the mean value theorem, for some $\xi \in [a,b]$ we have $\overline{a}-a =(f'(\overline{a})-f'(a))/f''(\xi) \gg \|f'(a)\| M^2/T$ since $f'(\overline{a})$ must be an integer.

Thus for $\Delta_3(a)$, we have the following bound.
\begin{align*}
 \Delta_3(a) &= \int_{\overline{a}}^{b} \frac{g(x)}{f''(x)(x-a)^3}\left(1+\frac{1}{f''(x)M(x)}+\frac{1}{f''(x)(x-a)}  \right) \ dx \\
&\qquad + \frac{g(\overline{a})}{f''(\overline{a})^2(\overline{a}-a)^3} +\frac{g(b)}{f''(b)^2(b-a)^3}\\
&\ll \frac{UM^2}{T(\overline{a}-a)^2}\left(1+\frac{M}{T} + \frac{M^2}{T(\overline{a}-a)}\right)+ \frac{UM^4}{T^2(\overline{a}-a)^3}  +\frac{UM}{T^2}\\
&\ll \frac{UT}{M^2\|f'(a)\|^2}\left( 1+ \frac{M}{T}+\frac{1}{\|f'(a)\|}\right) + \frac{UM}{T^2}\\
&\ll \frac{U}{M\|f'(a)\|^2}+ \frac{UT}{M^2\|f'(a)\|^3} + \frac{UM}{T^2}
\end{align*}
A similar bound holds for $\Delta_3(b)$.

Now for the last term $\Delta_4$, we have the following much simpler calculation.
\begin{align*}
 \Delta_4&= \int_a^b \frac{g(x)}{f''(x)M(x)^3} \left(1+\sqrt{f''(x)}M(x)\right) \left( 1+ \frac{1+|M'(x)|}{f''(x)M(x)}\right) \ dx\\
&\ll \int_a^b \frac{U}{TM}\left(1+\sqrt{T}\right) \left( 1+\frac{M}{T}\right) \ dx \\
&\ll \frac{U}{\sqrt{T}}\left(1+ \frac{M}{T}\right)
\end{align*}
This completes the proof of Corollary \ref{thm:karatsubakorolevimprovement}.

We conclude with two additional theorems which extend the results of Theorem \ref{thm:main} in specific directions.  The following theorem shows that if the various error terms from Theorem \ref{thm:main} are nicely bounded as $b$ tends to infinity, then many of them can be replaced by a constant plus a $o(1)$ error.

\begin{thm}\label{thm:toinfinity}

Assume all the conditions of Section \ref{section:assumptions} hold with $a$ fixed and $b$ tending to infinity. 

Let $K_b$ be the set of all $x\in[a,b]$ such that $x+M(x)>b$, where $M(x)$ is as in condition $(M)$.  Let $\partial S$ for a set $S$ have the usual meaning of the boundary of the set (the endpoints of all intervals and the isolated points).  

Suppose we have the following functions
\begin{align*}
\Delta'_3(b) &= \frac{U(b)}{f''(b)^2(b-a)^3} + \frac{U(b)}{f''(b)^2M(b)^3}\left(1+\sqrt{f''(b)}M(b)\right)\\
\Delta'_4(b) &=  \int_{K_b\cap [a,\overline{b}]} \frac{U(x)}{f''(x)(b-x)^3}\left(1+\frac{1}{f''(x)M(x)}+\frac{1}{M(x)(b-x)}\right) \ dx \\
&\qquad\qquad + \sum_{x\in \partial (K_b\cap[a,\overline{b}])} \frac{U(x)}{f''(x)^2(b-x)^3}+\Delta_2(b)\\
&\qquad \int_{K_b} \frac{U(x)}{f''(x) M(x)^3}\left( 1+\sqrt{f''(x)}M(x) \right)  \left(1+ \frac{1+|M'(x)|}{f''(x) M(x)}  \right)\ dx\\
&\qquad \qquad + \sum_{x\in \partial K_b} \frac{U(x)}{f''(x)^2M(x)^3}\left(1+\sqrt{f''(x)}M(x)\right),
\end{align*} 
and that $\Delta_2(b)$, $\Delta'_3(b)$, and $\Delta'_4(b)$ all tend to $0$ for some sequence of $b$'s tending to $\infty$ (The sequence may be different for different functions).  The function $\Delta_2(x)$ and the value of $\overline{b}$ are as defined in Theorem \ref{thm:main}.  In addition, assume that the sums and integrals in $\Delta_3(a)$, $\Delta_4(x)$, and the $\mathcal{K}$ terms all converge as $b$ tends to infinity in Theorem \ref{thm:main}.

Then 
\[
\sideset{}{^*}\sum_{a\le n \le b} g(n) e(f(n)) = \sideset{}{^*}\sum_{f'(a)\le r \le f'(b)} \frac{g(x_r)e(f(x_r)-rx_r+\frac{1}{8})}{\sqrt{f''(x_r)}}-\mathcal{D}(b)+c+\Delta
\]
as $a$ remains fixed and $b$ tends to infinity, where $c$ is some constant and
\begin{align*}
\Delta &= O(\Delta_1(b)+\Delta_2(b)+\Delta'_3(b)+ \Delta'_4(b)+\Delta_5)\\
\Delta_5 &= \int_{b}^{\infty} \frac{U(x)}{f''(x)(x-a)^3}\left(1+\frac{1}{f''(x)M(x)}+\frac{1}{f''(x)(x-a)}  \right)\\
&\qquad + \int_b^\infty \frac{U(x)}{f''(x)M(x)^3} \left(1+\sqrt{f''(x)}M(x)\right) \left( 1+ \frac{1+|M'(x)|}{f''(x)M(x)}\right) \ dx\\
&\qquad +  \mathcal{K}(J_0[b,\infty),W_0,r_0)+ \mathcal{K}(J_\pm[b,\infty),W_+,r_+) + \mathcal{K}(J_\pm[b,\infty),W_-,r_-)\\
&\qquad + \sum_{x \in J_{null}[b,\infty)} \frac{g''(x)^2}{g'(x)f''(x)^2}
\end{align*}
and all other functions are as in Theorem \ref{thm:main}.  The sets $J_0[b,\infty)$, $J_\pm[b,\infty)$, and $J_{null}[b,\infty)$ all refer to the similarly named sets from Section \ref{section:assumptions} corresponding to the set $[b,\infty)$ instead of the set $[a,b]$.

The implicit constants in the big-O terms are dependent upon the constants $C_2$, $C_{2^-}$, $C_4$, $D_0$, $D_1$, $D_2$, $\delta$ in condition $(M)$.

An analogous result holds if $b$ is fixed and $a$ tends to $-\infty$.
\end{thm}

Much of the accuracy gleaned from Theorems \ref{thm:main} and \ref{thm:alternate4} comes when $f''(x)$ is small.  If $f''(x)$ is large, instead, then $\mathcal{D}(x)$ shifts from being an explicit term to a big-O term.  However, with additional work it can be made somewhat explicit.

\begin{thm}\label{thm:refinement}
 Assume the conditions of Section \ref{section:assumptions} hold.  Suppose that $M(a), f''(a)\ge 1$ and let $C$ and $L$ be real numbers satisfying
\[
f''(a)^{-1/2}\ll C< M(a)\qquad \text{and} \qquad \sqrt{f''(a)} \ll L< f''(a)\cdot\min\{1,c\}.
\]
Also let $\epsilon=\langle a\rangle $ and $\epsilon'=\langle f'(a)\rangle $.  

Then we may replace 
\[
\mathcal{D}^\circ (a)=O\left(U(a)\left( 1+ \dfrac{1}{M(a)}+\dfrac{1}{\sqrt{f''(a)}M(a)} \right)\right) 
\]
 in the statement of Theorem \ref{thm:main} or \ref{thm:alternate4} with
\begin{equation}
\mathcal{D}^\circ (a)= \mathcal{D}_0(a)+O\left( \frac{U(a)f''(a)C^4L}{M(a)}+ \frac{U(a)L}{f''(a)C}+ \frac{U(a)f''(a)}{L^2}  +  \frac{U(a)}{f''(a)C^2}+\frac{U(a)}{M(a)}\right),
\end{equation}
where $\mathcal{D}_0(a)$ can be estimated multiple different ways.
\begin{itemize}
 \item If $\epsilon=0$, then  
\[
 \mathcal{D}_0(a) = O\left(  \frac{U(a)|\epsilon'|L}{f''(a)}  + \frac{U(a)|\epsilon'|(1+|\epsilon'|C)\log(1+f''(a))}{\sqrt{f''(a)}}  + U(a)f''(a)|\epsilon'|^3C^4 \right).
\]

 \item If $\epsilon'=0$, then
\[
\mathcal{D}_0(a) = \psi(a)g(a)e(f(a)-f'(a)a).
\]

 \item If $\epsilon'=0$ and $|\epsilon|>C$, then 
\[
 \mathcal{D}_0(a) = O\left( \frac{U(a)}{(|\epsilon|-C) \sqrt{f''(a)}}\right).
\]

\end{itemize}

Suppose that $\epsilon'=0$ and $M(a) \le f''(a)^7$.  We can then optimize the above results, so that $\mathcal{D}^\circ(a)$ equals
\[
\begin{cases}
\begin{aligned}&\psi(a)g(a)e(f(a)-f'(a)a)\\
&\qquad + O(U(a)f''(a)^{1/3}|\epsilon|^{2/3})\\
&\qquad + O\left(\dfrac{U(a)}{M(a)^{2/15}f''(a)^{1/15}}+\frac{U(a)}{M(a)}  \right)
\end{aligned} &  |\epsilon|\le f''(a)^{-1/2} \\
\begin{aligned} &O\left( \dfrac{U(a)}{f''(a)^{1/3}|\epsilon|^{2/3}}\right)\\
&\qquad +O\left(\dfrac{U(a)}{M(a)^{2/15}f''(a)^{1/15}}+\frac{U(a)}{M(a)}\right)
\end{aligned}  &         |\epsilon|\ge f''(a)^{-1/2}
\end{cases}.
\]

In all cases, the implicit constants depend on the constants associated to condition $(M)$ and on the implicit constants in the bounds on $C$ and $L$.  This theorem holds true if $a$ is everywhere replaced by $b$.
\end{thm}

\begin{rem}
The results of Theorems \ref{thm:toinfinity} and \ref{thm:refinement} can be combined to provide a more explicit $\mathcal{D}(b)$ term as $b$ tends to infinity.
\end{rem}

\section{Necessary lemmas and useful propositions}\label{sec:lemstationaryphase}

It is crucial to the van der Corput transform to have very good evaluations of integrals that contain stationary phase points.  The most powerful results known to the author are those of Redouaby and Sargos in \cite{redsar}.

\begin{lem}\label{lem:redsar1}(Lemma 2 in \cite{redsar})

Consider a function $f:[a,b]\to \mathbb{R}$ that is $C^4$, and postive real numbers $T,$ $ M,$ $C_2,$ $ C_{2^-},$ $ C_3,$ $ C_4,$ $ C_M$ such that 
$$\begin{array}{rl}
f''(z) \ge C_{2^-}^{-1}TM^{-2} & \text{ for }a\le z \le b \\
f^{(j)}(z) \le C_j TM^{-j} & \text{ for } a\le z \le b \text{ and } j=2,3,4\\
b-a \le C_M M & \\
f'(c)=0 & \text{ for some } c\text{ in }[a,b]
\end{array}$$
Moreover, let $\delta < 1$ be positive, and define $$\beta:= \frac{3\delta}{C_{2^-}C_3}.$$

If $0< c-a \le \beta M$, then
\begin{align*}
\int_a^c e(f(x)) \ dx &= \frac{e(f(c)+1/8)}{2\sqrt{f''(c)}} - \frac{f^{(3)}(c)e(f(c))}{6\pi i f''(c)^2} - \frac{e(f(a))}{2\pi i f'(a)}\\
&\qquad + O\left(\frac{M^4}{T^2(c-a)^3} \right) + O\left( \frac{M}{T^{3/2}} \right)
\end{align*}
where
$$\alpha=\frac{1}{2}f''(c) \qquad \text{and}\qquad \psi(t)=\text{sgn}(t)\left(\frac{f(c+t)-f(c)}{\alpha}\right)^{1/2}.$$

If $0< b-c \le \beta M$, then
\begin{align*}
\int_c^b e(f(x)) \ dx &= \frac{e(f(c)+1/8)}{2\sqrt{f''(c)}} + \frac{f^{(3)}(c)e(f(c))}{6\pi i f''(c)^2}+\frac{e(f(a))}{2\pi i f'(a)}\\
&\qquad + O\left(\frac{M^4}{T^2(b-c)^3} \right) + O\left( \frac{M}{T^{3/2}} \right)
\end{align*}

The implicit constant is dependent upon the $C$'s and $\delta$.
\end{lem}

\begin{rem}
The restriction $\delta<1$ in condition $(M)$ guarantees that $|\psi(t)| \asymp |t|$.  See Lemma 4 in \cite{redsar}.

Also, the above result---which may be seen as a one-sided form of Huxley's result in \cite{huxley2}---is a corollary to a more general result (Lemma 6 in \cite{redsar}), which makes the term $O(M^4/T^2|\mu-c|^3)$ explicit.

Redouaby and Sargos include the additional assumption that $T,M$ are both $\ge 1$; however, this condition is never used in the lemmas we cite, and so may be safely ignored.\footnote{This was confirmed in private correspondence with Sargos.}
\end{rem}

\begin{lem}\label{lem:redsar2}(Lemma 9 in \cite{redsar})

In addition to the hypothesis of the previous lemma, suppose that we have a function $g:[a,b]\to \mathbb{C}$ that is $C^3$, and suppose there exist positive real numbers $U,N,D_0,D_1,D_2$ such that 
$$g^{(j)}(z) \le D_j UN^{-j} \text{ for }a\le z\le b \text{ and }j=0,1,2.$$
Then 
\begin{align*}
\int_a^c g(x)e(f(x)) \ dx &= g(c)\int_a^c e(f(x)) \ dx + \frac{g(c)-g(a)}{2\pi i f'(a)} e(f(a))\\
&\qquad + \frac{g'(c)}{2\pi i f''(c)}e(f(c)) + O\left( \frac{UM^2}{NT^{3/2}} \left( 1+\frac{M}{N}\right)\right)\\ 
\int_c^b g(x)e(f(x)) \ dx &= g(c)\int_c^b e(f(x)) \ dx + \frac{g(b)-g(c)}{2\pi i f'(b)} e(f(b))\\
&\qquad - \frac{g'(c)}{2\pi i f''(c)}e(f(c)) + O\left( \frac{UM^2}{NT^{3/2}} \left( 1+\frac{M}{N}\right)\right)\\ 
\end{align*}

This result still holds if we weaken the assumptions of the previous lemma so that $f$ is only $C^3$ instead of $C^4$, and the condition on $f^{(4)}(x)$ may likewise be ignored.
\end{lem}

\begin{prop}\label{prop:main1}
Suppose the hypotheses of Lemmas \ref{lem:redsar1} and \ref{lem:redsar2} hold and suppose $a \le c \le b$.  Then 
\begin{align*}
&\int_a^c g(x)e(f(x)) \ dx\\
 &\qquad=g(c)\frac{e(f(c)+1/8)}{2\sqrt{f''(c)}} - \frac{g(c)f^{(3)}(c)e(f(c))}{6\pi i f''(c)^2} - \frac{g(a)e(f(a))}{2\pi i f'(a)}\\
&\qquad \qquad+ \frac{g'(c)}{2\pi i f''(c)}e(f(c)) + \mathcal{E}(a),
\end{align*}
and
\begin{align*}
&\int_c^b g(x)e(f(x)) \ dx\\ 
&\qquad=g(c)\frac{e(f(c)+1/8)}{2\sqrt{f''(c)}} + \frac{g(c)f^{(3)}(c)e(f(c))}{6\pi i f''(c)^2} + \frac{g(b)e(f(b))}{2\pi i f'(b)}\\
&\qquad\qquad - \frac{g'(c)}{2\pi i f''(c)}e(f(c)) + \mathcal{E}(b),
\end{align*}
where
\[
\mathcal{E}(y) = O\left(\frac{UM^4}{T^2|y-c|^3}\right)+O\left(\frac{UM}{T^{3/2}}\left(1+\frac{M}{N}\right)^2\right).
\]
\end{prop}

\begin{proof} We just apply Lemma \ref{lem:redsar2} followed by Lemma \ref{lem:redsar1}.
\end{proof}

The advantage of our condition $(M)$ is now the following.

\begin{prop}\label{prop:main2}
Suppose condition $(M)$ holds, that $f$ has a stationary phase point at $x\in J$, and that $x\in [\alpha,\beta]\subset [x-M(x),x+M(x)]\cap J$.  Then the conditions of Lemmas \ref{lem:redsar1} and \ref{lem:redsar2} hold with 
\begin{align*}
C_3&=1, &  M&=\frac{1}{\eta} M(x), & N&=M(x),\\ C_M&=2\eta, & U&=U(x), & T&=\eta^2 f''(x) M(x)^2, 
\end{align*}
and with $C_{2^-}$, $C_2$, $C_4$, $D_0$, $D_1$, $D_2$ and $\delta$ equal to the similarly named constants in condition $(M)$.

Morever, with these definitions, the error terms become 
\[
\mathcal{E}(\alpha)+\mathcal{E}(\beta) =O\left( \frac{U(x)}{f''(x)^2|x-\alpha|^3}\right)+ O\left(\frac{U(x)}{f''(x)^2|x-\beta|^3}\right) +  O\left(\frac{U(x)}{f''(x)^{3/2}M(x)^2} \right),
\]
where if $x=\alpha$ (resp., $x=\beta$) then the first (resp., second) big-O term disappears.
\end{prop}

However, we will also need good estimates on oscillatory integrals without a stationary phase point.  The first and second derivative tests below are well-known and often accompany any discussion of the van der Corput transform.  Following them we give some more exact estimates that make use of the integration by parts technique we described in Section \ref{section:heuristics}. 

\begin{lem}\label{lem:1stderivtest} (First derivative test---Lemma 5.1.2 in \cite{huxley1})

Let $f(x)$ be real and differentiable on the open interval $(\alpha, \beta)$ with $f'(x)$ monotone and $f'(x) \ge \kappa > 0$ on $(\alpha, \beta)$.  Let $g(x)$ be real, and let $V$ be the total variation of $g(x)$ on the closed interval $[\alpha,\beta]$ plus the maximum modulus of $g(x)$ on $[\alpha,\beta]$.  Then
\[
\left| \int_\alpha^\beta g(x)e(f(x)) \ dx \right| \le \frac{V}{\pi \kappa}
\]
\end{lem}

\begin{lem}\label{lem:2ndderivtest} (Second derivative test---Lemma 5.1.3 in \cite{huxley1})

Let $f(x)$ be real and twice differentiable on the open interval $(\alpha, \beta)$ with $f''(x) \ge \lambda >0$ on $(\alpha, \beta)$.  Let $g(x)$ be real, and let $V$ be the total variation of $g(x)$ on the closed interval $[\alpha,\beta]$ plus the maximum modulus of $g(x)$ on $[\alpha,\beta]$.  Then 
$$\left| \int_\alpha^\beta g(x)e(f(x))\ dx \right| \le \frac{4V}{\sqrt{\pi \lambda}}.$$

\end{lem}

\begin{prop}\label{prop:2ndderivtest} (Condition $(M)$ version of first and second derivative tests)

Suppose condition $(M)$ holds.  If for some $x \in J$, we have $[\alpha,\beta]\subset I_x:=J \cap [x-M(x),x+M(x)]$, then 
\[
\left| \int_\alpha^\beta g(y)e(f(y)) \ dx \right| \ll \frac{U(x)}{\min_{x\in[\alpha,\beta]} |f'(x)|}
\]
and
\[
\left| \int_\alpha^\beta g(y)e(f(y))\ dy \right| \ll \frac{U(x)}{\sqrt{f''(x)}}
\]
with uniform implicit constant. 
\end{prop}

\begin{proof}
We can bound $V$ in Lemmas \ref{lem:1stderivtest} and \ref{lem:2ndderivtest} by $$g(\alpha)+g(\beta)+\int_\alpha^\beta |g'(y)| \ dy \ll  U(x) + 2 M(x)\cdot  \frac{U(x)}{M(x)}\ll U(x)$$ using condition $(M)$.  Also by condition $(M)$, we have that the minimum of $f''(z)$ on $I_x$ is at worst $\gg f''(x)$.
\end{proof}

In some cases, the first derivative test needs to be made explicit, for which we have the following result.

\begin{prop}\label{prop:intbyparts}
Let $f\in C^3( [\alpha,\beta])$, and $g\in C^2([\alpha,\beta])$, and define $h_r(x)$ by \[
h_r(x) := \frac{(f'(x)-r)g'(x) - g(x)f''(x)}{(f'(x)-r)^3}.
\]

Suppose that $f'(x)\neq r$ on an interval $[\alpha,\beta]$, and let $$K_r(\alpha,\beta):=\sum \left| h_r(x) \right| $$ where the sum ranges over all $x\in[\alpha,\beta]$ where $h_r'(x)=0$.  Then we have
\begin{align*}
\int_\alpha^\beta  g(x)e(f(x)-rx) \ dx &= \left[\frac{g(x)}{2\pi i (f'(x)-r)}e(f(x)-rx)\right]_\alpha^\beta \\ &\qquad +O(K_r(\alpha,\beta)) +O( |h_r(\alpha)|)+O( |h_r(\beta)|)
\end{align*}
\end{prop}

\begin{proof}
We begin by applying integration by parts twice to our original integral.
\begin{align}
&\int_\alpha^\beta  g(x)e(f(x)-rx) \ dx \notag \\
&\qquad= \left[\frac{g(x)}{2\pi i (f'(x)-r)}e(f(x)-rx)\right]_\alpha^\beta \\ &\qquad-\int_\alpha^\beta \frac{1}{2\pi i} \frac{(f'(x)-r)g'(x)-g(x)f''(x)}{(f'(x)-r)^2} e(f(x)-rx) \ dx \notag\\
&\qquad=  \left[\frac{g(x)}{2\pi i (f'(x)-r)}e(f(x)-rx)\right]_\alpha^\beta \notag \\
&\qquad\qquad -\left[ \frac{1}{(2\pi i)^2}h_r(x) e(f(x)-rx)\right]_\alpha^\beta \label{eq:propintbyparts1}\\
&\qquad\qquad + O\left( \int_\alpha^\beta \left| \frac{d}{dx}  h_r(x)  \right| \ dx  \right) \label{eq:propintbyparts2}
\end{align}

The two terms on line \eqref{eq:propintbyparts1} are bounded trivially by $O(h_r(a))+O(h_r(b))$.

The remaining integral on line \eqref{eq:propintbyparts2} is bounded by the total variation of $h_r(x)$, which is at most the modulus of $h_r(x)$ at every critical point added to $|h_r(\alpha)|$ and $|h_r(\beta)|$.  $K_r(\alpha,\beta)$ is, by definition, the sum of the moduli $h_r(x)$ at every critical point (since the derivative---by the assumption that $f'(x)\neq r$ on $[\alpha,\beta]$---always exists), so the proof is complete.
\end{proof}

By making use of the following proposition, we can sometimes avoid the $K_r$ term in Proposition \ref{prop:intbyparts}.  An alternate form of the resulting error terms can be found in Lemma 5.5.5 of \cite{huxley1}.

\begin{prop}\label{prop:f'bound}
Suppose condition $(M)$ holds.  If $\epsilon=\pm 1$, then for all $x$ such that $x,x+\epsilon \cdot M(x) \in J$ and all $r$ such that $$f'(x)-r=0 \qquad \text{or} \qquad \text{sgn}(f'(x)-r)=\text{sgn}(k),$$ we have $$|f'(x+\epsilon \cdot M(x))-r|\gg M(x)f''(x).$$
\end{prop}

\begin{proof}
By Taylor's remainder theorem, we can write
\begin{align}
 f'(x+\epsilon \cdot M(x))-r &= (f'(x)-r)+\epsilon \cdot M(x)f''(x)\notag \\
&\qquad +O\left(\frac{1}{2} M(x)^2 \max_{y\in [x,x+kM(x)]} f^{(3)}(y) \right) \label{eq:f'bound}
\end{align} with implicit constant 1.

By condition $(M)$, we have  
\[
\max_{y\in [x,x+\epsilon \cdot M(x)]} |f^{(3)}(y)| \le \frac{\eta f''(x)}{M(x)}.
\]
Therefore the big-O term in \eqref{eq:f'bound} is at most $\eta M(x)f''(x)/2$.  Since $\eta<2$ and since $f'(x)-r$ is either $0$ or shares the same sign as $\epsilon$, this gives the result.
\end{proof}

\begin{prop}\label{prop:intbypartsvar}
 Suppose $f$ and $g$ satisfy condition $(M)$ for some function $M(x)$ and its associated constants.  Let $\epsilon =\pm 1$.  Assume that $x_r$ is not in the interval from $x$ to $x+\epsilon \cdot M(x)$, and that $x+\epsilon\cdot M(x)$ is farther from $x_r$ than $x$ is.  Then we have that 
\begin{align*}
\int_{x+\epsilon\cdot M(x)}^x g(y) e(f(y)-ry) \ dy &=\left. \frac{g(y)}{2\pi i (f'(y)-r)} e(f(y)-ry)\right]_{y=x+\epsilon \cdot M(x)}^x \\
&\qquad + O\left( \frac{U(x) f''(x)}{|f'(x)-r|^3} \right)+ O\left( \frac{U(x)}{M(x)(f'(x)-r)^2}\right)
\end{align*}

We also have the following bound:
\[
\left. \frac{g(y)}{2\pi i (f'(y)-r)} e(f(y)-ry)\right]_{y=x+\epsilon \cdot M(x)}=O\left( \frac{U(x)}{M(x)f''(x)}  \right).
\]
\end{prop}

\begin{proof}
 As in Proposition \ref{prop:intbyparts}, we apply integration by parts twice to our starting integral
\begin{align}
&\int_{x+\epsilon \cdot M(x)}^x  g(y)e(f(y)-ry) \ dy\notag\\
&\qquad=  \left[\frac{g(y)}{2\pi i (f'(y)-r)}e(f(y)-ry)\right]_{y=x+\epsilon\cdot M(x)}^x\notag\\
&\qquad\qquad  +O(h_r(x))+O(h_r(x+\epsilon\cdot M(x)))\notag \\
&\qquad\qquad  + O\left( \int_{x+\epsilon\cdot M(x)}^x \left|\frac{d}{dy}  \frac{(f'(y)-r)g'(y)-g(y)f''(y)}{(2\pi i)^2(f'(y)-r)^3}  \right| \ dy \right) \label{eq:intbypartsvar}
\end{align}

The relation 
\[
\left.\frac{g(y)}{2\pi i (f'(y)-r)}e(f(y)-ry)\right]_{y=x+\epsilon \cdot M(x)} = O\left( \frac{g(x)}{M(x)f''(x)}  \right)
\]
holds by Proposition \ref{prop:f'bound}.

Using condition $(M)$, we have the following additional bounds.
\begin{align*}
h_r(x) &= O\left( \frac{g(x)}{M(x)(f'(x)-r)^2}\right) + O\left( \frac{g(x) f''(x)}{|f'(x)-r|^3} \right)\\
 h_r(x+\epsilon\cdot M(x)) &= O\left(\frac{g'(x+\epsilon\cdot M(x))}{(f'(x+\epsilon\cdot M(x))-r)^2} \right)\\
&\qquad + O\left( \frac{g(x+\epsilon\cdot M(x))f''(x+\epsilon\cdot M(x))}{(f'(x+\epsilon\cdot M(x))-r)^3} \right) \\
&=  O\left( \frac{g(x)}{M(x)(f'(x)-r)^2}\right) + O\left( \frac{g(x) f''(x)}{|f'(x)-r|^3} \right)
\end{align*}
The last equality holds since $|f'(x+\epsilon \cdot M(x))-r|$ is bigger than $|f'(x)-r|$ by our assumption about the relative placement of $x$, $x_r$, and $x+\epsilon \cdot M(x)$: for example, if $x_r < x < x+M(x)$, then $r < f'(x) < f'(x+M(x))$.

Finally, to estimate the integral in \eqref{eq:intbypartsvar} first notice that 
\begin{align*}
\left|\frac{d}{dy}  \frac{(f'(y)-r)g'(y)-g(y)f''(y)}{(2\pi i)^2(f'(y)-r)^3}  \right|&=O\left( \frac{g''(y)}{(f'(y)-r)^2} \right)+O\left( \frac{g'(y)f''(y)}{(f'(y)-r)^3} \right)\\
&\qquad +O\left( \frac{g(y)f^{(3)}(y)}{(f'(y)-r)^3} \right)+O\left( \frac{g(y)f''(y)^2}{(f'(y)-r)^4} \right).
\end{align*}
On the interval between $x$ and $x+\epsilon\cdot M(x)$, the maximum of $g''(y)/(f'(y)-r)^2$ is
\[
\ll \frac{g(x)}{M(x)^2(f'(x)-r)^2},
\]
and since we are integrating over an interval of length $M(x)$, the contribution of this term is at most
\[
 \frac{g(x)}{M(x)(f'(x)-r)^2}.
\]
Similarly, one can show that 
\[
\int_{x+\epsilon\cdot M(x)}^x O\left( \frac{g'(y)f''(y)}{|f'(y)-r|^3} \right)+O\left( \frac{g(y)f^{(3)}(y)}{|f'(y)-r|^3} \right) \ dy= O\left(\frac{g(x)f''(x)}{|f'(x)-r|^3}\right).
\]

We estimate the remaining integral by using $f''(y)=O(f''(x))$:
\begin{align*}
 \int_{x+\epsilon\cdot M(x)}^x \frac{g(y)f''(y)^2}{(f'(y)-r)^4} \ dy &=O\left( g(x)f''(x) \int_{x+\epsilon \cdot M(x)}^x \frac{f''(y)}{(f'(y)-r)^4} \ dy\right)\\
&= O\left( g(x) f''(x) \left[ \frac{1}{(f'(y)-r)^3} \right]_{x+\epsilon \cdot M(x)}^x \right) \\
&= O\left(\frac{g(x) f''(x)}{|f'(x)-r|^3}\right). \qedhere
\end{align*}
\end{proof}

\begin{rem}
 If $g(x)$ is a constant, then we may remove the term
\[
 O\left( \frac{U(x)}{M(x)(f'(x)-r)^2}\right)
\]
from the statement of Proposition \ref{prop:intbypartsvar}.  This would remove the term $U(x)/M(x)\|f'(x)\|^2$ from $\Delta_2(x)$ in Theorem \ref{thm:main}.  However, very often one finds a term of this same size occuring in $\Delta_3(x)$ regardless.
\end{rem}

The primary technique we shall use to evaluate sums will be a variant the Euler-Maclaurin summation formula; however, the Euler-Maclaurin formula itself sums over all integers in an interval $[a,b]$, while we will often want to sum over all values $x$ in some interval $[\alpha, \beta]$ for which $F(x)$ is an integer (for some function $F$).  

\begin{lem}\label{lem:emsum} (Euler-Maclaurin summation, first derivative version---page 10 in \cite{karatsuba})

Suppose $f$ is a differentiable function on $[\alpha,\beta]$ then
\begin{align*}\sum_{\alpha < n \le \beta} f(n)&=\int_\alpha^\beta f(x) \ dx +\int_\alpha^\beta s(x) f'(x) \ dx\\ &\qquad -s(\beta) f(\beta) + s(\alpha)f(\alpha) .\end{align*}
\end{lem}

\begin{prop}\label{prop:inversesum}
Suppose $G$ is a real-valued, differentiable function on $[a,b]$.  Suppose $F$ is a differentiable function such that $F([a,b])=[\alpha,\beta]$, then 
\begin{align*}\sum_{\alpha\le n\le\beta} G(F^{-1}(n)) &= \int_a^b G(y)\cdot |F'(y)|\ dy +O\left(  \int_a^b |G'(y)|\ dy \right)\\ 
& \qquad +O\left(\left|s(a)\cdot G(a) \right|\right)+O\left(\left|s(b)\cdot G(b) \right|\right)\\
&\qquad  +O\left( \sum_{x\text{ is a local extremum of }F\text{ on }[a,b]} \left|s(x)\cdot G(x) \right| \right)
\end{align*}
where $G(F^{-1}(n))$ is a sum over all $G(x)$ for $x\in F^{-1}(n)$.
\end{prop}

\begin{proof}
We apply the first derivative version of Euler-Maclaurin summation to each interval where $F$ is monotonic (and hence, $F^{-1}$ is 1-1).  Then we apply a change of variables with $y=F^{-1}(x)$.  The sum $$O\left( \sum_{x\text{ is a local extremum of }F\text{ on }[a,b]} \left|s(x)\cdot G(x) \right| \right)$$ bounds the contributions of $s(\alpha)f(\alpha)$ and $-s(\beta) f(\beta)$ arising from the endpoints of the intervals of monotonicity.
\end{proof}

\begin{lem}\label{lem:poissonsummation} (Lemma 5.4.2 in \cite{huxley1}\footnote{Huxley assumes $a$ and $b$ are integers, so we provide the proof when $a$, $b$ are real.}---Poisson summation)

Let $f\in C^2([a,b])$ for real numbers $a<b$, then
$$\sideset{}{^*}\sum_{a\le n \le b} f(n) = \lim_{R\to \infty} \sum_{r=-R}^R \int_a^b f(x)e(rx) \ dx.$$
\end{lem}

\begin{proof}
We rewrite the Euler-Maclaurin summation formula as
\[
\sideset{}{^*}\sum_{\alpha \le n \le \beta} f(n)=\int_\alpha^\beta f(x) \ dx -\psi(\beta) f(\beta) + \psi(\alpha)f(\alpha) +\int_\alpha^\beta \psi (x) f'(x) \ dx.
\]
But we also have 
\begin{align*}
\int_\alpha^\beta \psi (x) f'(x) \ dx &= -\frac{1}{\pi} \lim_{R\to \infty} \int_\alpha^\beta \sum_{r=1}^R \frac{sin(2 \pi rx)}{r}f'(x) \ dx \\
&= \psi(\beta)f(\alpha)-\psi(\alpha)f(\alpha)+\lim_{R\to \infty}\sum_{r=1}^R \int_\alpha^\beta 2\cos(2\pi r x) f(x) \ dx\\
&= \psi(\beta)f(\alpha)-\psi(\alpha)f(\alpha)+\lim_{R\to\infty}\sum_{r=1}^R \int_\alpha^\beta (e(rx)+e(-rx)) f(x) \ dx,
\end{align*}
by integration by parts, which completes the proof.
\end{proof}

We end this section with two more specific propositions.  The first will give us good bounds on the partial sums and tails of $\psi(x,\epsilon)$, our modified sawtooth function, as well as guarantee its convergence.  The second will show that the integrals in Theorem \ref{thm:main} involving the $W$ and $r$ functions will converge.

\begin{prop}\label{prop:fouriertail} Suppose $\beta> \alpha\ge 1$, $\epsilon \in [-1/2,1/2]$, and $x\in \mathbb{R}$.  Then
 \[
  \sum_{\alpha\le |m+\epsilon|\le \beta} \frac{e(mx)}{m+\epsilon}=O\left(\min\left\{1,\dfrac{1}{\alpha\| x\|^*}  \right\}\right) 
 \]
where the implicit constant is uniform.  Also, the sum is convergent as $\beta$ tends to $\infty$ and the same bounds hold in this case.
\end{prop}

\begin{proof}
Without loss of generality, we may replace that $x$ with $\langle x\rangle $ and so may assume that $x\in[-1/2,1/2)$.  We start with the assumption that $\beta$ is finite.

Next, we remove all appearances of the $\epsilon$ from the sum.  We use $C_{\alpha,x,\epsilon}$ and $C'_{\alpha,x,\epsilon}$ to denote constants which only depend on the variables $\alpha$, $x$, and $\epsilon$.
\begin{align*}
 \sum_{\alpha\le |m+\epsilon|\le \beta} \frac{e(mx)}{m+\epsilon} &=  \sum_{\alpha\le |m|\le \beta} \frac{e(mx)}{m+\epsilon} +C_{\alpha,x,\epsilon}+O\left(\frac{1}{\beta}\right)\\
&= \sum_{\alpha\le |m|\le \beta} \frac{e(mx)}{m}-\sum_{\alpha\le |m|\le \beta} e(mx)\left(\frac{\epsilon}{m(m+\epsilon)}\right)\\
&\qquad +C_{\alpha,x,\epsilon}+O\left(\frac{1}{\beta}\right)\\
&= \sum_{\alpha\le |m|\le \beta} \frac{e(mx)}{m}-\sum_{\alpha\le |m|} e(mx)\left(\frac{\epsilon}{m(m+\epsilon)}\right)+O\left(\frac{1}{\beta}\right)\\
&\qquad   +C_{\alpha,x,\epsilon}+O\left(\frac{1}{\beta}\right)\\
&= \sum_{\alpha\le |m|\le \beta} \frac{e(mx)}{m}+C'_{\alpha,x,\epsilon}  +C_{\alpha,x,\epsilon}+O\left(\frac{1}{\beta}\right)
\end{align*}
Here, $C_{\alpha,x,\epsilon}$ contains the terms where $|m|<\alpha\le |m+\epsilon|$ or $|m+\epsilon|<\alpha\le |m|$.  In particular, it is no larger than $2/(\alpha-1/2)=O(1/\alpha)$.  Similarly, $C'_{\alpha,x,\epsilon}$ is dominated by the sum of terms $1/m^2$, and hence is also $O(1/\alpha)$.

Now we pause a moment to show that we may let $\beta=\infty$ with no problems of convergence.  The sum
\[
 \sum_{1\le |m|} \frac{e(mx)}{m}
\]
is, up to a constant multiplier, the Fourier series for the sawtooth function $\psi(x)$ and converges for all $x$.  This implies that 
\[
 \lim_{\beta\to \infty} \sum_{\alpha\le |m|\le \beta} \frac{e(mx)}{m} 
\]
converges for all $\alpha$ and $x$.  Therefore, since
\[
 \lim_{\beta\to \infty} \left(\sum_{\alpha\le |m+\epsilon|\le \beta} \frac{e(mx)}{m+\epsilon}-\sum_{\alpha\le |m|\le \beta} \frac{e(mx)}{m} \right)=C'_{\alpha,x,\epsilon}  +C_{\alpha,x,\epsilon}
\]
we have that 
\[
 \sum_{\alpha\le |m|} \frac{e(mx)}{m+\epsilon}
\]
converges for all $\alpha$, $x$, and $\epsilon$.  

If $x=0$ then 
\[
 \sum_{\alpha\le |m|\le \beta} \frac{e(mx)}{m}=0,
\]
and hence 
\[\sum_{\alpha\le |m+\epsilon|\le \beta} \frac{e(mx)}{m+\epsilon}=O\left(\frac{1}{\alpha}\right).\]
So we may assume for the remainder of the proof that $x\neq 0$.

Now we remove the absolute value in the condition on the sum and apply summation by parts.
\begin{align*}
 \sum_{\alpha\le |m|\le \beta} \frac{e(mx)}{m}&=2 i \sum_{\alpha\le m \le \beta} \frac{\sin(2\pi mx)}{m} \\
&= 2i \sum_{\alpha\le m\le \beta} \left(\frac{1}{m}-\frac{1}{m+1}\right)\left(\sum_{\alpha\le k \le m} \sin(2\pi k x)\right)\\
&= 2i \sum_{\alpha\le m \le \beta} \frac{1}{m^2+m} \cdot \frac{\sin(\pi(m+\lceil\alpha\rceil)x)\sin(\pi(m-\lceil\alpha\rceil)x)}{\sin(\pi x)}
\end{align*}
Since we assumed $x\in[-1/2,1/2]$, we have that $\sin(\pi x)\asymp x$.  For the numerator, we use that $|\sin(y)| \le \min\{1,|y|\}$.  This gives the following estimate, which completes the proof.
\begin{align*}
 \sum_{\alpha\le |m|\le \beta} \frac{e(mx)}{m} &=O\left(\sum_{\alpha\le m\le \beta} \frac{1}{m^2} \min\left\{ \frac{1}{\|x\|},m^2 \|x\| \right\}\right) \\
&= O\left(\sum_{\alpha \le m \le 1/\|x\|} \|x\|\right)+O\left(\sum_{m\ge \max\{\alpha,1/\|x\|\}} \frac{1}{m^2\|x\|}\right)\\
&=\begin{cases}
   O(1) & \text{if }\alpha \|x\|\le 1\\
   O\left(\dfrac{1}{\alpha \|x\|}\right) & \text{otherwise}
  \end{cases}
\end{align*}
\end{proof}

\begin{prop}\label{prop:integrability}
Suppose the conditions of Section \ref{section:assumptions} hold, then the function $|W_0'(x)|+|W_0(x)\cdot r_0(x)|$ is integrable on $J_0$ and the function $|W_+'(x)|+|W_-'(x)|+|W_+(x)\cdot r_+'(x)|+|W_-(x)\cdot r_-'(x)|$ is integrable on $J_\pm$.
\end{prop}

\begin{proof}
Recall that
\begin{align*}
H&=gf^{(3)}+3g'f''\\
G&=12gg''(f'')^2\\ 
W_\pm &=\frac{(2g'')^2g'}{(H\pm \sqrt{H^2-G})^2} -\frac{(2g'')^3(f''g)}{(H\pm \sqrt{H^2-G})^3} \\
W_0 &= -\frac{H^2 f^{(3)}}{27g(f'')^5}\\
r_\pm &=f'-\frac{H\pm \sqrt{H^2-G}}{2g''}\\
r_0 &=f'-\frac{3g(f'')^2}{H},
\end{align*}
The functions $H$ and $G$ are continuous and bounded on $J$.  The functions $W_0$ and $r_0$ are continuous inside $J_0$, and the functions $W_\pm$ and $r_\pm$ are continuous inside $J_\pm$.  The only possible barrier to integrability is the presence of a zero in the denominator of the function at an endpoint of an interval.

The only terms in the denominator of $W_0'$ are $g$ and $f''$, which are both bounded away from $0$ on $J_0$.  The only terms in the denominator of $W_0$ are $g$ and $f''$ again, but $r_0'$ has a factor of $H^2$ in the denominator; however, in $|W_0(x)\cdot r_0(x)|$, the $H^2$ in the denominator of $r_0'$ is cancelled by the factor of $H^2$ in the numerator of $W_0(x)$, so this too has no zeroes in the denominator at the endpoints of $J_0$.

The case when we are considering an interval $(a',b')$ where $H^2(x)-G(x)=0$ but $H(x)$ does not equal nor tend to $0$ is analogous.  So we will assume we are in the case where $H^2(x)-G(x)$ does not equal nor tend to $0$ on the interval in question for the remainder of this proof.

We have
\[
r_\pm'(x)=f''-\left(  \frac{H\pm\sqrt{H^2-G}}{2g''}\right)'
\]
and
\begin{align*}
W_\pm'(x)&=\left(  \frac{H\pm\sqrt{H^2-G}}{2g''}\right)^{-3}(f^{(3)}g+f''g')-\left(  \frac{H\pm\sqrt{H^2-G}}{2g''}\right)^{-2}g''\\
&\qquad -3\left(  \frac{H\pm\sqrt{H^2-G}}{2g''}\right)^{-4}\left(  \frac{H\pm\sqrt{H^2-G}}{2g''}\right)'(f'' g)\\
&\qquad +2\left(  \frac{H\pm\sqrt{H^2-G}}{2g''}\right)^{-3}\left(  \frac{H\pm\sqrt{H^2-G}}{2g''}\right)'g'.
\end{align*}
In addition,
\begin{align}
\left(  \frac{H\pm\sqrt{H^2-G}}{2g''}\right)'&=\frac{H\pm\sqrt{H^2-G}}{2g''^2}g^{(3)} \label{eq:HoverG1}\\
&\qquad -\frac{1}{2g''\sqrt{H^2-G}}\left(H'\left(\sqrt{H^2-G} \pm H\right) \mp\frac{G'}{2}\right) \notag
\end{align}
So, the only terms in the denominator of $r_\pm'$ are $g''$ and $\sqrt{H^2-G}$, and the only terms in the denominator of $W_\pm'$ are $H\pm\sqrt{H^2-G}$ and $\sqrt{H^2-G}$ as all the $g''$ terms cancel.

The technique required will change slightly depending on whether we consider the $+$ terms or the $-$ terms.  As $G$ tends to $0$ (which is equivalent to $g''$ tending to $0$), $H\pm \sqrt{H^2-G}$ will tend to either $2H$ or $0$.  By the assumptions of Section \ref{section:assumptions}, if $G$ tends to $0$ at the endpoint of an interval in $J_\pm$, then $H$ cannot tend to $0$ at the same point.

If the sign is chosen so $H\pm\sqrt{H^2-G}$ tends to $2H$, then the terms $H\pm\sqrt{H^2-G}$ and $\sqrt{H^2-G}$ in the denominator of $W_\pm'(x)$ do not vanish, and the two copies of $g''$ in the denominator of $r_\pm'(x)$ are cancelled by the two copies of $g''$ in the numerator of $W_\pm(x)$.  Therefore, in this case, there are no zeroes in the denominator.

For the remaining case, when $H\pm\sqrt{H^2-G}\to 0$, we have that
\begin{align}
\frac{H\pm\sqrt{H^2-G}}{2g''} &= \frac{|H|}{2g''}\left(\frac{H}{|H|}\pm\sqrt{1-\frac{G}{H^2}}\right)\notag\\
&= \frac{|H|}{2g''}\left(\mp \frac{G}{2H^2} +O\left(\frac{G^2}{H^4}\right) \right)  \label{eq:HoverG}\\
&= \frac{3g(f'')^2}{H}(1+O(g'')),\notag
\end{align}
where the implcit constant depends on the size of $H$ near this point.  In particular, since $g$ and $f''$ do not approach $0$ at the endpoints of $J_\pm$, equation \eqref{eq:HoverG} implies that a copy of $H\pm\sqrt{H^2-G}$ tending to $0$ in the denominator of a function can be cancelled by a copy of $g''$ in the numerator to prevent the presence of a zero in the denominator.  In particular, $W_\pm(x)$ will not have a zero in the denominator.

By \eqref{eq:HoverG}, the only term in $|W_\pm'(x)|+|W_\pm(x) \cdot r_\pm'(x)|$ that could produce a zero in the denominator is the derivative
\begin{equation}
\left(  \frac{H\pm\sqrt{H^2-G}}{2g''}\right)'.  \label{eq:HoverG'}
\end{equation}
Therefore, it suffices to show that \eqref{eq:HoverG'} has no zeroes in the denominator at an endpoint of $J_\pm$.

By applying \eqref{eq:HoverG} to line \eqref{eq:HoverG1} and noting that in this case $\pm(H^2-G)^{-1/2}=-H^{-1}(1+O(G))$, we have
\begin{align*}
\left(  \frac{H\pm\sqrt{H^2-G}}{2g''}\right)'&=\frac{3g(f'')^2g^{(3)}}{g''}\left(\frac{1}{H}\pm\frac{1}{\sqrt{H^2-G}}\right)\mp \frac{3g(f'')^2H'}{2\sqrt{H^2-G}}\\
&\qquad\pm\frac{12g'(f'')^2+24gf''f^{(3)}}{4\sqrt{H^2-G}}+O(1)\\
&=\frac{3g(f'')^2g^{(3)}}{g''}\cdot O\left(\frac{G}{H}\right)\mp \frac{3g(f'')^2H'}{2\sqrt{H^2-G}}\\
&\qquad\pm\frac{12g'(f'')^2+24gf''f^{(3)}}{4\sqrt{H^2-G}}+O(1)\\
&= \mp \frac{3g(f'')^2H'}{2\sqrt{H^2-G}}\pm\frac{12g'(f'')^2+24gf''f^{(3)}}{4\sqrt{H^2-G}}+O(1).
\end{align*}
Since $\sqrt{H^2-G}$ does not tend to $0$ on $J_\pm$, the derivative \eqref{eq:HoverG'} has no zeroes in the denominator.
\end{proof}

\section{Proof of Theorem \ref{thm:main}}
\subsection{The initial step}\label{section:prooffirststep}

We begin by applying Poisson summation (Lemma \ref{lem:poissonsummation}).  This gives $$\sideset{}{^*}\sum_{a\le n \le b} g(n)e(f(n)) = \lim_{R\to \infty} \sum_{r=-R}^R \int_a^b g(x)e(f(x)-rx) \ dx.$$

We wish to alter the end-points of some of these integrals.  Define $a_r$, the new left endpoint, by 
\begin{equation}\label{eq:a_rdef}
 a_r:= \begin{cases} 
a+M(a) & \text{if }0< f'(a)-r <f''(a)\\
a-M(a) &  \text{if } 0 > f'(a)-r > -f''(a)\\
a & \text{otherwise} \\
 \end{cases}
\end{equation}
and $b_r$, the new right endpoint, by
\[
b_r:= \begin{cases} 
b+M(b) & \text{if }0< f'(b)-r<f''(b)\\
b-M(b) &  \text{if } 0 > f'(b)-r > -f''(b)\\
b & \text{otherwise} \\
 \end{cases}.
\]
Note that $f'(a)-r$ is positive, if and only if $x_r$ lies to the left of $a$ and vice-versa.  The transformation from $[a,b]$ to $[a_r,b_r]$ has the effect that if $x_r$ is close to---but not equal to---an endpoint $a$ or $b$, then we shift that endpoint away from $x_r$ by a distance $M(a)$ or $M(b)$, respectively.

In the statement of Theorem \ref{thm:main}, $m_x$ is defined by
\[
 m_x:=\left| \mathbb{Z} \cap (f'(x)-f''(x),f'(x)+f''(x))\setminus \{f'(x)\} \right|.
\]
The values $m_a$ and $m_b$ thus count the number of $r$ such that $a_r\neq a$ and $b_r \neq b$, respectively.

After altering the end-points, we have
\begin{align}
 \sideset{}{^*}\sum_{a\le n \le b} g(n)e(f(n)) &=  \lim_{R\to \infty} \sum_{r=-R}^R \int_{a_r}^{b_r} g(x)e(f(x)-rx) \ dx \notag \\
&\qquad - \lim_{R\to \infty} \sum_{r=-R}^R \left( \int_{a_r}^a + \int_{b_r}^b \right) g(x)e(f(x)-rx) \ dx \label{eq:shiftedintegral}
\end{align}
Since there are only finitely many values of $r$ for which $a_r\neq a$ and $b_r\neq b$, the sum on line \eqref{eq:shiftedintegral} is finite and so we may omit the limit and sum from $-\infty$ to $\infty$.

First, consider those integrals in the sum on line \eqref{eq:shiftedintegral} arising from $0 < |f'(a)-r| < 1$.  If $\langle f'(a)\rangle=0$ or if $m_a=0$, there are no such integrals, and so the total contribution of this (null) set of integrals is $0$.  Otherwise, there are at most $2$ such integrals, each of which by Proposition \ref{prop:2ndderivtest} (our variant of the first and second derivative tests) is bounded by 
\begin{equation}\label{eq:smallrbound}
O\left(\min\left\{ \frac{U(a)}{\sqrt{f''(a)}},\frac{U(a)}{\| f'(a)\|} \right\}\right).
\end{equation}
Thus the contribution of these integrals is $O(\Delta_1(a))$.

There can only be additional integrals in line \eqref{eq:shiftedintegral} if $f''(a)\ge 1$, so we will assume as such when bounding them.

Consider next those integrals from $a_r$ to $a$ arising from $1 \le|f'(a)-r| < \sqrt{f''(a)}$.  There are $\ll\sqrt{f''(a)}$ such integrals, each of which, by Proposition \ref{prop:2ndderivtest} is bounded by 
\[
O\left( \frac{U(a)}{\sqrt{f''(a)}}\right).
\]
So the total contribution of these integrals is bounded by $O(g(a)).$

Lastly, we consider the sum of those integrals in line \eqref{eq:shiftedintegral} arising from $\sqrt{f''(a)}\le|f'(a)-r| <f''(a)$.  In particular, let
\[
S_1:= \sum_{\sqrt{f''(a)}\le |f'(a)-r| < f''(a)} \int_{a_r}^a g(x) e(f(x)-rx) \ dx
\]
denote the sum in question.  We apply Proposition \ref{prop:intbypartsvar} to each integral in $S_1$ to obtain
\begin{align}
S_1 &=\sum_r \frac{g(a)}{2\pi i (f'(a)-r)} e(f(a)-ra) \label{eq:S1 1}\\
& \qquad+  \sum_r  O \left( \frac{U(a)}{f''(a)M(a)}  \right)\label{eq:S1 2}\\
& \qquad + \sum_r \left( O\left( \frac{U(a)}{M(a)(f'(a)-r)^2} \right)  + O \left( \frac{U(a)f''(a)}{(f'(a)-r)^3} \right)  \right). \label{eq:S1 3}
\end{align}
Here, each sum is over all $r$ satisfying $\sqrt{f''(a)}\le|f'(a)-r| <f''(a)$.

For the sum on line \eqref{eq:S1 1}, we apply a change of variables $m+\epsilon=f'(a)-r$, where $\epsilon=\langle f'(a)\rangle$, and then apply Proposition \ref{prop:fouriertail}.  The sum is then bounded by $O(U(a)/\sqrt{f''(a)})=O(U(a))$, since we are under the temporary assumption that $f''(a)\ge 1$.

The sum on line \eqref{eq:S1 2} has at most $2\cdot f''(a)$ terms, so is bounded by $$O\left(  \frac{U(a)}{M(a)} \right).$$

By Euler-Maclaurin summation (Lemma \ref{lem:emsum}), the sum on line \eqref{eq:S1 3} is bounded by $$O\left( U(a)\left( \frac{1}{\sqrt{f''(a)}M(a)}+1 \right) \right).$$

Thus the total contribution of the integrals from $a_r$ to $a$ in line \eqref{eq:shiftedintegral} arising from $1\le |f'(a)-r|< f''(a)$ is bounded by
\begin{align*}
&O\left( U(a)\left( 1+\frac{1}{M(a)}+\frac{1}{\sqrt{f''(a)}M(a)} \right) \right)\\
&\qquad = O(\mathcal{D}^\circ_{f'' \ge 1}(a)),
\end{align*}
where 
\[
\mathcal{D}^\circ_{f'' \ge 1}(a)= U(a)\left( 1+\frac{1}{M(a)}+\frac{1}{\sqrt{f''(a)}M(a)} \right).
\]

Similarly the contribution of the integrals from $b_r$ to $b$ in line \eqref{eq:shiftedintegral} is bounded by 
\begin{align*}
O(\mathcal{D}^\circ_{f'' \ge 1}(b))+O(\Delta_1(b)).
\end{align*}

Therefore, 
\begin{align}
 \sideset{}{^*}\sum_{a\le n \le b} g(n)e(f(n)) &= \lim_{R\to \infty} \sum_{r=-R}^R \int_{a_r}^{b_r} g(x)e(f(x)-rx) \ dx\notag \\
&\qquad + O(\mathcal{D}^\circ_{f'' \ge 1}(a))+ O(\mathcal{D}^\circ_{f'' \ge 1}(b))\label{eq:endfirststep} \\
&\qquad + O(\Delta_1(a)+ \Delta_1(b)). \notag
\end{align}

\subsection{The half stationary phase estimate}\label{section:proofhalf}

Suppose that $f'(a)=r\in\mathbb{Z}$.  We will need to estimate the integral at this particular $r$ in \eqref{eq:endfirststep} separately.

We shall break the integral into a piece surrounding the stationary phase point and an integral away from the stationary phase point.  In particular, we write $$\int_{a}^{b_r}g(x)e(f(x)-rx)\ dx=\left(\int_{a}^{\beta_r}+ \int_{\beta_r}^{b_r}\right)g(x)e(f(x)-rx)\ dx,$$ where $\beta_r=\min\{a+M(a),b_r  \}$.  Since $x_r=a$ is to the left of $b$, we have that $b_r$ either equals $b$ or $b+M(b)$.

By Propositions \ref{prop:main1} and \ref{prop:main2}, we have
\begin{align}
&\int_{a}^{\beta_r} g(y)e(f(y)-ry)\ dy\notag\\
&\qquad =\frac{g(a)e(f(a)-ra+1/8)}{2\sqrt{f''(a)}} + \frac{g(\beta_r) e(f(\beta_r)-r\beta_r)}{2\pi i (f'(\beta_r)-r)}  \notag \\ 
&\qquad \qquad + \dfrac{g(a)f^{(3)}(a)e(f(a))}{6\pi i f''(a)^2}-\dfrac{g'(a)e(f(a))}{2\pi i f''(a)} + O\left(  \frac{U(a)}{f''(a)^2 (b-a)^3} \right)\notag\\
&\qquad \qquad +O\left(  \frac{U(a)}{f''(a)^2 M(a)^3} + \frac{U(a)}{f''(a)^{3/2}M(a)^2}\right)\notag\\
  &\qquad =\frac{g(a)e(f(a)-ra+1/8)}{2\sqrt{f''(a)}} + \frac{g(\beta_r) e(f(\beta_r)-r\beta_r)}{2\pi i (f'(\beta_r)-r)} +\mathcal{D}^*(a) \label{eq:halfstationaryata}\\
&\qquad \qquad +  O(\Delta_1(a)+\Delta_2(a)).   \notag
\end{align}
Here we used the bound
\[
\frac{1}{(\beta_r - a)^3} = O\left(\frac{1}{(b-a)^3}\right) + O\left( \frac{1}{M(a)^3} \right).
\]

Similarly, if $f'(b)=r\in\mathbb{Z}$, then we dissect the integral from $a_r$ to $b$ in \eqref{eq:endfirststep} in a similar way, obtaining,
\begin{align}
&\left(\int_{a_r}^{\alpha_r}+\int_{\alpha_r}^b\right) g(x)e(f(x)-rx)\ dx\notag\\
&\qquad=\frac{g(b)e(f(b)-rb+1/8)}{2\sqrt{f''(b)}} - \frac{g(\alpha_r) e(f(\alpha_r)-r\alpha_r)}{2\pi i (f'(\alpha_r)-r)}-\mathcal{D}^*(b)\label{eq:halfstationaryatb}\\
&\qquad \qquad +O(\Delta_1(b))+O(\Delta_2(b))+ \int_{a_r}^{\alpha_r} g(x)e(f(x)-rx)\ dx, \notag
\end{align}
where $\alpha_r=\max\{b-M(b),a_r\}$.

\subsection{The full stationary phase estimates}\label{section:prooffull}

For all remaining integrals with a stationary phase point in \eqref{eq:endfirststep}, we may assume that $|f'(a)-r|,|f'(b)-r|>0$.  As before, we denote the stationary phase point corresponding to a given $r$ by $x_r$.  For each such $r\in(f'(a),f'(b))$, we write the corresponding integral in line \eqref{eq:endfirststep} as 
$$\int_{a_r}^{b_r} g(x)e(f(x)-rx) \ dx = \left( \int_{a_r}^{\alpha_r} + \int_{\alpha_r}^{\beta_r} + \int_{\beta_r}^{b_r}  \right) g(x)e(f(x)-rx) \ dx ,$$
where $\alpha_r=\max\{x_r-M(x_r),a_r\}$ and $\beta_r=\min\{x_r+M(x_r),b_r\}$.  By construction, $[\alpha_r,\beta_r]=I_{x_r}\cap [a_r,b_r]$ with $I_x$ as in the statement of condition $(M)$.

For this section, we shall focus on the contribution of the middle terms.  Consider 
\[
S_2:=\sum_{f'(a)<r<f'(b)} \int_{\alpha_r}^{\beta_r} g(x)e(f(x)-rx) \ dx.
\]
Applying Propositions \ref{prop:main1} and \ref{prop:main2} with $c=x_r$ to $S_2$, this becomes 
\begin{align}
S_2&=\sum_{f'(a)<r<f'(b)}  \frac{g(x_r)e(f(x_r)-rx_r+\frac{1}{8})}{\sqrt{f''(x_r)}} \label{eq:fullstationary}\\ 
&\qquad+ \sum_{f'(a)<r<f'(b)}O\left( \frac{g(x_r)}{f''(x_r)^2}\left( \frac{1}{(x_r-a_r)^3} +\frac{1}{(b_r-x_r)^3}   \right)\right)\label{eq:fullstationaryerror1}\\ 
&\qquad+\sum_{f'(a)<r<f'(b)} O\left( \frac{g(x_r)}{f''(x_r)^ 2M(x_r)^3} \left(1+\sqrt{f''(x_r)}M(x_r)  \right)\right)\label{eq:fullstationaryerror2}\\
&\qquad -\sum_{f'(a)<r<f'(b)} \frac{g(\alpha_r)e(f(\alpha_r)-r\alpha_r)}{2\pi i (f'(\alpha_r)-r)}  +\sum_{f'(a)<r<f'(b)} \frac{g(\beta_r)e(f(\beta_r)-r\beta_r)}{2\pi i (f'(\beta_r)-r)}  \label{eq:fullstationaryendpoints} 
\end{align}

The sum in line \eqref{eq:fullstationary} added to the first terms in lines \eqref{eq:halfstationaryata} and \eqref{eq:halfstationaryatb}---if they exist---sum to
\[
\sideset{}{^*}\sum_{f'(a)\le r \le f'(b)}  \frac{g(x_r)e(f(x_r)-rx_r+\frac{1}{8})}{\sqrt{f''(x_r)}},
\]
the main term in the van der Corput transform.

To begin evaluating the sum in line \eqref{eq:fullstationaryerror1}, consider values $r$ such that $0<x_r-a \le \min\{C_2^{-1}, M(a)\}$.  By the mean value theorem, we can write $x_r-a=(r-f'(a))/f''(\zeta)$ for some $\zeta \in (a,x_r) \subset (a,a+M(a))$.  By the bounds of condition $(M)$, the $r$ under consideration must also satisfy $0<r-f'(a) \le f''(a)$, which, in turn implies that $x_r-a_r= x_r-a+M(a)\ge M(a)$.

Therefore, we split the sum over 
\[
\frac{g(x_r)}{f''(x_r)^2(x_r-a_r)^3}
\]
in line \eqref{eq:fullstationaryerror1} into two pieces.  The first piece, over all $r$ such that $x_r-a \le \min\{ C_2^{-1}, M(a)\}$, has at most $m_a$ terms, with each term bounded by 
\begin{equation} \label{eq:nearaerror}
\frac{g(x_r)}{f''(x_r)^2 M(a)^3}\ll \frac{U(a)}{f''(a)^2 M(a)^3}.
\end{equation}
Thus the total contribution of these terms is at most 
\[
O\left( \frac{U(a)m_a}{f''(a)^2 M(a)^3}\right) = O(\Delta_2(a)),
\]
since $m(a)=O(1+f''(a))$.

The second piece, over all remaining $r$ between $f'(a)$ and $f'(b)$, is bounded by 
\begin{equation} \label{eq:fullstationaryerror1mod}
\sum_{\overline{a}\le x_r\le b} \frac{g(x_r)}{f''(x_r)^2(x_r-a)^3},
\end{equation}
where $\overline{a}$ is the smallest value in the interval $[a+\min\{M(a),C_2^{-1}\},b]$ such that $f'(\overline{a})$ is an integer.  Note that if no such integer exists, then this sum is empty.  We now apply Proposition \ref{prop:inversesum} to \eqref{eq:fullstationaryerror1mod} with 
\[
G(y)=\frac{g(y)}{f''(y)^2(y-a)^3}                                                                                                                                                                                                                                                                                                                                                    \]
and $F(y)=f'(y)$ on the interval $[f'(\overline{a}),f'(b)]$.  Using the bounds from condition $(M)$ liberally, we see that \eqref{eq:fullstationaryerror1mod} is bounded by
\begin{align}
&O\left(\int_{\overline{a}}^{b} \frac{U(x)}{f''(x)^2(x-a)^3}\left(f''(x)+\frac{1}{M(x)}+\frac{1}{x-a}  \right)  \ dx\right)\label{eq:delta3 1}\\
&\qquad\qquad + O\left( \frac{U(a)}{f''(a)^2(\overline{a}-a)^3} \right)   +O\left(\frac{U(b)}{f''(b)^2 (b-a)^3}  \right) \label{eq:delta3 2}\\
&\qquad =O(\Delta_3(a)).\notag
\end{align}
Since $f''(x)>0$ on $[a,b]$, $f'$ has no local extrema on this interval, so the final term from Proposition \ref{prop:inversesum} does not appear.

We have a similar bound for the second sum in line \eqref{eq:fullstationaryerror1}:
\[
\sum_{f'(a)<r<f'(b)} \frac{g(x_r)}{(f''(x_r))^2(b_r-x_r)^3}=O(\Delta_2(b))+O(\Delta_3(b)).
\]

We apply Proposition \ref{prop:inversesum} again to the sum on line \eqref{eq:fullstationaryerror2}  to get
\begin{align}
&\sum_{f'(a)<r<f'(b)} \frac{g(x_r)}{f''(x_r)^{2} M(x_r)^3}\left(1+\sqrt{f''(x_r)}M(x_r)  \right) \label{eq:GoverFEinverse}\\
&\qquad =O\left( \int_a^b \frac{U(x)}{f''(x) M(x)^3}\left( 1+\sqrt{f''(x)}M(x) \right)  \left(1+ \frac{1+|M'(x)|}{f''(x) M(x)}  \right)\ dx\right) \label{eq:delta4 1a} \\
&\qquad\qquad+ O\left(  \frac{U(a)}{f''(a)^2M(a)^3}\left(1+\sqrt{f''(a)}M(a)  \right)\right) \notag \\
&\qquad\qquad + O\left( \frac{U(b)}{f''(b)^2 M(b)^3} \left(1+\sqrt{f''(b)}M(b)  \right)\right) \label{eq:delta4 2a} \\
&\qquad=O(\Delta_4)+O(\Delta_2(a)+\Delta_2(b)) . \notag
\end{align}

Now consider
\begin{align*}
S_3:=&\lim_{R\to \infty} \sum_{\substack{|r|\le R \\ r< f'(a) \text{ or } r > f'(b)  }} \int_{a_r}^{b_r} g(x) e(f(x)-rx) \ dx \\
&\qquad +\sum_{f'(a)\le r < f'(b)} \left( \frac{g(\beta_r)e(f(\beta_r)-r\beta_r)}{2\pi i (f'(\beta_r)-r)} +\int_{\beta_r}^{b_r}  g(x) e(f(x)-rx) \ dx \right) \\
&\qquad +\sum_{f'(a) < r \le f'(b)} \left( \int_{a_r}^{\alpha_r}g(x) e(f(x)-rx) \ dx-\frac{g(\alpha_r)e(f(\alpha_r)-r\alpha_r)}{2\pi i (f'(\alpha_r)-r)}\right).
\end{align*}
Note that the remaining terms of $S_2$ in line \eqref{eq:fullstationaryendpoints} as well as the second terms from lines \eqref{eq:halfstationaryata} and \eqref{eq:halfstationaryatb}---if they exist---appear in $S_3$.

We have thus far shown that 
\begin{align*}
\sideset{}{^*}\sum_{a\le n \le b} g(n) e(f(n)) &= \sideset{}{^*}\sum_{f'(a)\le r\le f'(b)} \frac{g(x_r)e(f(x_r)-rx_r+\frac{1}{8})}{\sqrt{f''(x_r)}}+S_3-\mathcal{D}^*(b)+\mathcal{D}^*(a)\\
&\qquad + O(\mathcal{D}^\circ_{f'' \ge 1}(a))+ O(\mathcal{D}^\circ_{f'' \ge 1}(b))+ \sum_{i=1}^4 O(\Delta_i(a)+\Delta_i(b))+O(\Delta_4).
\end{align*}

\subsection{The remaining integrals}\label{section:remainingintegrals}

For each integral in $S_3$, we apply Proposition \ref{prop:intbyparts}.  This gives
\begin{align}
S_3&=- \lim_{R\to \infty}\sum_{r\neq f'(a)} \frac{g(a_r)e(f(a_r)-ra_r)}{2\pi i (f'(a_r)-a)}+\lim_{R\to \infty} \sum_{r\neq f'(b)} \frac{g(b_r)e(f(b_r)-rb_r)}{2\pi i (f'(b_r)-r)}\label{eq:s3 1}\\
&\qquad +\lim_{R\to \infty}\sum_{r\neq f'(a)}O\left( h_r(a_r) \right)  +\lim_{R\to \infty}\sum_{r\neq f'(b)}O\left( h_r(b_r) \right) \label{eq:s3 2}\\
&\qquad + \sum_{f'(a)<r<f'(b)} \left( O(h_r(\alpha_r))+O(h_r(\beta_r))\right) \label{eq:s3 3}\\
&\qquad + \sum_{f'(a)\le r\le f'(b)} \left( O(K_r(a_r,\alpha_r))+O(K_r(\beta_r,b_r))\right) \notag \\
&\qquad + \lim_{R\to \infty}\sum_{r<f'(a)\text{ or }r>f'(b)} O(K_r(a_r,b_r)) \notag
\end{align}
where in each sum it is assumed that $|r| \le R$. 

We start estimating the first sum of \eqref{eq:s3 1} and write 
\begin{align}
&- \lim_{R\to \infty}\sum_{r\neq f'(a)} \frac{g(a_r)e(f(a_r)-ra_r)}{2\pi i (f'(a_r)-r)}\notag\\
&\qquad=\sum_{0< r-f'(a)\le  f''(a)} O\left( \frac{g(a-M(a))}{f'(a-M(a))-r}\right)\label{eq:s31a}\\
&\qquad\qquad +\sum_{0< -(r-f'(a)) \le f''(a)} O\left(\frac{g(a+M(a))}{f'(a+M(a))-r}\right)\label{eq:s31b}\\
&\qquad\qquad - \frac{g(a)e(f(a))}{2\pi i} \lim_{R\to \infty}  \sum_{\substack{|r|\le R \\ |f'(a)-r| > f''(a)}} \frac{e(-ra)}{f'(a)-r}\label{eq:fourier1}
\end{align}
There are at most $m_a$ terms on lines \eqref{eq:s31a} and \eqref{eq:s31b}, each with size
\[
O\left( \frac{U(a)}{f''(a)M(a)} \right) 
\]
by Proposition \ref{prop:f'bound}, for a total contribution of 
\[
 O\left(\frac{U(a)m_a}{f''(a)M(a)} \right)=O(\Delta_2(a)).
\]

We can evaluate the sum on line \eqref{eq:fourier1} explicitly in some cases.  Suppose $f''(a)<1-\|f'(a)\|$.  If $f''(a)<\| f'(a)\|$, then we pull out the term where $f'(a)-r=\langle f'(a)\rangle$, so that the remaining terms on line \eqref{eq:fourier1} form a modified sawtooth function.  In particular, the sum on line \eqref{eq:fourier1} becomes
\[
\begin{cases}
\begin{aligned}&g(a)e(f(a)+\left\llbracket f'(a)\right\rrbracket a)\times\\ &\quad \times \left(-\dfrac{1}{ 2\pi i \langle f'(a) \rangle} + \psi(a,\langle f'(a) \rangle)\right)\end{aligned} & f''(a) < \| f'(a) \|\\
g(a)e(f(a)+\left\llbracket f'(a)\right\rrbracket a)\psi(a,\langle f'(a) \rangle) & \| f'(a) \| \le f''(a)< 1-\|f'(a)\|
\end{cases}
\]
which equals $\mathcal{D}^\circ(a)$, when $f''(a)<1-\|f'(a)\|$.  Otherwise, if $f''(a)\ge 1-\|f'(a)\|$, we may bound the sum on line \eqref{eq:fourier1} using Proposition \ref{prop:fouriertail} by $O(U(a)/f''(a))=O(U(a))$.  Recall that $\mathcal{D}^\circ(a)$ is not explicit when $f''(a)\ge 1-\|f'(a)\|$, and instead is just a big-O approximation with one of its terms of the form $O(U(a))$.

By a similar argument, one can show that all the sums on line \eqref{eq:s3 1} are equal to
\begin{align*}
-\mathcal{D}^\circ(b)+\mathcal{D}^\circ(a)+O(\Delta_2(a)+\Delta_2(b)).
\end{align*}

We break the sum from \eqref{eq:s3 2} into three smaller sums like so
\begin{align}
\lim_{R\to \infty} \sum_{\substack{|r|\le R \\ r\neq f'(a)}} O(h_r(a_r)) &= \lim_{R\to \infty} \sum_{\substack{|r|\le R \\ |f'(a)-r|\ge f''(a)}} O(h_r(a)) \label{eq:s32a}\\
&\qquad + \sum_{0< r-f'(a) < f''(a)} O(h_r(a-M(a)))  \label{eq:s32b}\\
&\qquad + \sum_{0< f'(a)-r < f''(a)} O(h_r(a+M(a)) \label{eq:s32c}.
\end{align}

The sum on the right-hand side of line \eqref{eq:s32a} is
\begin{align}
&\lim_{R\to \infty} \sum_{\substack{|r|\le R \\ |f'(a)-r|\ge f''(a)}} \left| \frac{(f'(x)-r)g'(x)-g(x)f''(x)}{(f'(x)-r)^3} \right|_{x=a} \notag \\
&\qquad \ll \lim_{R\to \infty} \sum_{\substack{|r|\le R \\ |f'(a)-r|\ge f''(a)}}\left( \left| \frac{U(a)}{M(a)(f'(a)-r)^2}\right| + \left| \frac{U(a)f''(a)}{(f'(a)-r)^3}\right| \right) \label{eq:s32a1}
\end{align}
If $\|f'(a)\| =0$ or $m_a\ge 1$, then the smallest $|f'(a)-r|$ could be in this sum is $\min\{1/2,f''(a)\}$; otherwise, we will have a term where $|f'(a)-r|=\|f'(a)\|$.  Thus this sum is bounded by
\begin{align*}
 &\ll \begin{cases}
       \begin{aligned}&\dfrac{U(a)}{M(a)}\min\left\{1,\dfrac{1}{f''(a)}  \right\}\\ &\quad +U(a)\min\left\{f''(a),\dfrac{1}{f''(a)}  \right\}\end{aligned} & \text{if }\|f'(a)\|=0 \text{ or } m_a \ge 1\\
       \dfrac{U(a)}{M(a)\|f'(a)\|^2} + \dfrac{U(a)f''(a)}{\|f'(a)\|^3} & \text{otherwise}
      \end{cases}\\
 &\ll \Delta_2(a).
\end{align*}

By using condition $(M)$ and Proposition \ref{prop:f'bound}, each term on lines \eqref{eq:s32b} and \eqref{eq:s32c} has size at most
\[
 \ll \frac{U(a)}{f''(a)^2M(a)^3},
\]
and there are at most $m_a=O(1+f'(a))$ such terms for a total contribution of at most 
\[
O\left(\frac{U(a)}{f''(a)^2M(a)^3}(1+f''(a))\right)=O(\Delta_2(a)).
\]

By a similar argument both the sums on line \eqref{eq:s3 2} are bounded by $O(\Delta_2(a)+\Delta_2(b))$.

For the first sum on line \eqref{eq:s3 3}, we have $\alpha_r$ equals $a_r$ or $x_r-M(x_r)$.  The sum of terms with $\alpha_r=a_r$ are bounded by the sum in line \eqref{eq:s3 2}, which we just showed has size $O(\Delta_2(a))$, so we only need to estimate the terms where $\alpha_r=x_r-M(x_r)$.  By Proposition \ref{prop:f'bound} and the bounds of condition $(M)$, each such term is bounded by
\[
\ll \frac{g(x_r)}{f''(x_r)^2 M(x_r)^3}.
\]
We may therefore bound the sum over those $r$ with $\alpha_r=x_r-M(x_r)$ by a sum of terms $g(x_r)/f''(x_r)^2M(x_r)^3$ for all $r\in[f'(a),f'(b)]$ and obtain
\[
\sum_{x_r\in[a,b]} \frac{g(x_r)}{f''(x_r)^2 M(x_r)^3} = O(\Delta_4)
\]
by the same argument as in line \eqref{eq:GoverFEinverse}.

Therefore, we have shown that
\begin{align*}
\sideset{}{^*}\sum_{a\le n \le b} g(n) e(f(n)) &= \sideset{}{^*}\sum_{f'(a)\le r\le f'(b)} \frac{g(x_r)e(f(x_r)-rx_r+\frac{1}{8})}{\sqrt{f''(x_r)}}-\mathcal{D}(b)+\mathcal{D}(a)\\
&\qquad +\sum_{i=1}^3 O(\Delta_i(a)+\Delta_i(b)) +O(\Delta_4)\\
&\qquad + \sum_{f'(a)\le r\le f'(b)} \left( O(K_r(a_r,\alpha_r))+O(K_r(\beta_r,b_r))\right)\\
&\qquad + \sum_{r<f'(a)\text{ or }r>f'(b)} O(K_r(a_r,b_r)).
\end{align*}
The proof will therefore be finished when we show that the sum of all the $K_r$ terms is bounded by $O(\Delta_4)$.

\subsection{Bounding the variation}\label{sec:variation}

The $K_r$'s are the sum of terms of the size $$h_r(x)=\frac{(f'(x)-r)g'(x)-g(x)f''(x)}{(f'(x)-r)^3}$$ at points $x$ where the derivative with respect to $x$ vanishes.   We may safely ignore points where the derivative does not exist, since all the intervals that give rise to the $K_r$ terms do not contain stationary phase points.  

The derivative of $h_r(x)$ is (using $f'_r$ as shorthand for $f'(x)-r$) 
\[
 h_r'(x)=\frac{-gf'_rf^{(3)}+g''(f'_r)^2+3g(f'')^2-3g'f'_rf''}{(f'_r)^4}=\frac{g''(f'_r)^2-Hf'_r+3g(f'')^2}{(f'_r)^4}.
\]
We set the numerator equal to 0 and solve for $f'_r$.  

First, suppose $g(x)=0$.  If $g''(x)$ also equals $0$, then the numerator is $0$ if and only if $g'(x)$ also equals $0$ ($f''(x)$ is never $0$, and we assumed $f'(x)\neq 0$ at all points in consideration); but if $g(x)$ and $g'(x)$ equal $0$, then $h_r(x)$ also equals $0$, so these points contribute nothing to $K_r$.  Thus the only contribution from points $x$ where $g(x)=0$ come when $g''(x)\neq 0$ and $g'(x)\neq 0$, i.e., those points in $J_{null}$ as defined in Section \ref{section:assumptions}.  At these points, solving for $f'_r$ yields $3g'(x)f''(x)/g''(x)$, so the contribution to the $K_r$ terms from these points is bounded by.
\[
\ll \sum_{x\in J_{null}} \frac{g''(x)^2}{g'(x)f''(x)^2}.
\]

For the remainder of this section, we will suppose then that $g(x) \neq 0$.  There is no way to make the numerator equal $0$ if $g''(x)=H(x)=0$, since both $g(x)$ and $f''(x)$ are assumed to be non-zero, or if $H(x)^2-G(x)<0$, since all the above functions are real-valued; otherwise, the solution is given by
\begin{equation}\label{eq:f'_r}
 f'_r =\begin{cases} \dfrac{H\pm \sqrt{H^2-G}}{2g''}  & x\in J_\pm \\
\dfrac{3g(f'')^2}{H} & x\in J_0\end{cases},
\end{equation}
where $J_\pm$ and $J_0$ are as defined in Section \ref{section:assumptions}.  Plugging this value of $f'_r$ into $h_r$, we obtain
\[
h_r(x)=\begin{cases} 
W_\pm(x) &  x\in J_\pm\\
W_0(x) & x\in J_0\end{cases},
\]
where
\begin{align*}
 W_\pm&=\dfrac{(2g'')^2g'}{(H\pm \sqrt{H^2-G})^2} -\dfrac{(2g'')^3(f''g)}{(H\pm \sqrt{H^2-G})^3}, \qquad \text{and} \\
W_0 &= \dfrac{H^2g'}{(3g(f'')^2)^2}-\dfrac{H^3f''g}{(3g(f'')^2)^3}\\
&=-\dfrac{H^2 f^{(3)}}{27g(f'')^5}.
\end{align*}

If we ignore the constraint that $r$ be an integer for the moment, then we can imagine that \eqref{eq:f'_r} determines functions 
\begin{align*}
r_\pm(x)&=f'-\frac{H\pm \sqrt{H^2-G}}{2g''}  & &\text{on }J_\pm\text{ and} \\
r_0(x)&=f'-\frac{3g(f'')^2}{H} & &\text{on }J_0.
\end{align*}

The contribution of the $K_r$ terms is then at most 
\begin{align*}
&\ll  \sum_{r_0^{-1}(n)\in J_0} W_0(r_0^{-1}(n)) \\
&\qquad  + \sum_{r_+^{-1}(n)\in J_\pm} W_+(r_+^{-1}(n))+\sum_{r_-^{-1}(n)\in J_\pm} W_-(r_-^{-1}(n)).
\end{align*}

We now apply Proposition \ref{prop:inversesum} to see that the contribution of the $K_r$ terms is bounded by 
\begin{align*}
&\ll \mathcal{K}(J_0,W_0,r_0) + \mathcal{K}(J_\pm,W_+,r_+) + \mathcal{K}(J_\pm,W_-,r_-)\\
&\ll O(\Delta_4).
\end{align*}

This completes the proof of Theorem \ref{thm:main}.

\section{Proof of Theorem \ref{thm:alternate4}}

We follow the proof of Theorem \ref{thm:main}, until Section \ref{section:remainingintegrals}.

As in the conditions of Theorem \ref{thm:alternate4}, suppose that 
\[
M(x) \ge \max\{b-x,x-a\}
\] 
for all $x \in [a,b]$ and that $m_a=m_b=0$.  Then we may use Proposition \ref{prop:intbypartsvar} in place of Proposition \ref{prop:intbyparts} to the integrals in $S_3$ and obtain
\begin{align}
S_3&=- \lim_{R\to \infty}\sum_{r\neq f'(a)} \frac{g(a)e(f(a)-ra)}{2\pi i (f'(a)-a)}+\lim_{R\to \infty} \sum_{r\neq f'(b)} \frac{g(b)e(f(b)-rb)}{2\pi i (f'(b)-r)} \label{eq:s3var 1} \\
&\qquad +\lim_{R\to \infty}\sum_{r\neq f'(a)}O\left( \frac{U(a)f''(a)}{|f'(a)-r|^3}+\frac{U(a)}{M(a)(f'(a)-r)^2} \right) \label{eq:s3var 2} \\
&\qquad  +\lim_{R\to \infty}\sum_{r\neq f'(b)}O\left( \frac{U(b)f''(b)}{|f'(b)-r|^3}+\frac{U(b)}{M(b)(f'(b)-r)^2} \right)  \label{eq:s3var 3}
\end{align}
We used $a$ and $b$ in place of $a_r$ and $b_r$ due to $m_a$ and $m_b$ being zero.  Likewise, we have no $\alpha_r$, $\beta_r$ terms since by our assumption on $M(x)$, $\alpha_r=a$ and $\beta_r=b$ for all $r$.

Just as in the proof of Theorem \ref{thm:main}, the terms in line \eqref{eq:s3var 1} equal
\[
-\mathcal{D}(b)+\mathcal{D}(a)+O(\Delta_2(a)+\Delta_2(b)).
\]
The terms on lines \eqref{eq:s3var 2} and \eqref{eq:s3var 3} we bound as we did the terms in line \eqref{eq:s32a1}, obtaining
\begin{align*}
&\ll \frac{U(a)}{M(a)\|f'(a)\|^{*2}} + \frac{U(a)f''(a)}{\|f'(a)\|^{*3}}\\
&\qquad + \frac{U(b)}{M(b)\|f'(b)\|^{*2}} + \frac{U(b)f''(b)}{\|f'(b)\|^{*3}}\\
&\ll \Delta_2(a)+\Delta_2(b) .
\end{align*}

Since we have no $K_r$ terms, we can let $\Delta_4$ just equal
\[
 \int_a^b \frac{U(x)}{f''(x)M(x)^3} \left(1+\sqrt{f''(x)}M(x)\right) \left( 1+ \frac{1+|M'(x)|}{f''(x)M(x)}\right) \ dx,
\]
and this completes the proof of Theorem \ref{thm:alternate4} in this case.

\section{Proof of Theorem \ref{thm:toinfinity}}

The proof of this theorem mostly entails following the proof of Theorem \ref{thm:main} and simply being more careful with how certain error terms arise.

In Section \ref{section:prooffirststep}, we want to show that we can replace $O(\mathcal{D}^\circ_{f''\ge 1}(a))+O(\Delta_1(a))$ with some constant $c_1$.  But these terms simply serve as a bound for the sum
\[
- \sum_{r=-\infty}^\infty \int_{a_r}^a g(x)e(f(x)-rx) \ dx,
\]
which, for fixed $M(a)$, is a constant as $b$ tends to $\infty$.  (If $b$ is near $a$, we may need to take $M(a)$ to be smaller, thus it's not always a constant.)  Likewise, the term $O(\Delta_1(a)+\Delta_2(a))$ that appear in Section \ref{section:proofhalf} arise from bounding the term $\mathcal{E}(\beta_r)$ from Proposition \ref{prop:main1} applied with the $c$ and $b$ of the statement of the proposition equal to $a$ and $\beta_r$, respectively.  Provided $b$ is large enough, $\beta_r$ will be constant, so the term $\mathcal{E}(\beta_r)$ may be likewise replaced by a constant $c_2$.  

The error terms from Section \ref{section:prooffull} arise from bounding the sum
\begin{equation}\label{eq:toinfinity1}
\sum_{a< x_r < b} \left( \mathcal{E}_r(\alpha_r) + \mathcal{E}_r(\beta_r) \right)
\end{equation}
where $\mathcal{E}_r(\alpha_r)$ (respectively, $\mathcal{E}_r(\beta_r)$) is the error from Proposition \ref{prop:main1} applied with $a$ (respectively, $b$) and $c$ in the statement of the proposition equal to $\alpha_r$ (respectively, $\beta_r$) and $x_r$.

We would like to replace the $b$ in line \eqref{eq:toinfinity1} with $\infty$.  Consider first the $\mathcal{E}_r(\alpha_r)$ terms.  Each such term is fixed (provided $b$ is sufficiently large enough that $b-a >M(a)$).  Thus we write
\[
\sum_{a<x_r < b} \mathcal{E}_r(\alpha_r) = \left( \sum_{a< x_r < b'} -\sum_{b\le x_r < b'} \right) \mathcal{E}_r(\alpha_r). 
\]
Suppose now $b'$ is chosen to be on a sequence tending to infinity such that $\Delta'_3(b')$ tends to $0$.  Then we have that
\[
\lim_{b'\to \infty}  \sum_{a< x_r < b'} \mathcal{E}_r(\alpha_r) 
\]
exists and is a constant $c_3$ by the assumed convergence of the integral in $\Delta_3(a)$, and we have that 
\begin{align*}
\sum_{b\le x_r < b'} \mathcal{E}_r(\alpha_r) &\ll \int_b^{b'} \frac{U(x)}{f''(x)(x-a)^3}\left(1+\frac{1}{f''(x)M(x)}+\frac{1}{f''(x)(x-a)} \right) \ dx\\
&\qquad + \int_b^{b'} \frac{U(x)}{f''(x)M(x)^3}\left(1+\sqrt{f''(x)}M(x)\right)\left(1+\frac{1+|M'(x)|}{f''(x)M(x)}\right) \ dx\\
& \qquad + \Delta'_3(b) + \Delta'_3(b')\\
&\ll \Delta'_3(b)+\Delta'_3(b')+\Delta_5
\end{align*}
by Proposition \ref{prop:inversesum}.  Thus we have
\[
\sum_{a<x_r < b} \mathcal{E}_r(\alpha_r) = c_3 + O(\Delta'_3(b)) +O(\Delta_5).
\]

We must apply an additional idea to perform the same technique to the terms $\mathcal{E}_r(\beta_r)$.  In particular, if $x_r+M(x_r)>b_r$, then $\beta_r$ changes as $b$ grows.  We may replace all such terms by $\mathcal{E}_r(x_r)$ with an error of size $O(\Delta_4'(b))$.  Following the argument of the previous paragraph, we have
\[
\sum_{a<x_r < b} \mathcal{E}_r(\beta_r) = c_3 + O(\Delta'_4(b)) +O(\Delta_5).
\]

In Section \ref{section:remainingintegrals}, the first sum on line \eqref{eq:s3 1} is constant for $b$ sufficiently large, and the second sum is just $\mathcal{D}^\circ(b)+O(\Delta_2(b))$.

By following the details of Proposition \ref{prop:intbyparts} that we applied to $S_3$, we see that what we are left with at this point is 
\begin{align}
&\sum_{f'(a)\le r\le f'(b)} \left( \int_{a_r}^{\alpha_r}+\int_{\beta_r}^{b_r}  \right) \frac{ h_r(x)}{2\pi i} e(f(x)-rx) \ dx \label{eq:toinfinity2}\\
&\qquad + \sum_{r<f'(a) \text{ or }r>f'(b)} \int_{a_r}^{b_r} \frac{h_r(x)}{2\pi i} e(f(x))-rx) \ dx \label{eq:toinfinity3}
\end{align}
where in the latter sum we assume we are taking the limit as $R$ tends to infinity of the sum with the additional restriction that $|r|\le R$.  Let us, for ease of notation, rewrite the integrals as 
\[
\int_{[a_r,b_r]\setminus I_{x_r}} ,
\]
where as in condition $(M)$ the set $I_{x_r}$ equals $[x-M(x_r),x+M(x_r)]$.  We wish to replace $b$ by a sequence of $b'$'s tending to infinity, such that $\Delta_2(b')$ tends to $0$.  For such $b'$ we have that above sum equals
\begin{align*}
&\lim_{R\to \infty}\sum_{|r|\le R} \int_{[a_r,b'_r]\setminus I_{x_r}}\frac{ h_r(x)}{2\pi i} e(f(x)-rx) \ dx \\
&\qquad -\lim_{R\to \infty}\sum_{|r|\le R} \int_{[b_r,b'_r]\setminus I_{x_r}}\frac{ h_r(x)}{2\pi i} e(f(x)-rx) \ dx.
\end{align*}
We now finish applying Proposition \ref{prop:intbyparts} by taking integration by parts of all the integrals and then bounding the variation.  This combined with the ideas from earlier in the proof give us that the terms on \eqref{eq:toinfinity2} and \eqref{eq:toinfinity3} are bounded by
\[
c_4+O(\Delta_2(b)) + O(\Delta_5).
\]

This completes the proof of Theorem \ref{thm:toinfinity}.

\section{Proof of Theorem \ref{thm:refinement}}

In the estimates of Theorem \ref{thm:main}, we only roughly estimated the size of $-\mathcal{D}(b)+\mathcal{D}(a)$, which roughly are the first-order endpoint contributions at $b$ and $a$, when $f''$ was large.  In this case, we included error terms of the size $O(U(a))+O(U(b))$, which are at least the size of the first and last term of our initial sum.  These terms can be improved, but require additional work.  

But recall, at the beginning of the proof of Theorem \ref{thm:main}, we bounded a sum of the type\footnote{Since $a_r=a$ for almost all $r$, the sum is in fact finite.}
\begin{equation}\label{eq:shiftedintegrala}
\sum_{r=-\infty}^\infty \int_{a_r}^a g(x) e(f(x)-r x) \ dx
\end{equation}
in line \eqref{eq:shiftedintegral}, which along with the sum in line \eqref{eq:fourier1} contributed to the $O(U(a))$ error in $\mathcal{D}(a)$.  In bounding this sum, we exploited the fact that the first-order endpoint contributions at $a$ should cancel significantly (for example, in our estimation of line \eqref{eq:S1 1}).  We expect, however, that more should be true, that not only should the first-order endpoint contributions display cancellations, but the integrals themselves should also display cancellation near $a$.

As in the conditions of Theorem \ref{thm:refinement}, assume that $M(a),f''(a)\ge 1$, let $C$ and $L$ be real numbers  satisfying 
\[
f''(a)^{-1/2}\ll C< M(a)\qquad \text{and} \qquad \sqrt{f''(a)} \ll L< f''(a)\cdot\min\{1,C\},
\]
and let $\epsilon:=\langle a\rangle$ and $\epsilon':=\langle f'(a)\rangle$.  We now define a value $a_r'$ similarly to how we defined $a_r$ on line \eqref{eq:a_rdef}.
\[
a_r':= \begin{cases} 
a+C & \text{if }1\le f'(a)-\epsilon'-r \le L\\
a-C &  \text{if } -1 \ge f'(a)-\epsilon'-r \ge -L\\
a & \text{otherwise} \\
 \end{cases} 
\]
If $a_r'\neq a$ then $a_r'$ is between $a$ and $a_r$, except, possibly, for two values of $r$ if $L$ is very close to $f''(a)$.

We now break the original sum \eqref{eq:shiftedintegrala} into several pieces:
\[
\sum_{1\le |f'(a)-\epsilon'-r|\le L} \int_{a_r'}^a 
+ \sum_{1\le|f'(a)-\epsilon'-r|\le L} \int_{a_r}^{a_r'} 
+ \sum_{|f'(a)-\epsilon'-r|\ge L} \int_{a_r}^a
+O(\Delta_1(a))
\]
where the last term comes from the term $r=f'(a)-\epsilon'$  as in \eqref{eq:smallrbound}.  In the sequel, we shall refer to these three groups of integrals as Type I, Type II, and Type III integrals, respectively.

We will postpone fully estimating the Type I integrals for now, and instead modify them into a more symmetric form.  We first replace $g(x)$ by 
\[
g(a)+O\left(\max_{z\in [a_r',a]} g'(z) \cdot |x-a|\right)
\] and $f(x)-rx$ by
\begin{align*}
&(f(a)-ra )+(f'(a)-r)(x-a)+\frac{1}{2}f''(a)(x-a)^2\\
&\qquad+O\left(\frac{\max_{z\in[a_r',a]} f^{(3)}(z)}{3!} |x-a|^3  \right).
\end{align*}
Doing so gives
{\allowdisplaybreaks \begin{align*}
&\int_{a_r'}^a g(x)e(f(x)-rx) \ dx\\
&\qquad=g(a) \int_{a_r'}^a e(f(x)-rx) \ dx + O\left( \frac{U(a)}{M(a)}\int_{a_r'}^a |x-a| \ dx\right)\\
&\qquad= g(a) \int_{a_r'}^a e(f(x)-rx) \ dx +O\left(\frac{U(a)C^2}{M(a)}  \right)\\
&\qquad= g(a)\int_{a_r'}^a e\left((f(a)-ra)+(f'(a)-r)(x-a)+\frac{1}{2}f''(a)(x-a)^2\right)\times \\
&\qquad\qquad \qquad \times \left(1+O\left(\frac{f''(a)}{M(a)}|x-a|^3\right)\right)\ dx\\
&\qquad\qquad +O\left(\frac{U(a)C^2}{M(a)}  \right)\\
&\qquad= g(a)\int_{a_r'}^a e\left((f(a)-ra)+(f'(a)-r)(x-a)+\frac{1}{2}f''(a)(x-a)^2\right) \ dx\\
&\qquad\qquad + O\left( \frac{U(a)f''(a)C^4}{M(a)}\right).
\end{align*}}
Shifting the bounds of integration and sending $f'(a)-r$ to $r+ \epsilon'$, the Type I integrals can be rewritten as
\begin{align*}
&\sum_{1\le |f'(a)-\epsilon'-r|\le L} \int_{a_r'}^a  g(x)e(f(x)-rx) \ dx\\
&\qquad =  g(a)e(f(a)+(\epsilon'-f'(a))a) \sum_{1\le |r|\le L} \int_{a_r'-a}^0e\left( \frac{f''(a)}{2}x^2+rx+ra+\epsilon' x\right) \ dx\\
&\qquad \qquad + O\left( \frac{U(a)f''(a)C^4L}{M(a)}\right).
\end{align*}

The Type II and Type III integrals we evaluate using Proposition \ref{prop:intbypartsvar}.  This gives us
\begin{align}
&\left(\sum_{1\le|f'(a)-\epsilon'-r|\le L} \int_{a_r}^{a_r'} + \sum_{|f'(a)-\epsilon'-r|\ge L} \int_{a_r}^a\right) g(x)e(f(x)-rx)\ dx\notag \\
&\qquad=\sum_{|f'(a)-\epsilon' -r|<L} \frac{g(a_r')}{2\pi i (f'(a_r')-r)}e(f(a_r')-ra_r') \label{eq:typeIIandIII1}\\
&\qquad\qquad + \sum_{\substack{L\le |f'(a)-\epsilon' -r|\\ |f'(a)-r|\le f''(a)}} \frac{g(a)}{2\pi i (f'(a)-r)}e(f(a)-ra) \notag\\
&\qquad\qquad+ O\left(  \sum_{|f'(a)-\epsilon' -r|\le L} \left(\frac{U(a)}{M(a)(f'(a_r')-r)^2}+\frac{U(a)f''(a)}{|f'(a_r')-r|^3} \right)\right)\label{eq:typeIIandIII4}\\
&\qquad\qquad+ O\left( \sum_{\substack{L\le |f'(a)-\epsilon' -r|\\ |f'(a)-r|\le f''(a)}} \left( \frac{U(a)}{M(a)f''(a)}+\frac{U(a)}{M(a)^3f''(a)^2} \right)\right)\label{eq:typeIIandIII2} \\
&\qquad\qquad+ O\left( \sum_{L\le |f'(a)-\epsilon' -r|} \left(\frac{U(a)}{M(a)(f'(a)-r)^2}+\frac{U(a)f''(a)}{|f'(a)-r|^3} \right)\right) \label{eq:typeIIandIII3}. 
\end{align}
Similar to Proposition \ref{prop:f'bound}, we have that 
\[
f'(a\pm c) = f'(a) \pm f''(a)\cdot C \left( 1+O\left(\frac{\eta \cdot C}{2 \cdot M(a)}\right)\right)
\]
with implicit constant $1$, so that $|f'(a_r')-r| \gg f''(a) C$ for $1\le |f'(a) - \epsilon' - r| \le L$.  Thus, the sum in line \eqref{eq:typeIIandIII1} is bounded by
\[
\sum_{|f'(a)-\epsilon' -r|<L} O\left(  \frac{U(a)}{f''(a)C}  \right) =O\left(\frac{U(a)L}{f''(a)C} \right),
\]
and the sum in line \eqref{eq:typeIIandIII4} is bounded by
\begin{align*}
& O\left( \sum_{1\le r} \frac{U(a)}{M(a)(f''(a)C+r)^2} + \frac{U(a)f''(a)}{(f''(a)C+r)^3} \right)\\
 &\qquad= O\left( \frac{U(a)}{M(a)f''(a)C} + \frac{U(a)}{f''(a)C^2}\right)\\
&\qquad =O\left(\frac{U(a)}{f''(a)C^2} \right),
\end{align*}
since $c \le M(a)$.  The terms on line \eqref{eq:typeIIandIII2} are bounded by 
\[
O\left(\frac{U(a)}{M(a)}+\frac{U(a)}{f''(a)M(a)^3 } \right)=O\left(\frac{U(a)}{M(a)}\right).
\]
Finally, we apply Euler-Maclaurin summation to the sum on line \eqref{eq:typeIIandIII3} to bound it by 
\[
O\left( \frac{U(a)}{M(a)L}+\frac{U(a)f''(a)}{L^2}\right)=O\left(\frac{U(a)f''(a)}{L^2} \right),
\]
since $L$ is at most $f''(a)$.

Using all of these estimates for the integrals on line \eqref{eq:shiftedintegrala} (with the negative sign now, as they appeared in the proof of Theorem \ref{thm:main}) and adding the terms from line \eqref{eq:fourier1}, we obtain the following.
\begin{align}
&-\sum_{r=-\infty}^\infty \int_{a_r}^a g(x) e(f(x)-r x) \ dx - \lim_{R\to \infty} \sum_{\substack{|r|\le R \\ |f'(a)-r|>f''(a)}} \frac{g(a)e(f(a)-ra)}{2\pi i(f'(a)-ra)}\notag\\
&\qquad = -g(a)e(f(a)+(\epsilon'-f'(a))a) \times \notag \\
&\qquad \qquad \qquad \times\sum_{1\le |r|\le L} \int_{c\cdot \operatorname{sgn}(r)}^0e\left( \frac{f''(a)}{2}x^2+rx+ra+\epsilon' x\right) \ dx \notag\\
&\qquad \qquad -\lim_{R\to\infty} \sum_{\substack{|r|\le R\\ L\le |f'(a)-\epsilon' -r|}} \frac{g(a)}{2\pi i (f'(a)-r)}e(f(a)-ra) \label{eq:fourier2}\\
&\qquad \qquad +O\left( \frac{U(a)f''(a)C^4L}{M(a)}+ \frac{U(a)L}{f''(a)c}+ \frac{U(a)f''(a)}{L^2}  +  \frac{U(a)}{f''(a)C^2}+\frac{U(a)}{M(a)}\right)\label{eq:newdeltaerror1}
\end{align}

Now we estimate the Type I integrals and the sum on line \eqref{eq:fourier2} in three separate cases.

\subsection{Case 1: $a$ is an integer.}

In this case, we begin by adding the $r$ and $-r$ terms from the Type I integrals together, obtaining 
\begin{align}
&-g(a)e(f(a)+(\epsilon'-f'(a))a) \sum_{1\le |r|\le L} \int_{C\cdot \operatorname{sgn}(r)}^0e\left( \frac{f''(a)}{2}x^2+rx+ra+\epsilon' x\right) \ dx \notag \\
&\qquad =- g(a) e(f(a)) \sum_{1\le r \le L}\int_C^0 e\left( \frac{f''(a)}{2}x^2+rx  \right)\cdot 2 i \sin(2\pi \epsilon' x) \ dx. \label{eq:case1}
\end{align}

We estimate each integral in \eqref{eq:case1} by an application of integration by parts.
\begin{align}
&\int_C^0 e\left( \frac{f''(a)}{2}x^2+rx  \right)\cdot 2 i \sin(2\pi \epsilon x) \ dx\notag \\
&\qquad = \left. \frac{\sin(2\pi \epsilon' x)}{\pi(f''(a)x+r)} e\left( \frac{f''(a)}{2}x^2+rx  \right)  \right]_C^0 \label{eq:case1 1} \\
&\qquad \qquad + \int_0^C   \frac{2\epsilon' \cos(2\pi \epsilon' x)}{f''(a)x+r}  e\left( \frac{f''(a)}{2}x^2+rx  \right) \ dx \label{eq:case1 2}\\
&\qquad \qquad - \int_0^C  \frac{f''(a) \sin(2\pi \epsilon' x)}{\pi(f''(a)x+r)^2}  e\left( \frac{f''(a)}{2}x^2+rx  \right) \ dx.\label{eq:case1 3}
\end{align}

The term
\[
\frac{\sin(2\pi \epsilon' x)}{\pi(f''(a)x+r)}
\]
on line \eqref{eq:case1 1} is $0$ when $x=0$ and bounded by $O(\epsilon'/f''(a))$ if $x=C$.  

We will then apply the first and second derivative tests (Lemmas \ref{lem:1stderivtest} and \ref{lem:2ndderivtest}) to the integral on line \eqref{eq:case1 2}.  The total variation may be bounded by
\begin{align*}
\int_0^C \left| \left( \frac{2\epsilon' \cos(2 \pi \epsilon' x)}{f''(a)x+r}\right)'\right| \ dx &=O\left( \int_0^C \left|\frac{\epsilon'^2 \sin(2\pi \epsilon' x)}{f''(a)x+r} \right|\ dx \right)\\
&\qquad +O\left( \int_0^C \left| \frac{\epsilon' f''(a)\cos(2\pi \epsilon' x)}{(f''(a)x+r)^2} \right| \ dx \right)\\
&=O\left( \frac{\epsilon'^2}{r} \int_0^C|\sin(2\pi \epsilon' x)| \ dx\right)\\
&\qquad + O\left( |\epsilon'| \int_0^C \frac{f''(a)}{(f''(a)x+r)^2} \ dx\right)\\
&=O\left( \frac{|\epsilon'|(1+|\epsilon'|C)}{r}\right) + O\left( \frac{|\epsilon'|}{r}\right).
\end{align*}
Since $|\cos(2\pi \epsilon x)|\le 1$ and $f''(a)x+r$ is at least $r$ for $x\in[0,C]$, the maximum modulus on the interval is at most $|\epsilon'|/r$.  Therefore the integral on line \eqref{eq:case1 2} is bounded by
\[
O\left(  \frac{|\epsilon'|(1+|\epsilon'| C)}{r} \min\left\{\frac{1}{\sqrt{f''(a)}}, \frac{1}{r}  \right\} \right).
\]

We will also apply the first and second derivative tests to the integral on line \eqref{eq:case1 3}.  Using $\sin(2\pi \epsilon' x) = 2\pi \epsilon' x +O(\epsilon^3 x^3)$, we have that 
\begin{align*}
&\int_0^C   \frac{f''(a) \sin(2\pi \epsilon' x)}{\pi(f''(a)x+r)^2}e\left( \frac{f''(a)}{2}x^2+rx  \right) \ dx\\
&\qquad =\int_0^C \frac{2\epsilon ' f''(a) x}{(f''(a)x+r)^2} e\left( \frac{f''(a)}{2}x^2+rx  \right) \ dx + O\left( \int_0^C \frac{f''(a)\epsilon'^3 x^3}{r^2} \ dx\right)\\
&\qquad =O\left( \frac{|\epsilon'|}{r}  \min\left\{\frac{1}{\sqrt{f''(a)}}, \frac{1}{r}  \right\} \right) +O\left( \frac{f''(a)|\epsilon'|^3 C^4}{r^2} \right).
\end{align*}
Here we used that the function $ax/(ax+b)^2$ is $0$ at $x=0$, increases monotonically to a maximum at $x=b/a$, and then decreases mononotonically to $0$ as $x$ tends to infinity.

Thus the Type I integrals \eqref{eq:case1} are bounded by
\begin{align}
&  g(a)\sum_{1\le r \le L} O\left(\frac{|\epsilon'|}{f''(a)}\right)\notag \\
&\qquad+ g(a)\sum_{1\le r \le L}O\left( \frac{|\epsilon'|(1+|\epsilon'| C)}{r}  \min\left\{\frac{1}{\sqrt{f''(a)}}, \frac{1}{r}  \right\} \right) \notag\\
&\qquad  +g(a)\sum_{1\le r \le L} O\left( \frac{f''(a)|\epsilon'|^3 C^4}{r^2} \right)\notag \\
& = O\left( \frac{U(a)|\epsilon'| L}{f''(a)}  \right)+O\left( \frac{U(a)|\epsilon'|(1+|\epsilon'|C)\log(1+f''(a))}{\sqrt{f''(a)}}  \right) \notag \\
&\qquad+ O\left( U(a)f''(a)|\epsilon'|^3C^4  \right)\notag,
\end{align}
which gives the bound on $\mathcal{D}_0(a)$ in this case.

Also, we bound the sum in line \eqref{eq:fourier2} by
\[
O\left( \frac{U(a)}{L}   \right) =O\left( \frac{U(a)f''(a)}{L^2} \right),
\]
using Proposition \ref{prop:fouriertail}.

\subsection{Case 2: $f'(a)$ is an integer, $a$ is close to an integer.}

As in case 1, we sum $r$ and $-r$ terms and use the first derivative test, obtaining
\begin{align*}
&g(a)e(f(a)+(\epsilon'-f'(a))a) \times\\
&\qquad \qquad \times\sum_{1\le |r|\le L} \int_{C\cdot \operatorname{sgn}(r)}^0e\left( \frac{f''(a)}{2}x^2+rx+ra+\epsilon' x\right) \ dx \\
&\qquad = g(a) e(f(a)-f'(a)a) \sum_{1\le r \le L} 2 i \sin(2\pi \epsilon r)\int_C^0 e\left( \frac{f''(a)}{2}x^2+rx  \right) \ dx\\
&\qquad =O\left( U(a) \sum_{1\le r \le L} |\epsilon| r \cdot \frac{1}{r} \right) \\
&\qquad =O\left( U(a) |\epsilon| L\right).
\end{align*}

By the same argument, we may complete the Fourier series in \eqref{eq:fourier2}.
\begin{align*}
 &-\lim_{R\to\infty} \sum_{\substack{|r|\le R \\ L\le |f'(a)-r|}} \frac{g(a)}{2\pi i (f'(a)-r)}e(f(a)-ra)\\
&\qquad = -\frac{g(a)e(f(a)-f'(a)a)}{2\pi i}\lim_{R\to \infty} \sum_{|r|\le R} \frac{e(\epsilon r)}{r}\\
&\qquad \qquad +O\left(U(a) \sum_{ r<L} \frac{sin(2 \pi \epsilon r)}{r}\right) \\
&\qquad = \psi(a)g(a)e(f(a)-f'(a)a) + O\left(U(a)|\epsilon| L  \right).
\end{align*}

Together these give $\mathcal{D}_0(a)$ in this case.

\subsection{Case 3: $f'(a)$ is an integer, $a$ is far from an integer.}

In particular, by saying $a$ is far from an integer, we want $|\epsilon|>C$.  This implicitly requires that $C< 1/2$.

In this case, we sum over all positive values of $r$ seperately from all negative values of $r$.  Ignoring the constant multiplier 
\[
 g(a)e(f(a)+(\epsilon'-f'(a))a)
\]
for the moment, the sum over positive $r$ values give
\begin{align}
& \sum_{1\le r\le L} \int_C^0e\left( \frac{f''(a)}{2}x^2+rx+ra+\epsilon' x\right) \ dx \notag \\
&\qquad= \int_C^0 e\left( \frac{f''(a)}{2}x^2 \right) \sum_{1\le r \le L} e(r(x+a)) \ dx\notag \\
&\qquad=    \int_C^0 e\left( \frac{f''(a)}{2}x^2+\frac{1}{2}x\right)  \frac{i}{2\sin(\pi (x+a))}\ dx \label{eq:case3error1}\\
&\qquad \qquad  -  \int_C^0 e\left( \frac{f''(a)}{2}x^2+\left(L+\frac{1}{2}\right)x\right)  \frac{i}{2\sin(\pi (x+a))}\ dx. \label{eq:case3error2}
\end{align}
Here we will apply the second derivative test to the integrals in lines \eqref{eq:case3error1}--\eqref{eq:case3error2}.

The variation plus maximum modulus in both integrals is bounded by
\[
 O\left(\frac{1}{(|\epsilon|-C)}\right),
\]
so that both integrals (as well as the corresponding integrals for the sum over negative $r$) are bounded by
\[
 O\left(\frac{1}{(|\epsilon|-C) \sqrt{f''(a)}} \right)
\]

In this case, again, we bound the terms in line \eqref{eq:fourier2} by $O(g(a)/L|\epsilon|)$.  But since $\sqrt{f''(a)}\ll L$, this term is dominated by
\[
 O\left( \frac{1}{(|\epsilon|-C)\sqrt{f''(a)}} \right).
\]

Thus we have $\mathcal{D}_0(a)$ in this final case.

\subsection{Combined case analysis}

When $\epsilon'=0$ and $M(a) \le f''(a)^7$, the following choices of $L$ and $C$ optimize the error terms from line \eqref{eq:newdeltaerror1}:
\[
L=f''(a)^{8/15}M(a)^{1/15} \qquad \text{and} \qquad C=f''(a)^{-2/5}M(a)^{1/5},
\]
when $|\epsilon|\le f''(a)^{-3/5}M(a)^{-1/5}$ or $f''(a)^{-2/5} M(a)^{1/5} \le |\epsilon|$;
\[
L=f''(a)^{1/3}|\epsilon|^{-1/3} \qquad \text{and}\qquad C=f''(a)^{-1}|\epsilon|^{-1},
\]
when $ f''(a)^{-3/5}M(a)^{-1/5}\le |\epsilon| \le f''(a)^{-1/2} $; and,
\[
L=f''(a)^{2/3}|\epsilon|^{1/3} \qquad \text{and} \qquad C=\epsilon/2,
\]
when $ f''(a)^{-1/2} \le |\epsilon| \le f''(a)^{-2/5} M(a)^{1/5}$.

Inserting these values into the error terms for cases 2 and 3 gives the result at the end of Theorem \ref{thm:refinement}.

These methods work similarly at $b$ in place of $a$.

\section{Proof of Corollary \ref{thm:exampleerror}}

Recall that we are attempting to bound the error in the following van der Corput transform.
\begin{equation}\label{eq:exampleerror2}
 \sideset{}{^*}\sum_{n \le N} e\left( \left(\frac{n}{3}\right)^{3/2}\right) = \sideset{}{^*}\sum_{(1/12)^{1/2} \le r \le (N/12)^{1/2}} \sqrt{24 r}\cdot e(-4r^3+1/8) + \Delta
\end{equation}

We can apply the main theorems of this paper to either the sum on the left or the right of \eqref{eq:exampleerror2}.  However, we obtain best results if we use both at different times, dependent on how big $\| (N/12)^{1/2} \|$ is.

\subsection{Case 1: $\| (N/12)^{1/2} \| > (12N)^{-1/4}$ or $\|(N/12)^{1/2}\|=0$}

In this case, we apply Theorem \ref{thm:toinfinity} to the left-hand side of \eqref{eq:exampleerror2}.  Thus
\[
 a=1,\qquad b=N, \qquad f(x) = \left(\frac{x}{3} \right)^{3/2}, \qquad g(x)=1,
\]
\[
 G(x)=0, \qquad H(x) = f^{(3)}(x) = -\frac{1}{8\sqrt{3}}x^{-3/2}, \qquad W_0(x) = \frac{2}{9x^2}, \qquad r_0(x)=2\sqrt{\frac{x}{3}}
\]
and $r_0'(x)$ is non-zero.  For $C_2$, $C_{2^-}$, $C_4$, $D_0$, $D_1$, and $D_2$ equal to $2$ and $\delta$ equal to $1/2$, we have that condition $(M)$ is satisfied for $M(x)=\epsilon \cdot x$ for some small fixed $\epsilon >0$ (independent of $N$) and $U(x)=1$.  Hence $H^2-G >0$ on $J=[a(1-\epsilon),b(1+\epsilon)]$, so $J_0$ is the full interval $J$, and both $J_\pm$ and $J_{null}$ are empty.

We first need to check that the conditions of Theorem \ref{thm:toinfinity} hold.  

Using our assumption that $\| (N/12)^{1/2} \| > (12N)^{-1/4}$, we can estimate $\Delta_2(b)$ by
\begin{align*}
 \Delta_2(b) &\ll N^{-2}(1+N^{3/4})(1+N^{-1/2}) + N^{-1/2} \\
&\qquad +\begin{cases}
\dfrac{1}{N}+\dfrac{1}{N^{1/2}} & \|(N/12)^{1/2}\|=0\\
 \dfrac{1}{N\|(N/12)^{1/2}\|^2}+ \dfrac{1}{N^{1/2} \| (N/12)^{1/2}\|^{3} } & \| (N/12)^{1/2} \| > (12N)^{-1/4}
\end{cases}\\
&\ll N^{-1/2} \| (N/12)^{1/2}\|^{*-3}.
\end{align*}
Thus $\Delta_2(b)$ tends to $0$ as $N=12k^2$ and $k$ tends to infinity along the positive integers.

We can bound $\Delta'_3(b)$ by 
\begin{align*}
\Delta'_3(b) &\ll \frac{1}{N^2}+\frac{1}{N^2}(1+N^{-1/2})\\
&\ll N^{-3/2},
\end{align*}
which clearly tends to $0$ as $N$ tends to infinity.

The set $K_b$ equals $[N(1+\epsilon)^{-1},N]$.  Let $\overline{N}$ be the largest integer in $[1,N-1/2]$ for which $(\overline{N}/12)^{1/2}$ is also an integer.  Then the first integral and sum of $\Delta'_4(b)$ are bounded by
\begin{align*}
 &\int_{[N(1+\epsilon)^{-1},\overline{N}]}  \frac{x^{1/2}}{(N-x)^3}\left( 1+ x^{-1/2} + \frac{x^{1/2}}{N-x}\right) \ dx\\
&\qquad + \frac{\overline{N}}{(N-\overline{N})^3} + \frac{N(1+\epsilon)^{-1}}{(N-N(1+\epsilon)^{-1})^3}\\
&\ll \frac{\overline{N}^{1/2}}{(N-\overline{N})^2} + \frac{\overline{N}}{(N-\overline{N})^3}.
\end{align*}
However, by the mean value theorem, we have, for some $\eta\in[\overline{N},N]$,
\begin{align*}
N-\overline{N} &= \left( \left( \frac{N}{12}\right)^{1/2} - \left( \frac{\overline{N}}{12} \right)^{1/2} \right) \cdot 24 \left( \frac{\eta}{12} \right)^{1/2}\\
&\gg\| (N/12)^{1/2} \|^* \cdot \overline{N}^{1/2},
\end{align*}
and $\overline{N}\asymp N$, so that the first integral and sum of $\Delta'_4(b)$ is bounded by
\[
\ll \frac{1}{N^{1/2} \|(N/12)^{1/2}\|^{*2}} + \frac{1}{N^{1/2} \|(N/12)^{1/2} \|^{*3}} \ll \frac{1}{N^{1/2} \|(N/12) ^{1/2}\|^{*3}}.
\]
Likewise the final integral and sum of $\Delta'_4(b)$ are bounded by
\begin{align*}
 &\ll\int_{N(1+\epsilon)^{-1}}^N x^{-5/2}(1+x^{3/4})\left(1+ \frac{1}{x^{1/2}}\right) \ dx\\
&\qquad + N^{-2}(1+N^{3/4})\\
&\ll N^{-3/4}.
\end{align*}
Thus $\Delta'_4(b) $ is bounded by
\[
 \ll N^{-1/2} \|(N/12) ^{1/2}\|^{*-3}
\]
and converges to $0$ on the sequence $N=12\cdot k^2$, $k=1,2,\dots$.

Finally, we need to show that the integrals and sums in $\Delta_3(a)$, $\Delta_4$, and $\mathcal{K}$ converge.  The integrand of the integral in $\Delta_3(a)$ is 
\[
 \frac{x}{(x-1)^3}\left( 1+ \frac{1}{x^{1/2}}+ \frac{x^{1/2}}{x-1}\right),
\]
which converges like $x^{-2}$.  The integrand of the integral in $\Delta_4$ is 
\[
 x^{-5/2}(1+x^{3/4})\left(1+ \frac{1}{x^{1/2}}\right),
\]
which converges like $x^{-7/4}$.  The integrand of the integral in $\mathcal{K}(J_0,W_0,r_0)$ is $x^{-5/2}$.  The sum in the $\mathcal{K}$ term is bounded by $|W_0(1-\epsilon)|+|W_0(N(1+\epsilon))|$, since $J_0$ is a single interval and $r'_0$ never changes sign.  Since $W_0(N(1+\epsilon))$ converges like $N^{-2}$, this shows that the conditions of Theorem \ref{thm:toinfinity} are satisfied.  It also shows that $\Delta_5$ is bounded by $O(N^{-3/2})$.

To apply the theorem we calculate $\mathcal{D}(b)$ explicitly by
\[
 \mathcal{D}(b) = \begin{cases}
O(N^{-1/2}) & \|(N/12)^{1/2}\|=0\\
\begin{aligned}&e\left(\left(N/3\right)^{3/2}\right)\times \\ &\quad \times\left(-\dfrac{1}{2\pi i\langle (N/12)^{1/2}\rangle} + \psi(N,\langle (N/12)^{1/2}\rangle ) \right) 
\end{aligned} & \| (N/12)^{1/2} \| > (12N)^{-1/4}
\end{cases}.
\]

If $\|(N/12)^{1/2}\|=0$, then $\Delta_1(b) = N^{-2}$; and otherwise, we get that $\Delta_1(b)=0$, since $(12N)^{-1/4}$ is greater than $f''(N)=(N/3)^{-1/2}/12$ for $N\ge 1$, and so $m_N$ equals $0$.

Combining the error terms together we see that
\[
\Delta=O(N^{-1/2} \|(N/12)^{1/2}\|^{*-3}),
\]
which completes the proof of the corollary in this case.

\subsection{Case 2: $0\neq \| (N/12)^{1/2} \|\le  (12N)^{-1/4}$}

Now we apply Theorem \ref{thm:refinement} to $e(-1/8)$ times the conjugate of the right-hand side of \eqref{eq:exampleerror2}.  We use the conjugate to guarantee that $f''(x)$ is positive.

We have
\[
a=\sqrt{1/12}, \qquad b=\sqrt{N/12}, \qquad f(x) = 4x^3, \qquad g(x) = \sqrt{24 x},
\]
\[
G(x) =-41472x, \qquad H(x) = 120\sqrt{6x}, \qquad H(x)^2-G(x) = 127872 x.
\]
Since $G$ is never $0$ on $[a,b]$ and $H^2-G$ is always positive on $[a,b]$ we have that $J_0$ is empty and $J_\pm=[a,b]$.
So, to complete our useful definitions, we have
\[
r_\pm(x) = 12(11\pm 2\sqrt{37})x^2, \qquad W_\pm(x) = \frac{\sqrt{37}\pm 7 }{96\sqrt{6}(\sqrt{37}\pm 5)^3}x^{-9/2}
\]
and again $r_\pm'(x)$ never equals $0$ on $[a,b]$.

Again, for $C_2$, $C_{2^-}$, $C_4$, $D_0$, $D_1$, and $D_2$ equal to $2$ and $\delta$ equal to $1/2$, we have that condition $(M)$ is satisfied for $M(x)=\epsilon \cdot x$ for some small fixed $\epsilon >0$ independent of $N$ and $U(x)=g(x)=\sqrt{24 x}$.

Since $a$ is fixed, we have
\[
\mathcal{D}(a)+\Delta_1(a)+\Delta_2(a) =O(1).
\]

By Theorem \ref{thm:refinement}, we have that $\mathcal{D}(b)$ equals 
\begin{align*}
&-\psi(b)g(b)e(f(b)-f'(b)b)+ O\left(\frac{U(b)}{M(b)^{2/15}f''(b)^{1/15}}+\frac{U(b)}{M(b)}\right)\\
&\qquad \qquad + O\left(U(b)f''(b)^{1/3}\| b\|^{2/3} \right)\\
&\qquad = -2\psi(\sqrt{N/12})(3N)^{1/4}    e(-(N/3)^{3/2})+O(N^{3/20})\\
&\qquad \qquad +O\left(N^{5/12}\| (N/12)^{1/2}\|^{2/3}\right).
\end{align*}

By our assumption that $N$ is an integer, $\|f'(b)\|=0$ and so
\[
\Delta_1(b) \ll b^{-7/2} \ll 1.
\]

For $\Delta_2(b)$ note that in this case $f'(b)$ is always an integer, and hence $\| f'(b) \|$ equals $0$.  Therefore, we have
\begin{align*}
\Delta_2(b) &\ll b^{-9/2}(1+b^{3/2})(1+b)+b^{-3/2} +b^{-1/2}+ b^{-1/2} \\
&\ll N^{-1/4}
\end{align*}

For $\Delta_3(a)+\Delta_3(b)$ we may assume $\overline{a}=a+1/2$ and $\overline{b}=b-1/2$ (at worst, this assumption only makes the $\Delta_3$ terms larger).  Thus,
\begin{align*}
\Delta_3(a) &=O\left( \int_{a+1/2}^b \frac{1}{x^{1/2}(x-1/2)^3}\left(1+x^{-2}+\frac{1}{x(x-1/2)}\right) \ dx\right)\\
&\qquad + O(1)+O(b^{-9/2})\\
&=O(1)
\end{align*}
and
\begin{align*}
\Delta_3(b) &=O\left( \int_a^{b-1/2} \frac{1}{x^{1/2}(b-x)^3}\left(1+x^{-2}+\frac{1}{x(b-x)}\right) \ dx\right)\\
&\qquad + O(b^{-3/2}) + O(b^{-3})\\
&= O\left(  \int_0^{b-1/2} \frac{1}{(b-x)^3}+\frac{1}{(b-x)^4} \ dx \right) +O(b^{-3/2})\\
&= O(1).
\end{align*}

To estimate $\Delta_4$, note that by our above definitions we have
\[
|W_-(x)||r_-'(x)|+|W_-'(x)|+|W_+(x)||r_+'(x)|+|W_+'(x)|=O(x^{-7/2}),
\]
and so,
\begin{align*}
\Delta_4 &= O\left( \int_a^b x^{-7/2}(1+x^{3/2})(1+x^{-2}) \ dx \right) +O\left( \int_{J_\pm} x^{-7/2} \ dx\right)\\
&\qquad + |W_-(a)|+|W_-(b)|+|W_+(a)|+|W_+(b)|\\
&=O(1)+O(1)+O(1)+O(b^{-9/2})+O(1)+O(b^{-9/2})\\
&=O(1).
\end{align*}

Thus we have shown that, in this case
\begin{align*}
\Delta&=2\psi(\sqrt{N/12})(3N)^{1/4}    e\left(\left(\frac{N}{3}\right)^{3/2}+\frac{1}{8}\right)+O(N^{3/20})\\
&\qquad  +O\left(N^{5/12}\| (N/12)^{1/2}\|^{2/3}\right),
\end{align*}
by undoing our conjugation, multiplying by $e(1/8)$, and moving this term to the correct side of the equation.

\section{Improving the results when $\|f'(\mu)\|$ is small}\label{section:fresnel}

In all the main results and the various examples given in this paper, the results tend to be at their best when $\|f'(\mu)\|$ (for $\mu$ equal to $a$ or $b$) is large compared with $\sqrt{f''(\mu)}$.  However, when $\|f'(\mu)\|$ is on the order of $\sqrt{f''(\mu)}$, the estimates tend to be dominated by an inexplicit term, and even when $\|f'(\mu)\|$ is much smaller than $\sqrt{f''(\mu)}$, we still do not have great asymptotics, such as in Corollary \ref{thm:exampleerror}.  

Could the error term $\Delta$ in this range be made more explicit?  If one's notion of explicit is a closed form that can be estimated moderately fast on a computer, then yes; if one's notion of explicit requires good asymptotic data, then the answer so far appears to be no.

Suppose that $f''(x)$ is generally smaller than $1$ on $[a,b]$.  The largest contributions to $\Delta$ come from bounding an integral of the form
\[
 \int_{x_r}^\mu g(x)e(f(x)-rx) \ dx
\]
where $|x_r-\mu|\ll \sqrt{f''(\mu)}$.  We will assume for a moment that $x_r$ is in the interval $[a,b]$.

In the proof of Theorem \ref{thm:main}, we bounded this integral with two techniques: first, possibly shifting $\mu$ to $\mu_r=\mu\pm M(\mu)$ if $|x_r-\mu|$ was smaller than $f''(\mu)$, and second, applying Propositions \ref{prop:main1} and \ref{prop:main2}, our variants of Redouaby and Sargos' results.

However, another result of Redouaby and Sargos provides an alternate bound for such an integral.  Under the conditions of Lemma \ref{lem:redsar1} and the additional assumption that $c-a \ll N,M$ to simplify the error terms, we have the following result (combining Lemma 3 and Lemma 9 in \cite{redsar}) for $a<c$:
\begin{align*}
 \int_a^c g(x) e(f(x)) \ dx &= \frac{g(c)}{\sqrt{f''(c)}}e(f(c))\mathcal{F}(\sqrt{f''(c)}\phi(a-c))+O\left(\frac{U(c-a)^2}{N}\right)\\
&\qquad +O\left(\frac{UM}{T}\left(1+\frac{M}{N}\right)^2\left(T^{-1/2}+(c-a)  \right)+\frac{U(c-a)}{T}\right),
\end{align*}
where $\alpha=f''(c)/2$,
\[
 \phi(t) = \operatorname{sgn}(t)\left(\frac{f(c+t)-f(c)}{\alpha}\right)^{1/2},
\]
and $\mathcal{F}(u)$ denotes the modified Fresnel integral function\footnote{More standard forms of the Fresnel integral functions would include $x^2/4$ or $x^2/2\pi$ where we have used $x^2/2$ and would further break the function up into its real and imaginary parts.},
\[
\mathcal{F}(u) = \int_0^u e(x^2/2) \ dx.
\]
An analogous result holds for the integral from $c$ to $b$.

Let us assume once again that $M=N$.  Just as in Propositions \ref{prop:main1} and \ref{prop:main2}, we see an error term of the form $O(UM/T^{3/2})$.  This new estimate on stationary phase integrals is useful---in the sense that the remaining error terms are of order of magnitude smaller than the explicit term---provided $M\gg 1$, $T$ is of order much larger than $1$, and $c-a$ is of order much smaller than 
\[
\min\left\{ \frac{M^{1/2}}{f''(c)^{1/4}},M\sqrt{f''(c)} \right\}.
\]
This in turn is useful to the van der Corput transform if the good range for $c-a$ includes $c-a=f''(c)^{-1/2}$, and this will be true provided $f''(c)$ is of an order of magnitude much greater than $M^{-1}$.  

To apply these new estimates, we need to make a few further small changes in the proof of Theorem \ref{thm:main}.  First, for this $r$, we do not change $\mu$ to $\mu_r$, removing the term $O(U(\mu)/\sqrt{f''(\mu)})$ from $\Delta_1(\mu)$.  Second, since we only apply the above estimates to an integral on one side of $x_r$, we obtain terms of size $U(\mu)/f''(\mu)M(\mu)=O(\Delta_2(\mu))$ from the terms now left uncancelled when we apply Propositions \ref{prop:main1} and \ref{prop:main2} to the other side.  Finally, since we do not have any integral from $\mu$ to $\mu_r$ in $S_3$, all the associated terms no longer need to be bounded.  If $\mu=a$, then this means we lack the term
\[
\frac{g(a_r)e(f(a_r)-ra_r)}{2\pi i(f'(a_r)-a)}
\]
as well as $O(h_r(a_r))$ and $O(h_r(\alpha_r))$.  This thus removes the term with $\langle f'(\mu)\rangle$ from $\mathcal{D}(\mu)$, and reduces the last pair of terms in $\Delta_2(\mu)$ to 
\[
\frac{U(\mu)}{M(\mu)}+ U(\mu)f''(\mu).
\]

When $x_r$ is not in the interval $[a,b]$, say $x_r>b$, then we need to be slightly more careful.  We will need to assume that condition $(M)$ holds on the larger interval $[a,x_r]$, and then consider the difference of integrals
\[
\left(  \int_{x_r-M(x_r)}^{x_r} - \int_b^{x_r} \right) g(x) e(f(x)) \ dx.
\]
We would apply Propositions \ref{prop:main1} and \ref{prop:main2} to the first integral and the above estimates to the second.

While this may seem like a great improvement over earlier theorems---and in a sense, it is---it also is very hard to apply.  The region $\|f'(\mu)\| \asymp \sqrt{f''(\mu)}$ requires us to understand the function $\mathcal{F}(u)$ in the region $u\asymp 1$.  Unfortunately, while there are very good asymptotics for $\mathcal{F}(u)$ when $u$ tends to $0$ or infinity, there are no good asymptotics when $u \asymp 1$.

\section{Further questions}

In Section \ref{sec:variation}, we do not make use of the fact that the intervals taken as arguments in the $K_r$ not only fail to include the point $x_r$ but also are excluded from an interval around $x_r$.  Is there a simple way to take this into account and does doing so improve the results in a substantial way? 	                                                                                                                                                                                                                                                                                                                                                                                                                                                                                                  

Can the results of this paper be explained in terms of the geometry of curves, in the same way as \cite{coutkaz2}?

Dekking and Mend{\`e}s-France conjecture in \cite{dmf} that $\Gamma(n\log n)$, the curve related to the exponential sum
\[
\sum_{1\le n \le x} e(n\log n),
\] is resolvable.  Can the results of this paper be used to prove this?

\section{Acknowledgements}

The author acknowledges support from National Science Foundation grant DMS 08-38434 ``EMSW21-MCTP: Research Experience for Graduate Students,'' under which the investigations into the van der Corput transform began.

\end{document}